\documentclass[reqno,english]{amsart}
\usepackage{amsfonts,amsmath,latexsym,verbatim,amscd,mathrsfs,color,array}

\usepackage{amsmath,amssymb,amsthm,amsfonts,graphicx,color}
\usepackage{amssymb}
\usepackage{pdfsync}
\usepackage{epstopdf}
\usepackage[colorlinks=true]{hyperref}
\usepackage{subfigure}

\makeatletter
\def\@currentlabel{2.1}\label{e:dispaa}
\def\@currentlabel{2.21}\label{e:dispau}
\def\@currentlabel{2.22}\label{e:dispav}
\def\@currentlabel{2.23}\label{e:dispaw}
\def\@currentlabel{2.24}\label{e:dispax}
\def\theequation{\thesection.\@arabic\c@equation}
\makeatother
\let\oldbibliography\thebibliography
\renewcommand{\thebibliography}[1]{%
\oldbibliography{#1}%
\setlength{\itemsep}{0pt}%
}

\renewcommand{\theequation}{\thesection.\arabic{equation}}
\newtheorem{lemma}{Lemma}[section]

\newtheorem{definition}{Definition}

\newtheorem{proposition}{Proposition}[section]
\newtheorem{corollary}{Corollary}[section]
\newtheorem{remark}{Remark}[section]
\newtheorem{conjecture}{Conjecture}[section]
\newcommand{\bremark}{\begin{remark} \em}
\newcommand{\eremark}{\end{remark} }

\newtheorem{numerical/experimental results}{Numerical/Experimental results}[section]

\newtheorem{theorem}{Theorem}[section]

\newcommand{\BE}{\begin{equation}}
\newcommand{\BEN}{\begin{equation*}}
\newcommand{\EE}{\end{equation}}
\newcommand{\EEN}{\end{equation*}}
\newcommand{\BL}{\begin{lemma}}
\newcommand{\EL}{\end{lemma}}
\newcommand{\BT}{\begin{theorem}}
\newcommand{\ET}{\end{theorem}}
\newcommand{\BP}{\begin{proposition}}
\newcommand{\EP}{\end{proposition}}
\newcommand{\BC}{\begin{corollary}}
\newcommand{\EC}{\end{corollary}}

\begin{document}


\title[Optimal Competing Lattices]{\bf On minima of  sum of theta functions and  Mueller-Ho Conjecture}

\author{Senping Luo}
\author{Juncheng Wei}

\address[S.~Luo]{Department of Mathematics, Jiangxi Normal University, Nanchang, 330022, China}
\address[S.~Luo]{Department of Mathematics,  University of Cincinnati, OH, 45221, USA}

\address[J.~Wei]{Department of Mathematics, University of British Columbia, Vancouver, B.C., Canada, V6T 1Z2}

\email[S.~Luo]{luosp1989@163.com or luosg@ucmail.uc.edu}

\email[J.~Wei]{jcwei@math.ubc.ca}

\begin{abstract}
Let  $z=x+iy \in \mathbb{H}:=\{z= x+ i y\in\mathbb{C}: y>0\}$  and $
\theta (s;z)=\sum_{(m,n)\in\mathbb{Z}^2 } e^{-s \frac{\pi }{y }|mz+n|^2}$ be the theta function associated with the lattice $\Lambda ={\mathbb Z}\oplus z{\mathbb Z}$. In this paper we consider the following pair of minimization problems
$$ \min_{  \mathbb{H} } \theta (2;\frac{z+1}{2})+\rho\theta (1;z),\;\;\rho\in[0,\infty),$$
$$ \min_{   \mathbb{H} } \theta (1; \frac{z+1}{2})+\rho\theta (2; z),\;\;\rho\in[0,\infty),$$
 where  the parameter $\rho\in[0,\infty)$  represents the competition of two intertwining  lattices.
 We find that as $\rho$ varies the optimal lattices admit a novel pattern: they move from rectangular (the ratio of long and short side changes from $\sqrt3$ to 1), square, rhombus (the angle changes from $\pi/2$ to $\pi/3$) to hexagonal; furthermore, there exists a closed interval of $\rho$ such that
 the optimal lattices is always square lattice. This is in sharp contrast to optimal lattice shapes  for single theta function ($\rho=\infty$ case), for which the hexagonal lattice prevails.  As a consequence, we give a partial answer to optimal lattice arrangements of vortices in competing systems of  Bose-Einstein condensates as conjectured (and numerically and experimentally verified) by Mueller-Ho \cite{Mue2002}.

\end{abstract}

\maketitle


\section{Introduction and Statement of Main Results}
\setcounter{equation}{0}

Let $ z\in \mathbb{H}:=\{z= x+ i y\in\mathbb{C}: y>0\}$ and $\Lambda ={\mathbb Z}\oplus z{\mathbb Z}$ be the lattice in $ \mathbb{R}^2$. The theta function associated with the lattice $ \Lambda$ is defined as
 \begin{equation}
 \label{thetas}
\theta (s; z)=\sum_{(m,n)\in\mathbb{Z}^2} e^{-s \frac{\pi }{y }|mz+n|^2}.
\end{equation}

In 1988, Montgomery \cite{Mon1988} proved the following celebrated result:
\begin{theorem} \label{Theorem4} For all $s> 0$ and $z\in \mathbb{H}$,
\begin{equation}\aligned
\theta (s; z) \geq \theta (s; z_0)
\endaligned\end{equation}
where $z_0=\frac{1}{2}+i \frac{\sqrt{3}}{2}$ (the triangular lattice, or called hexagonal lattice).  Equality holds if and only if $z=z_0$ (up to the group $\mathcal{G}_1$ (See \eqref{GroupG1}, Section 3)).
\end{theorem}


For the higher dimensional cases, the corresponding minimization problems on lattices was first investigated in Sarnak and Strombergsson \cite{Sar2006} and recently by Cohn-Kumar-Miller-Radchenko-Viazovska \cite{Coh2017, Coh2019}. For relations with sphere packing problems, see Viazovska \cite{Via2017} and Cohn-Kumar-Miller-Radchenko-Viazovska \cite{Coh2017} and the references therein. We mention that minimization problems for Dedekind eta function (equivalent to the theta function \eqref{thetas} via Melin transform) also arise in the extremal determinants of Laplace-Beltrami Operators. See  Osgood-Phillips-Sarnak \cite{Osg1988}, Faulhuber \cite{Fau2020} and the reference therein.

The celebrated Theorem \ref{Theorem4} has laid foundations in many optimal lattice problems in number theory and has been frequently  used in applied matthematical and physical models such as  crystallizations of particle interactions (Blanc-Lewin \cite{Bla2015}, B$\acute{e}$termin \cite{Bet2016,Bet2018}, B$\acute{e}$termin-Zhang \cite{Bet2015}),  Ginzburg-Landau theory in superconductors (Abrikosov \cite{Abr}, Sandier-Serfaty \cite{Serfaty2010,Serfaty2012}, Serfaty \cite{Serfaty2014}), Ohta-Kawasaki models in di-block copolymers (Chen-Oshita \cite{Che2007}, Goldman-Muratov-Serfaty \cite{GMS2013}, Ren-Wei \cite{RW2015}), minimal frame operator norms (Faulhuber \cite{Fau2018}) and many others. The related minimization of theta fucntions/eta functions on lattices has application to Gross-Pitaeskii theory in superfluids or  Bose-Einstein condensates (Aftalion-Blanc-Nier \cite{ABN}, Aftalion-Serfacty \cite{AS}), Ohta-Kawasaki models triblock copolymers (Luo-Ren-Wei \cite{Luo2019}) and many others.

\medskip

In this paper, we consider a minimization problem with sum of {\bf two}  theta functions, which represent {\bf two} intertwining  lattices, one lattice lying at the center of the other lattice. See Figure 1 and the physical explanation in the next section.
\begin{figure}
\centering
 \includegraphics[scale=0.08]{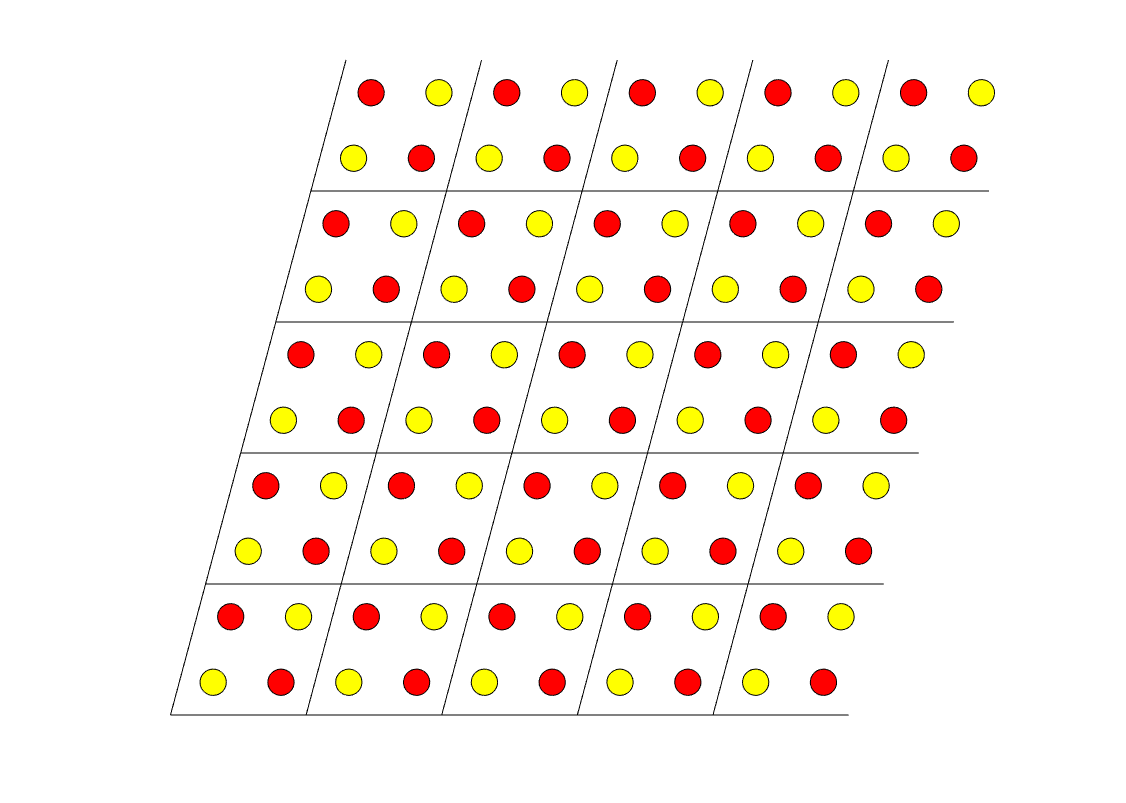}\includegraphics[scale=0.08]{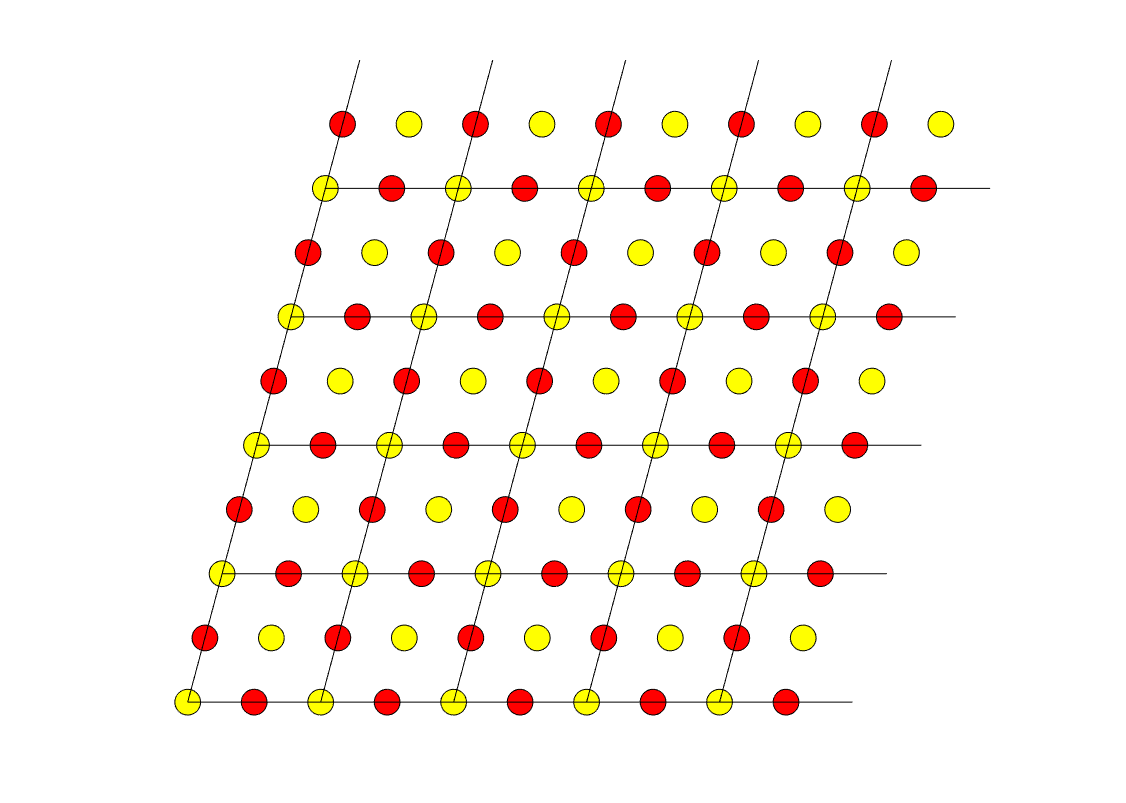}
 \caption{Two lattices with  centers at the lattice points and the half lattice
   points.}
\label{f-assemblies}
\end{figure}

Let $ \rho>0$ denote the relative strength of the two lattices. Consider the following functional
\begin{equation}
\label{W1rho}
\mathcal{W}_{1,\rho}(z):= \theta (2;\frac{z+1}{2}) + \rho \theta (1;z).
\end{equation}

It is easy to see that $\mathcal{W}_{1,\rho} (z)$ is invariant under the group (see Section 3)
\begin{equation}\aligned\label{Group2}
\mathcal{G}_2: \hbox{the group generated by} \;\;z\mapsto -\frac{1}{z},\;\; z\mapsto z+2,\;\;z\mapsto -\overline{z}.
\endaligned\end{equation}

The new minimization problem we consider is the following
\begin{equation}\aligned\label{Min1}
\min_{z\in\mathbb{H}} \mathcal{W}_{1, \rho} (z),\;\;\rho\in[0,\infty).
\endaligned\end{equation}

Our  main result is the following theorem which gives a complete characterization of the minimization problem \eqref{Min1}, as $\rho$ varies:

\begin{theorem}\label{Th1} The minimization problem \eqref{Min1} admits a unique minimizer $z_{1, \rho}$ which moves continuously on a special curve as the parameter $\rho$ varies (up to the group $\mathcal{G}_2$).
The trajectory curve of the minimizer, denoted by $\Omega_e$ (see Figure 2), is given by
\begin{equation}\aligned\label{Curve1}
&\Omega_e:=\Omega_{ea}\cup\Omega_{eb},\\
&\Omega_{ea}:=\{z: x=0,1\leq y\leq\sqrt3\},\\
&\Omega_{eb}:=\{z: |z|=1,0\leq x<\frac{1}{2}\}.
\endaligned\end{equation}
More precisely, there exist two thresholds $\sigma_{1,a}=0.04016\cdots<\sigma_{1,b}=0.83972\cdots$ such that
\begin{itemize}
  \item [(1)] if $\rho$ varies in $[0,\sigma_{1,a}]$, the minimizer $z_{1,\rho}$ moves from top to bottom along the vertical line segment $\Omega_{ea}$;
  \item [(2)] if $\rho\in[\sigma_{1,a},\sigma_{1,b}]$, the minimizer $z_{1,\rho}$ stays fixed on the corner of the curve $\Omega_e$, i.e.,
    \begin{equation}\aligned\nonumber
z_{1,\rho}\equiv i,\;\;\hbox{if}\;\;\rho\in[\sigma_{1,a},\sigma_{1,b}];
\endaligned\end{equation}
  \item [(3)] if $\rho$ varies in $[\sigma_{1,b},\infty)$, the minimizer $z_{1,\rho}$ moves from $i$   to $\frac{1}{2}+i\frac{\sqrt3}{2}$ along the unit arc, $\Omega_{eb}$. Moreover \begin{equation}\aligned\nonumber
  \mbox{as} \ \rho\rightarrow\infty, \ \;\;z_{1,\rho}\rightarrow\frac{1}{2}+i\frac{\sqrt3}{2}\;\;\hbox{from left hand side of}\;\;\Omega_{eb}.
  \endaligned\end{equation}

\end{itemize}

\end{theorem}

\begin{figure}
\centering
 \includegraphics[scale=0.28]{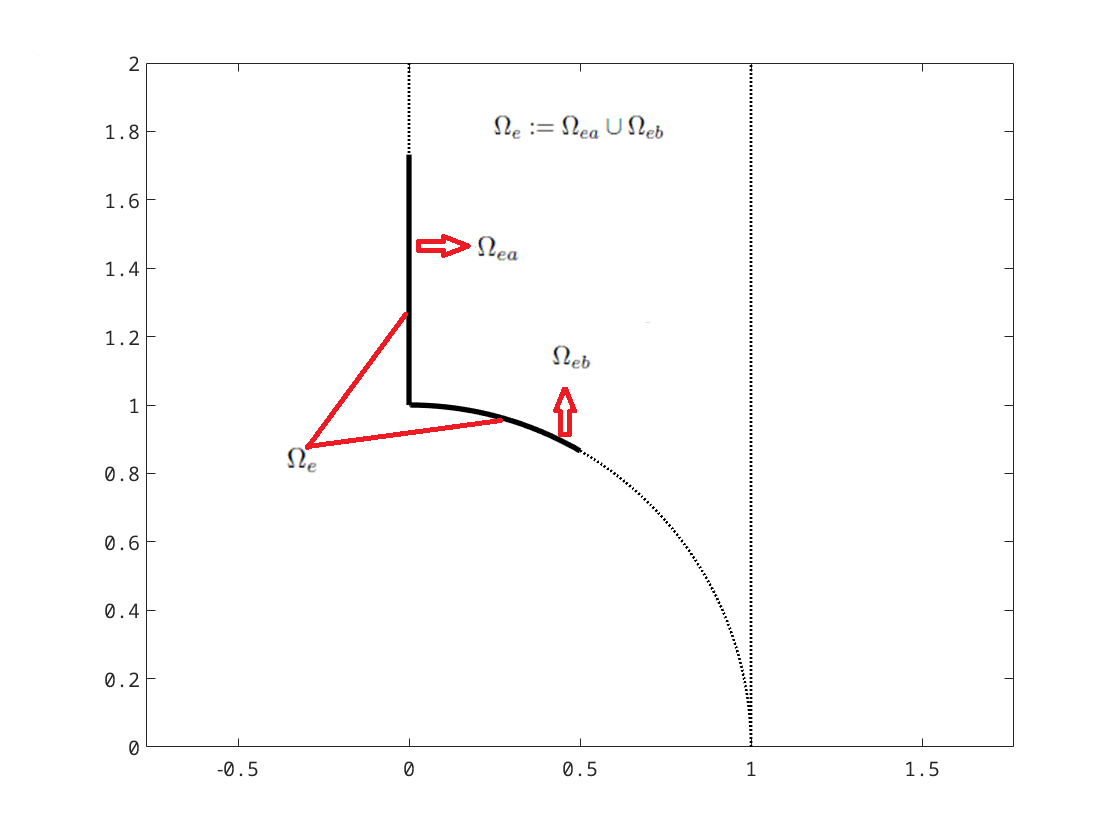}
 \caption{The curve $\Omega_e$}
\label{f-W}
\end{figure}

\begin{remark}  In \cite{Luo2019}, with X. Ren, we have studied another minimization problem
\begin{equation}\aligned\label{Minetat100}
\min_{z\in \mathbb{H}}-\Big((1-b)\big(\frac{1}{2}\log(\sqrt{y}|\eta(z)|^2\big)+b\big(\frac{1}{2}\log(\sqrt{y}|\eta(\frac{z+1}{2})|^2\big)\Big), z=x+iy, \;\;b\in[0,1],
\endaligned\end{equation}
 where $\eta$ is the Dedekind eta function
\begin{equation}\aligned
\eta(z)=e^{\frac{\pi}{3}\pi i}\prod_{n=1}^\infty (1-e^{2\pi nz i})^4.
\endaligned\end{equation}
When $b=0$, this is the minimization problem studied by Chen-Oshita \cite{Che2007} and  Sandier-Serfaty \cite{Serfaty2012}. While Chen and Oshita used analytical method to prove that the triangular lattice is the optimal, Sander and Serfaty  made use of a relation between the Dedekind eta function
and the Epstein zeta function (Melin transform), and  then Theorem \ref{Theorem4} to arrive at the same conclusion. When $ 0<b<1$, we have showed a similar transition phenomenon  from rectangle lattice to hexagonal lattice to Theorem \ref{Th1} in \cite{Luo2019} for the functional in \eqref{Minetat100}.
\end{remark}

We also consider another minimization problem, which can be viewed as a "conjugate" problem to \eqref{Min1}

\begin{equation}\aligned\label{Min2}
\min_{z\in\mathbb{H}} \mathcal{W}_{2, \rho} (z),\;\;\rho\in[0,\infty), \ \ \mbox{where} \ \mathcal{W}_{2,\rho}(z):=\theta (1; \frac{z+1}{2})+\rho \theta (2;z).
\endaligned\end{equation}

The precise relation between $\mathcal{W}_{1,\rho}$ and $ \mathcal{W}_{2, \rho}$ can be found in Lemma \ref{ThWW}. The minimizers of \eqref{Min2} can be characterized as follows:

\begin{theorem}\label{Th2} The minimization problem \eqref{Min2} admits a unique minimizer $z_{2, \rho}$ which lies on  the curve $\Omega_e$  \eqref{Curve1} (up to the group $\mathcal{G}_2$\eqref{Group2}). There exist two thresholds $\sigma_{2,a}=1.190861337\cdots,\;\;\sigma_{2,b}=24.89618074\cdots$  such that

\begin{itemize}
  \item [(1)] if $\rho$ varies from left to right on $[0,\sigma_{2,a}]$, the minimizer $z_{2,\rho}$ moves from top to bottom on the vertical line segment $\Omega_{ea}$;
  \item [(2)] if $\rho \in [\sigma_{2,a},\sigma_{2,b}]$, the minimizer $z_{2,\rho}$ stays fixed on the corner of curve \eqref{Curve1}, i.e. $\;z_{2,\rho}\equiv i$;

  \item [(3)] if $\rho$ moves from left to right on $[\sigma_{2,a},\infty)$, the minimizer $z_{2,\rho}$ moves from left to right along the unit curve $\Omega_{eb}$. Furthermore
      \begin{equation}\aligned\nonumber
\mbox{as}\  \rho\rightarrow\infty, \ z_{2,\rho}\rightarrow\frac{1}{2}+i\frac{\sqrt3}{2}\;\;\hbox{from left hand side of}\;\;\Omega_{eb}.
  \endaligned\end{equation}
\end{itemize}

\end{theorem}
\begin{remark} The values of $\sigma_{1,a}, \sigma_{1, b}, \sigma_{2,a} $ and $\sigma_{2, b}$ are given explicitly in terms of Jacobi Theta functions. See Theorem 1.4  below.
\end{remark}
\begin{remark} The minimizers of the minimization problems \ref{Min1} and \ref{Min2} admit a novel pattern: they bond together in a very special way and form a nice geometric shape and move with the parameter in a monotone way. The optimal lattices have richer structures than that of Theorem \ref{Theorem4}.
\end{remark}

\medskip

There are some hidden connections revealed later between the two minimization problems \eqref{Min1} and \eqref{Min2}.
They are like "a pair" as shown in Table 1 below. The following theorem gives more  qualitative behaviors of minimizers in Theorem \ref{Th1} and Theorem \ref{Th2}.

\begin{theorem}\label{Th3} Let $ z_{1,\rho}$ and $z_{2,\rho}$ be  the minimizers of \eqref{Min1} and \eqref{Min2} respectively.
\begin{itemize}
\item [(1)] Minimizers of \eqref{Min1} and \eqref{Min2} for each $\rho\in[0,\infty)$ are given in the following Table 1.
{\bf
\begin{table}[!htbp]\label{Table0}
\caption{Minimizers of $\mathcal{W}_{1,\rho}(z),\mathcal{W}_{2,\rho}(z)$ for parameter $\rho\in[0,\infty)$}
\label{demo1}
\centering
\begin{tabular}{|c|c|c|c|}
\hline

{$\mathcal{W}_{1,\rho}(z)$}& {$\mathcal{W}_{1,\rho}(z)$}&{$\mathcal{W}_{2,\rho}(z)$}& {$\mathcal{W}_{2,\rho}(z)$}
\\

\hline
Domain of $\rho$ & Minimizer &  Domain of $\rho$ & Minimizer\\

\hline
$\rho\in [\rho_{1},1/\rho_{2}]$ & $z_{1,\rho}\equiv i$ & $\rho\in [\rho_{2},1/\rho_{1}]$ & $z_{2, \rho}\equiv i$ \\

\hline
 $\rho\in(1/\rho_{2},\infty)$
  & $z_{1,\rho}=\frac{y_{2,1/\rho}^2-1}{y_{2,1/\rho}^2+1}+i\frac{2y_{2,1/\rho}}{y_{2,1/\rho}^2+1}$  & $\rho\in(0,\rho_{2})$ & $z_{2,\rho}=iy_{2,\rho}\in \Omega_{ea}$ \\
  \hline
 $\rho\in(0,\rho_{1})$ & $z_{1,\rho}=iy_{1,\rho}\in \Omega_{ea}$ &$\rho\in(1/\rho_{1},\infty)$ & $z_{2,\rho}=\frac{y_{1, 1/\rho}^2-1}{y_{1,1/\rho}^2+1}+i\frac{2y_{1,1/\rho}}{y_{1,1/\rho}^2+1}$  \\
  \hline
\end{tabular}
\end{table}
}
  \item [(2)]  The thresholds in Theorems \ref{Th1} and \ref{Th2} are given by

  \begin{equation}\aligned\nonumber
  \sigma_{1,a}=\frac{1}{\sigma_{2,b}}=\rho_1,\;\;\sigma_{1,b}=\frac{1}{\sigma_{2,a}}=\frac{1}{\rho_2},
  \endaligned\end{equation}
  where $\rho_{1}$ and $ \rho_{{2}}$  are determined explicitly by
\begin{equation}\aligned\nonumber
\rho_{1}=-\frac{\mathcal{Y}''(1)}{\mathcal{X}''(1)}, \;\;\rho_{2}=-1-\frac{\mathcal{B}''(1)}{\mathcal{A}''(1)}.
 \endaligned\end{equation}
Here
\begin{equation}\aligned\label{XYAB}
\mathcal{X}(y):&=\vartheta_3(y)\vartheta_3(\frac{1}{y}), \;\;\;\;\;\;\;\;\;\;\mathcal{Y}(y):=2\big(
\vartheta_3(4y)\vartheta_3(\frac{4}{y})+\vartheta_2(4y)\vartheta_2(\frac{4}{y})
\big)\\
 \mathcal{A}(y):&=\sqrt2\vartheta_3(2y)\vartheta_3(\frac{2}{y}),\;\;\mathcal{B}(y):=\sqrt 2\vartheta_2(2y)\vartheta_2(\frac{2}{y})
\endaligned\end{equation}
 and the Jacobi Theta functions are defined as
 \begin{equation}
 \label{XYAB2}\aligned
\vartheta_2(y)=\sum_{n\in \mathbb{Z}} e^{-\pi(n-\frac{1}{2})^2 y},\;\;\vartheta_3(y)=\sum_{n\in \mathbb{Z}} e^{-\pi n^2 y},\;\; \vartheta_4(y)=\sum_{n\in \mathbb{Z}} (-1)^n e^{-\pi n^2 y}.
 \endaligned\end{equation}

  \item [(3)] The $y_{1, 1/\rho}$ and $ y_{2,1/\rho}$ in the Table 1   are implicitly determined by
\begin{equation}\aligned\label{Ygamma}
&y_{1,1/\rho} \;\;\hbox{is the unique solution of}\;\;\frac{\mathcal{Y}'(y)}{\mathcal{X}'(y)}+1/\rho=0,\\
&y_{2,1/\rho} \;\;\hbox{is the unique solution of}\;\;1+\frac{\mathcal{B}'(y)}{\mathcal{A}'(y)}+1/\rho=0.
 \endaligned\end{equation}

Furthermore, there holds
\begin{equation}\aligned\nonumber
\frac{d}{d\rho}y_{1, \rho }<0,\;\;
\frac{d}{d\rho}y_{2, \rho } <0.
 \endaligned\end{equation}

\end{itemize}

\end{theorem}

The existence and uniqueness of $y_{1,1/\rho}, y_{2,1/\rho}$ in the Theorems \ref{Th1} and \ref{Th2} are consequences of  the following theorem whose proof will be  given by Theorem \ref{QXY} and \ref{QAB}. (Here $\mathcal{X}(y),\mathcal{Y}(y)$ and $\mathcal{A}(y),\mathcal{B}(y)$  are defined in \eqref{XYAB}.)
\begin{theorem}
\label{Th4}
\begin{itemize}

  \item The function $y\mapsto \frac{\mathcal{Y}'(y)}{\mathcal{X}'(y)}, y>0$ has only one critical point at $y=1$, and it holds that
  \begin{equation}\nonumber
\Big(\frac{\mathcal{Y}'(y)}{\mathcal{X}'(y)}\Big)'<0, \;\;y \in(0,1) \  \mbox{and}  \ \Big(\frac{\mathcal{Y}'(y)}{\mathcal{X}'(y)}\Big)'>0, \;\;y \in(1,\infty).
 \end{equation}

  \item The function $y\mapsto \frac{\mathcal{B}'(y)}{\mathcal{A}'(y)}, y>0$ has only one critical point at $y=1$, and it holds that
  \begin{equation}\nonumber
\Big(\frac{\mathcal{B}'(y)}{\mathcal{A}'(y)}\Big)'<0, \;\;y \in(0,1) \ \ \mbox{and}\ \Big(\frac{\mathcal{B}'(y)}{\mathcal{A}'(y)}\Big)'>0, \;\;y \in(1,\infty).
 \end{equation}

\end{itemize}

\end{theorem}

Theorem \ref{Th1} has  direct applications to the Mueller-Ho functional and Mueller-Ho Conjecture in vortices arrangements for competing systems of Bose-Einstein condensates, as we explain in the next section.

\section{Applications to Mueller-Ho conjecture}
\setcounter{equation}{0}

As we have mentioned in Section 1, the problem of finding optimal lattice shapes arise in many physical models. Besides those examples we mentioned in Section 1, another example is the so-called vortices in Bose-Einstein condensates.
 Vortices in Bose-Einstein condensates are also called topological
defects, correspond to a zero of the order parameter with a circulation
of the phase. When they get numerous, these vortices arrange themselves on a lattice. In fact, in rotating Bose Einstein condensates (BEC), vortices were first observed in two component BEC's  (Matthews etc \cite{Mat1999}): it is observed experimentally that the shape of the lattice can be either hexagonal or square depending on the rotational velocity of the condensate. Since then, following the pioneering work of Mueller-Ho \cite{Mue2002}, many authors have investigated the lattice shape in two component BEC's and for instance  Kasamatsu etc \cite{KTU1,KTU2};  related works include Keeli--Oktel \cite{KO} who numerically calculate the elastic coefficients of the lattice, Aftalion-Mason-Wei \cite{AMW} who study the system describing the vortex/spike and derive an interaction term. In Kuokanportti etc \cite{Kuop}, the authors investigate the case of different masses and attractive interactions.

The ground state of a two component condensate is well described by a Gross Pitaevskii energy depending on the wave functions of each component which are coupled by an interaction term.
The construction of the Bose-Einstein condensates with large number of vortices was deduced in Ho \cite{Ho2001} (one-component case) and Mueller-Ho \cite{Mue2002} (two-component case), with the potential energy given by
$$ \mathcal{V}= \frac{1}{2} g_1 |\Psi_1|^4+ \frac{1}{2} g_2 |\Psi_2|^4+ g_{12} |\Psi_1|^2 |\Psi_2|^2$$
where $ g_{12}$ represents the competing strength between the two components of Bose gas. We omit the details of the construction of the model here.
 In Mueller-Ho \cite{Mue2002} they  have reduced the minimization problems on lattices to the minimization problems for the  Mueller-Ho functional
 \medskip

\begin{equation}
\aligned
\min_{z\in\mathbb{H}, (a, b) } {\mathcal E}_{MH} (z; a, b), \alpha\in[-1,1], \ \mbox{where}\  {\mathcal E}_{MH} (z):=\theta (1; z)+\alpha\mathcal{J}(z;a,b).
\endaligned
\end{equation}

\medskip

Here $\Lambda ={\mathbb Z}\oplus z{\mathbb Z}$ denotes the lattice of one component Bose gas A, and the theta function $\theta (1;z)$ (defined at \eqref{thetas})
represents  the self-interaction part of single component  of A or B, i.e., the so-called Abrikosov energy. (See Abrikosov \cite{Abr}.) The  functional
\begin{equation}\aligned
\mathcal{J}(z;a,b)=\sum_{(m,n)\in\mathbb{Z}^2}e^{-\frac{\pi}{y}|mz-n|^2}\cos(2\pi(ma+nb)).
\endaligned
\end{equation}
characterizes  the competing strength of two components $A$ and $B$. $\alpha=\frac{g_{12}}{\sqrt{g_1 g_2}}$ represents  the strength of competition between two competing components $A$ and $B$. The vector $(a,b)$ characterizes  the relative position of the these lattice shape. See Figure 1 when $ (a, b)=(\frac{1}{2},\frac{1}{2})$.

 It is interesting to compare the two-component case with the single-component case. In the latter system, energy minimization reduces to minimizing $\theta (1; z)$ whose only local minimum is the triangular lattice, where $ z=z_0= e^{i\frac{\pi}{3}}$ and $\theta (1; z_0) = 1.1596$ (by Theorem \ref{Theorem4}); the square lattice $z=i$ is a saddle point with $\theta (1;i)= 1.1803$. For two-component case, the minimum of $\mathcal{E}_{MH} (z;a,b)$ depends on the relative strength $\alpha$ and the relative position of the lattices, as conjectured by Mueller-Ho \cite{Mue2002} (supported by numerical computations and experimental results):

\medskip

\noindent
{\bf Mueller-Ho Conjecture:} For a two-component Bose gas, the most favorable lattice minimizing $\theta (1;z)+\alpha\mathcal{J}(z;a,b)$ are
\begin{itemize}
  \item [(a)] $\alpha<0$: the vortices of the two components coincide with each other $(a=b=0)$ to form  a triangular lattice $(z=e^{i\frac{\pi}{3}})$.
  \item [(b)] $0<\alpha<0.172$: the vortex lattice in each component remains triangular. However one lattice is displaced to the center of the triangle of the other $a=b=\frac{1}{3}$. The lattice type $($characterized by $z=z_0=e^{i\frac{\pi}{3}}$$)$ remains constant within this interval.
  \item [(c)] $0.172<\alpha<0.373$: $(a,b)$ jumps from the center of the triangle (i.e., half of the unit cell) to the center of the rhombic unit cell $a=b=\frac{1}{2}$. The angle jumps from $60^o$ to $67.95^o$ at $\alpha=0.172$, and increases continuously to $90^o$ as $\alpha$ increases to $0.372$. As a result, the lattice shape type is no longer fixed and the unit cell is rhombus. The modulus $\frac{b}{a}$, however, remains fixed across this region.
  \item [(d)] $0.373<\alpha<0.926$: the two lattices are "mode locked" into a centered square structure throughout the entire interval $(z=i,a=b=\frac{1}{2})$.
  \item [(e)] $0.926<\alpha<1$: the lattice type again varies continuously with interaction $\alpha$. Each component's vortex lattice has a rectangular unit cell $($angle$=\frac{\pi}{2})$ whose aspect ratio $|z|$ increases with $\alpha$. At $\alpha=1$, the aspect ratio is $\sqrt3$.

\end{itemize}
\begin{remark} Both $Rb^{87}$ and $Na^{23}$ have interaction parameters with the range $(d)$, i.e., $0.373<\alpha<0.926$.
\end{remark}

\medskip

For more on the vortex shape and Bose-Einstein condensates, including the construction of theoretical models and numerical and experimental results, we refer to \cite{Mat1999,Kas2005,Kas2003} and the references  therein.  In \cite{Kea2006} the authors considered Tkachenko modes and verified the same numerical results as in Mueller-Ho Conjecture. It seems that the Mueller-Ho conjecture is a universal phenomenon, as commented by
B$\acute{e}$termin \cite{Bet2019}  that {\em "the same phenomenon in Mueller-Ho results is also expected in other physical and biological models involving infinite lattices and competitive interactions"}. See also numerical computations in  B$\acute{e}$termin-Faulhuber-Kn$\ddot{u}$pfer \cite{Bet2020}.

\medskip

To study the minimizer of the Muller-Ho functional
$\mathcal{E}_{MH} (z;a,b)=\theta (1; z)+\alpha\mathcal{J}(z;a,b)$ with respect to $(z;a,b)$,
 we first need to identify the  critical points of $\mathcal{E}_{MH}$ which satisfy
 \begin{equation}\aligned\label{Equ1}
\nabla_z\theta (1;z)+\alpha \nabla_z\mathcal{J}(z;a,b)=0,
\endaligned\end{equation}
\begin{equation}\aligned\label{Equ2}
\nabla_{(a,b)}\mathcal{J}(z;a,b)=0.
\endaligned\end{equation}

To consider the  global minimum of $\theta (1;z)+\alpha\mathcal{J}(z;a,b)$, a necessary condition is that $(a,b)$ must be a minimum of $\mathcal{J}(z;a,b)$. Thus
we first focus on critical point equation \eqref{Equ2}.

For the function $\mathcal{J}(z;a,b)$ with respect to $(a,b)$, one sees clearly that
\begin{equation}\aligned
\mathcal{J}(z;a+1,b)=\mathcal{J}(z;a,b),\;\;\mathcal{J}(z;a,b+1)=\mathcal{J}(z;a,b)
\endaligned\end{equation}
\begin{equation}\aligned
\mathcal{J}(z;1-a,1-b)=\mathcal{J}(z;a,b).
\endaligned\end{equation}

The periodicity and symmetry imply that $\mathcal{J}(z;a,b)$ with respect to $(a,b)$ has four universal critical points, which are denoted by
\begin{equation}\aligned
w_0:=(0,0), w_1:=(\frac{1}{2},0), w_2:=(0,\frac{1}{2}), w_3:=w_1+w_2=(\frac{1}{2},\frac{1}{2}).
\endaligned\end{equation}
We call "universal" here since they are independent of the lattice structures i.e., $z$.
Clearly, the critical point $w_0$ is the global maxima of $\mathcal{J}(z;a,b)$ with respect to $(a,b)$. For critical points $w_1,w_2,w_3$, we have the following partial classification result (the proof will be given in Section 9):
\begin{lemma}\label{Point1} Let $z= iy, y>0$. There holds:
\begin{itemize}
  \item $w_1, w_2$ are the saddle points of $\mathcal{J}(z;a,b)$ with respect to $(a,b)$. Explicitly, the Hessian at each point can be expressed by
  \begin{equation}\aligned\nonumber
D^2J(z;a,b)\mid_{\{z=iy,(a,b)=w_1\}}&=16\pi^2\vartheta_3(\frac{1}{y})\vartheta_3'(\frac{1}{y})\vartheta_4(y)\vartheta_4'(y)<0\\
D^2J(z;a,b)\mid_{\{z=iy,(a,b)=w_2\}}&=16\pi^2\vartheta_3(y)\vartheta_3'(y)\vartheta_4(\frac{1}{y})\vartheta_4'(\frac{1}{y})<0.
\endaligned\end{equation}
\item $w_3$ is the local minimum of $\mathcal{J}(z;a,b)$ with respect to $(a,b)$. Explicitly, one has the Hessian expression
  \begin{equation}\aligned\nonumber
D^2J(z;a,b)\mid_{\{z=iy,(a,b)=w_3\}}&=16\pi^2\vartheta_4(y)\vartheta_4'(y)\vartheta_4(\frac{1}{y})\vartheta_4'(\frac{1}{y})>0.
\endaligned\end{equation}
\end{itemize}

\end{lemma}

For $(a, b)= (0, 0),
\mathcal{J}(z;0,0)=\theta (1; z)$. Combining Theorem \ref{Theorem4} and using the fact that $w_0$ is the global maxima of $\mathcal{J}(z;a,b)$, we have the following proposition which confirms the (a) part of Mueller-Ho Conjecture:
\begin{proposition}\label{Lstable}  For $\alpha\in[-1,0]$, the minimizer of the functional $ \mathcal{E}_{MH} (z;a,b)=
\theta (1; z)+\alpha \mathcal{J}(z;a,b) $
is achieved at $z_0=\frac{1}{2}+i\frac{\sqrt3}{2}$ and $(a,b)=(0,0)$.
\end{proposition}

Besides the above 4 universal critical points, there may be  other additional pair critical points. (Note that by symmetry if $(a,b)$ is a critical point then $(1-a, 1-b)$ is also a critical point.) We have

\begin{lemma}\label{Lemma1313} If $z=i$, then $(a,b)=(\frac{1}{3},\frac{1}{3})$ is not a critical point of $\mathcal{J}(z;a,b)$;while $(a,b)=(\frac{1}{3},\frac{1}{3})$  (and $(a,b)=\frac{2}{3},\frac{2}{3}$) is a critical point of $\mathcal{J}(z;a,b)$ if $z=\frac{1}{2}+i\frac{\sqrt3}{2}$.
\end{lemma}

The proof of Lemma \ref{Lemma1313} will be given in Appendix 1.


On the critical point equation \eqref{Equ2}, the numerical simulation suggests the following conjecture:
\begin{conjecture} The function  $\mathcal{J}(z;a,b)$ with respect to the $a,b$ has either 4 or 6 critical points depending on modulus of the tori $z$. Let $\Omega_4 ($resp. $\Omega_6)$ be the subset of $\mathbb{H}$ which corresponds to tori $z$ having four $($resp. six$)$ critical points.
There holds
\begin{itemize}
  \item [a]: Alterative:
  \begin{equation}\aligned\nonumber
\mathbb{H}=\Omega_4\cup\Omega_6,\;\;\Omega_4\cap\Omega_6=\emptyset.
\endaligned\end{equation}
  \item [b]: Rectangular tori has only four critical points and the hexagonal one has six.
  \begin{equation}\aligned\nonumber
i \in \{z\mid: \Re(z)=0, \Im(z)>0\} \subset \Omega_4, \frac{1}{2}+i\frac{\sqrt3}{2} \in \Omega_6.
\endaligned\end{equation}

  \item [c]: Invariance:
 \begin{equation}\aligned\nonumber
z\in\Omega_4\Rightarrow \Gamma(z)\in\Omega_4; z\in\Omega_6\Rightarrow \Gamma(z)\in\Omega_6.
\endaligned\end{equation}

 Here the modular group is
\begin{equation}\aligned\label{modular}
\Gamma:=SL_2(\mathbb{Z})=\{
\left(
  \begin{array}{cc}
    a & b \\
    c & d \\
  \end{array}
\right), ad-bc=1, a, b, c, d\in\mathbb{Z}
\}.
\endaligned\end{equation}

\end{itemize}

\end{conjecture}

\begin{remark} This conjecture has some similarity to  the discovery in Lin-Wang \cite{Lin2010}, in which they showed  surprisingly that the Green function on the two dimensional torus has either 3 or 5 critical points.
\end{remark}
In summary, we  we see that  $(a,b)=(\frac{1}{3},\frac{1}{3})$
is not always a critical point of $\mathcal{J}(z;a,b)$ for $z \in\mathbb{H}$, while
$(a,b)=(\frac{1}{2},\frac{1}{2})$ is always the critical point of $\mathcal{J}(z;a,b)$ for all $z \in\mathbb{H}$. Moreover $(a,b)=(\frac{1}{2}, \frac{1}{2})$ is a local minimum at least for $z= i y, y>0$.

When $(a,b)=w_3=(\frac{1}{2},\frac{1}{2})$ we can simplify the Mueller-Ho functional using the following (whose proof will be given in Section 9)
\begin{lemma}\label{Cw3}
\begin{equation}\aligned\nonumber
\mathcal{J}(z;\frac{1}{2},\frac{1}{2}) =2\theta (2, \frac{z+1}{2})-\theta (1; z).
  \endaligned\end{equation}
\end{lemma}

As a consequence the Mueller-Ho functional becomes
\begin{equation}\label{Formula}
{\mathcal E}_{MH} (z; \frac{1}{2},\frac{1}{2})=(1-\alpha)\theta (1; z)+2\alpha \theta (2, \frac{z+1}{2}).
\end{equation}

Applying Theorem \ref{Th1} with $\rho= \frac{1-\alpha}{2\alpha}$, we have the following
\begin{theorem}\label{ThBEC} For the Mueller-Ho functional $\mathcal{E}_{MH} (z; \frac{1}{2}, \frac{1}{2})$, there exists thresholds $\alpha_1 \sim 0.3732155067\cdots <\alpha_2 \sim 0.9256496973\cdots$  such that
\begin{enumerate}
  \item [(1)] for $\alpha\in[0,\alpha_1]$, the minimizer is rhombic lattice $z=e^{i\theta_\alpha}$ given by
   $$\theta_\alpha=\arctan(\frac{2y_{2,\frac{1-\alpha}{2\alpha}}}{y_{2,\frac{1-\alpha}{2\alpha}}^2-1}),$$
   and the angle increases from $\frac{\pi}{3}$ to $\frac{\pi}{2}$;
  \item [(2)] for $\alpha\in[\alpha_1,\alpha_2]$, the minimizer is square lattice;
  \item [(3)] for $\alpha\in[\alpha_2,1]$, the minimizer is rectangular lattice $(iy_{1,\frac{1-\alpha}{2\alpha}})$ and the ratio of long side and short side increases from 1 to $\sqrt3$.
  \end{enumerate}

\end{theorem}

Proposition \ref{Lstable} and Theorem \ref{ThBEC} give a partial answer to the (a), (c), (d) and (e) part of Mueller-Ho Conjecture. Theorem \ref{ThBEC} shows that as the competition strength between the two Bose gases increases the lattice structures moves from
hexgonal, rhombus, square to rectangular. See Figure 3.

\begin{figure}
\centering
 \includegraphics[scale=0.11]{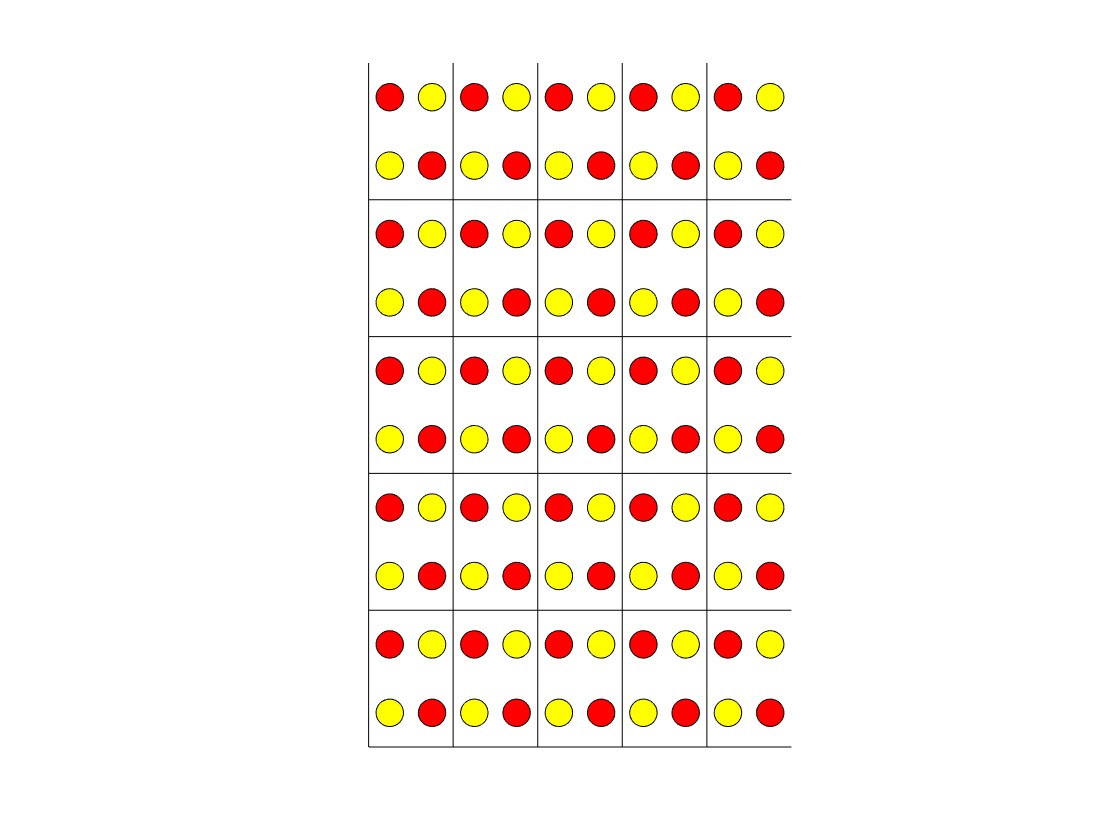}
 \includegraphics[scale=0.11]{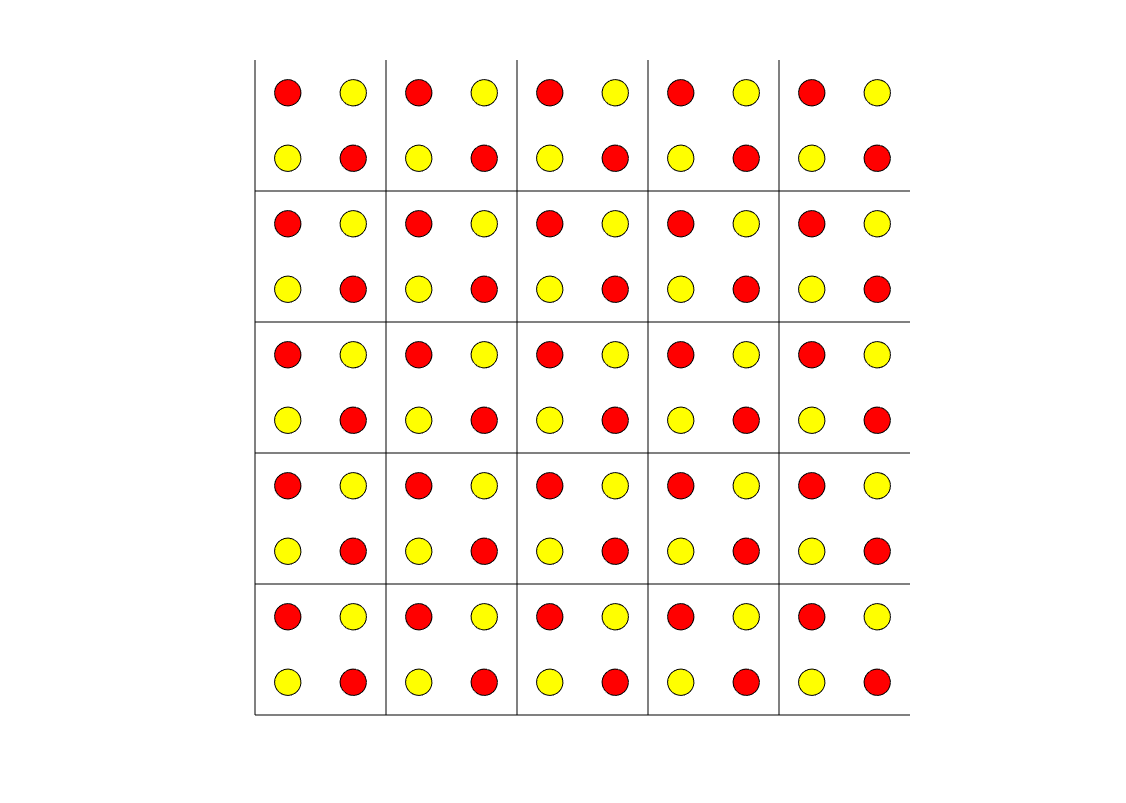}
 \includegraphics[scale=0.08]{rhombus.png}
 \includegraphics[scale=0.08]{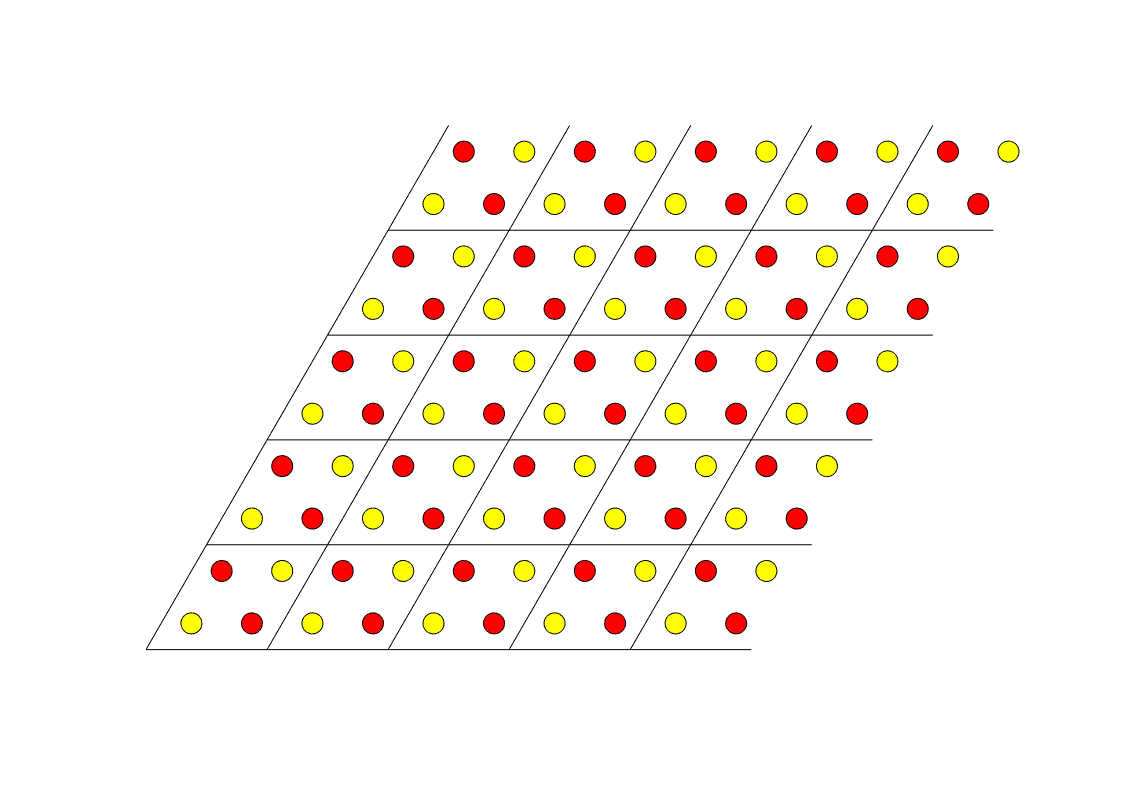}
 \caption{Two-component Bose gas in lattices.
   First row from left to right: a rectangular lattice and a square lattice.
   Second row from left to right: a rhombic lattice and a hexagonal lattice.}
\label{f-lattices}
\end{figure}

Finally we discuss the (b) part of Mueller-Ho Conjecture.
In the Mueller-Ho Conjecture, the expected lattice structure when $\alpha$ is small is triangular lattice, and the relative position of the two components $A,B$ is characterized by $(a,b)=(\frac{1}{3},\frac{1}{3})$.  To see this, there a clear competition between
$\theta (1; z)+\alpha\mathcal{J}(z;\frac{1}{2}, \frac{1}{2})$ and $\theta (1; z)+\alpha\mathcal{J}(z;\frac{1}{3},\frac{1}{3})$ when $\alpha$ is small.  Thus the upper bound of $\alpha$ preserving the triangular lattice structure is determined by the
\begin{equation}\aligned
\alpha_0:=\max_{\alpha\in[0,1]} \{ \alpha \ | \ \theta (1; \frac{1}{2}+i\frac{\sqrt3}{2})+\alpha\mathcal{J}(\frac{1}{2}+i\frac{\sqrt3}{2};\frac{1}{3},\frac{1}{3})
\leq\min_{z\in\mathbb{H}}\big(\theta (1;z)+\alpha\mathcal{J}(z;\frac{1}{2}, \frac{1}{2})\big) \}.
\endaligned
\end{equation}
To find $\alpha_0$, one first uses $\min_{z\in\mathbb{H}}\big(\theta (1;z)+\alpha\mathcal{J}(z;w_3)\big)\leq\theta (1;i)+\alpha\mathcal{J}(i;\frac{1}{2}, \frac{1}{2})$ to obtain a rough bound
\begin{equation}\aligned\label{a0ine}
\alpha_0\leq\frac{\theta(1;i)-\theta (1;\frac{1}{2}+i\frac{\sqrt3}{2})}{\mathcal{J}(\frac{1}{2}+i\frac{\sqrt3}{2};\frac{1}{2},\frac{1}{2})-\mathcal{J}(i;\frac{1}{2},\frac{1}{2})}
:=0.2419435012\cdots.
\endaligned
\end{equation}
By Theorem \ref{ThBEC}, one deduces that
\begin{equation}\aligned
\max_{\alpha\in[0,1]} \{ \alpha \ | \ \theta(1;\frac{1}{2}+i\frac{\sqrt3}{2})+\alpha\mathcal{J}(\frac{1}{2}+i\frac{\sqrt3}{2};\frac{1}{3},\frac{1}{3})
\leq\big(\theta (1; e^{i\theta_\alpha})+\alpha\mathcal{J}(e^{i\theta_\alpha};\frac{1}{2}, \frac{1}{2})\big)\}.
\endaligned
\end{equation}
In view of \eqref{a0ine}, the upper bound $\alpha_0$ satisfies the equation
\begin{equation}\aligned\label{inea}
\theta (1; \frac{1}{2}+i\frac{\sqrt3}{2})+\alpha\mathcal{J}(\frac{1}{2}+i\frac{\sqrt3}{2};\frac{1}{3},\frac{1}{3})
=\theta (1; e^{i\theta_\alpha})+\alpha\mathcal{J}(e^{i\theta_\alpha};\frac{1}{2}, \frac{1}{2}).
\endaligned
\end{equation}
Equation \eqref{inea} gives the upper bound in $($b$)$ of Mueller-Ho Conjecture which is
\begin{equation}\aligned
\alpha_0=0.1726645\cdots,
\theta_{\alpha_0}=1.186248384\cdots.
\endaligned
\end{equation}

In summary we have a complete proof of Mueller-Ho Conjecture as long as the conjecture on the critical points is proved.

The rest of the paper is organized as follows:  In Section 3, we collect some basic invariance properties of the functionals $\mathcal{W}_{1,\rho}(z)$ and $\mathcal{W}_{2,\rho}(z)$ and discuss the intricate relations between these two functionals. In Section 4, we prove a fundamental monotonicity  property of the theta function $\theta (s; \frac{z+1}{2})$. The conjugate monotonicity of $\mathcal{W}_{1,\rho}(z)$ and $\mathcal{W}_{2,\rho}(z)$ are established in Section 5. In Sections 6 and 7, we classify the shape of $\mathcal{W}_{1,\rho}(z)$ and $\mathcal{W}_{2,\rho}(z)$ on the
$y-$axis for all $\rho\in[0,\infty)$ respectively. In Section 8, we prove Theorems \ref{Th1}, \ref{Th2} and \ref{Th3}, the method of the proof relies on the
properties established in Sections 3-7. In Section 9, we prove the properties on Mueller-Ho functional and Theorem \ref{ThBEC}.

In the remaining part of the paper we use  the  common notation $
\sum_{m,n}:=\sum_{(m,n)\in\mathbb{Z}^2}$ so that the theta function becomes $
\theta (s;z)=\sum_{(m,n)} e^{-s \pi\frac{1}{y}|mz+n|^2}$. We also use  the notation:
\begin{equation}\aligned\label{Notation1}
\pi=\left(
      \begin{array}{cc}
        a & b \\
        c & d \\
      \end{array}
    \right)
\Leftrightarrow \pi(\tau)=\frac{a\tau+b}{c\tau+d}.
\endaligned\end{equation}

\section{Some preliminaries}
\setcounter{equation}{0}

In this section we present some simple symmetries of the two theta functions  $\theta (s; z)$ and $\theta (s; \frac{z+1}{2})$  and the associated fundamental domains.  As a result we establish the precise  connection between $\mathcal{W}_{1,\rho}(z)$ and $\mathcal{W}_{2,\rho}(z)$.

Let
$
\mathbb{H}
$
 denote the upper half plane and  $\Gamma $ denote the modular group (defined at \eqref{modular}).

We use the following definition of fundamental domain which is slightly different from the classical definition (see \cite{Mon1988}):
\begin{definition} ([page 108, \cite{Eva1973}]
The fundamental domain associated to group $G$ is a connected domain $\mathcal{D}$ satisfies
\begin{itemize}
  \item For any $z\in\mathbb{H}$, there exists $\pi(z)\in G$ such that $\pi(z)\in\overline{\mathcal{D}}$;
  \item Suppose $z_1,z_2\in\mathcal{D}$ and $\pi(z_1)=z_2$ for some $\pi\in G$, then $z_1=z_2$ and $\pi=\pm Id$.
\end{itemize}
\end{definition}

By Definition 1, the fundamental domain to modular group $\Gamma$ is
\begin{equation}\aligned\label{Fd1}
\mathcal{D}_\Gamma:=\{
z\in\mathbb{H}: |z|>1,\; -\frac{1}{2}<x<\frac{1}{2}
\}.
\endaligned\end{equation}
which is open.  Note that the fundamental domain can be open. (See [page 30, \cite{Apo1976}].)

Next we we introduce another two groups related  to the functionals $\mathcal{W}_{1,\rho}$ and $ \mathcal{W}_{2, \rho}$. The generators of these groups are given by
\begin{equation}\aligned\label{GroupG1}
\mathcal{G}_1: \hbox{the group generated by} \;\;\tau\mapsto -\frac{1}{\tau},\;\; \tau\mapsto \tau+1,\;\;\tau\mapsto -\overline{\tau},
\endaligned\end{equation}
\begin{equation}\aligned\label{GroupG2}
\mathcal{G}_2: \hbox{the group generated by} \;\;\tau\mapsto -\frac{1}{\tau},\;\; \tau\mapsto \tau+2,\;\;\tau\mapsto -\overline{\tau}.
\endaligned\end{equation}
It is easy to see that
the fundamental domains to group $\mathcal{G}_j,j=1,2$ denoted by $\mathcal{D}_{\mathcal{G}_1},\mathcal{D}_{\mathcal{G}_2}$ are
\begin{equation}\aligned\label{Fd3}
\mathcal{D}_{\mathcal{G}_1}:=\{
z\in\mathbb{H}: |z|>1,\; 0<x<\frac{1}{2}
\}
\endaligned\end{equation}
\begin{equation}\aligned\label{Fd4}
\mathcal{D}_{\mathcal{G}_2}:=\{
z\in\mathbb{H}: |z|>1,\; 0<x<1
\}.
\endaligned\end{equation}
Clearly we have that
\begin{equation}\aligned\nonumber
\mathcal{G}_1\supseteq\mathcal{G}_2,\;\;\mathcal{D}_{\mathcal{G}_1}\subseteq\mathcal{D}_{\mathcal{G}_2}.
\endaligned\end{equation}

As in \cite{Mon1988}, the fundamental domain for the single theta function $\theta (s; z)$ is $\mathcal{D}_{\mathcal{G}_1}$. As we will show in this section the fundamental domain for the sum of two theta functions $\mathcal{W}_{1, \rho}, \mathcal{W}_{2, \rho}$ is $\mathcal{D}_{\mathcal{G}_2}$, which is larger.

The follow lemma characterizes the basic symmetries of the theta functions $\theta (s; z)$ and $\theta (s; \frac{z+1}{2})$. The proof is trivial so we omit it.
\begin{lemma}\label{G111}
\begin{itemize}
\item For any $s>0$, any $\gamma\in \mathcal{G}_1$ and $z\in\mathbb{H}$,
$\ \theta (s; \gamma(z))=\theta (s;z)$.
 \item For any $s>0$, any $\gamma\in \mathcal{G}_2$ and $z\in\mathbb{H}$, $\
\theta (s; \frac{\gamma(z)+1}{2})=\theta (s; \frac{z+1}{2}). $
\end{itemize}
\end{lemma}

  A corollary of Lemma \ref{G111} yields
\begin{lemma} \label{Thgamma}
For any $\rho\in\mathbb{R}$, $\gamma\in \mathcal{G}_2$ and $z\in\mathbb{H}$,
\begin{equation}\nonumber
\mathcal{W}_{1,\rho}(\gamma(z))=\mathcal{W}_{1,\rho}(z), \
\mathcal{W}_{2,\rho}(\gamma(z))=\mathcal{W}_{2,\rho}(z).
\end{equation}
\end{lemma}

Next, we introduce the nonlinear connection between the two functionals $\mathcal{W}_{1,\rho}(\tau)$ and $ \mathcal{W}_{2,\rho}(\tau)$.

Let  $w\in \mathcal{G}_2$ be  $w: \tau\mapsto\frac{\tau-1}{\tau+1}$ and its the inverse  be $\tau: w\mapsto\frac{1+w}{1-w}.$
We have

\begin{lemma}\label{ThWW}
\begin{equation}\aligned\label{G333}
\theta (s; \frac{\tau+1}{2})=\theta (s; w),\;\;\theta (s;\tau)=\theta (s; \frac{w+1}{2}).
\endaligned\end{equation}
\begin{equation}
\label{311}
\mathcal{W}_{1,\rho}(\tau)=\rho\cdot \mathcal{W}_{2,1/\rho}(w), \
\mathcal{W}_{2,\rho}(\tau)=\rho\cdot \mathcal{W}_{1,1/\rho}(w).
\end{equation}
Or equivalently,
\begin{equation}
\label{312}
\mathcal{W}_{1,\rho}(w)=\rho\cdot \mathcal{W}_{2,1/\rho}(\tau), \
\mathcal{W}_{2,\rho}(w)=\rho\cdot \mathcal{W}_{1,1/\rho}(\tau).
\end{equation}
\end{lemma}

\begin{proof}

We check that $\theta (s; \frac{\tau+1}{2})=\theta (s; \frac{\frac{1+w}{1-w}+1}{2})=\theta (s; \frac{1}{1-w})=\theta (s; w)$ since the map $w\mapsto\frac{1}{1-w}\in \mathcal{G}_1$. Similarly $
\theta (s; \frac{w+1}{2})=\theta (s; \frac{\frac{\tau-1}{\tau+1}+1}{2})=\theta (s; \frac{\tau}{\tau+1})=\theta (s; \tau)$ since the map $ \tau\mapsto\frac{\tau}{1+\tau}\in \mathcal{G}_1$. This proves \eqref{G333}.

\eqref{311} and \eqref{312} follows from \eqref{G333}.

\end{proof}

Lemma \ref{ThWW} builds  a connection between the two functionals $\mathcal{W}_{1,\rho}(\tau)$ and $ \mathcal{W}_{2,\rho}(\tau)$
via a special element in $\mathcal{G}_2$. As an application of Lemma \ref{ThWW}, we have the following lemma which transfers the computations on unit circles to straight lines.
\begin{lemma}\label{Reduce} Suppose $|w|=1, w=w_1+iw_2$. There holds
\begin{equation}\aligned\nonumber
\frac{\partial}{\partial w_1}\mathcal{W}_{p,\rho}(w)
&=\rho\frac{\sqrt{1-w_1^2}}{1-w_1}\frac{\partial}{\partial \tau_2}\mathcal{W}_{q,1/\rho}(i\frac{\sqrt{1-w_1^2}}{1-w_1})\\
\frac{\partial}{\partial w_2}\mathcal{W}_{p,\rho}(w)
&=-\rho\frac{w_1}{1-w_1}\frac{\partial}{\partial \tau_2}\mathcal{W}_{q,1/\rho}(i\frac{\sqrt{1-w_1^2}}{1-w_1}),
\endaligned\end{equation}
where $p\neq q\in\{1,2\}$.

\end{lemma}

\begin{proof} Let $\tau=\tau_1+i\tau_2, w=w_1+iw_2$. Then we have
\begin{equation}\aligned\nonumber
\tau_1=\frac{1-w_1^2-w_2^2}{(1-w_1)^2+w_2^2},\;\;\tau_2=\frac{2w_2}{(1-w_1)^2+w_2^2}.
\endaligned\end{equation}
Differentiating the identities in Lemma \ref{ThWW}, we get
\begin{equation}\aligned\label{L331}
\frac{\partial}{\partial w_j}\mathcal{W}_{p,\rho}(w)=\rho\sum_{k=1}^2\frac{\partial}{\partial \tau_k}\mathcal{W}_{q,1/\rho}(\tau)\frac{\partial \tau_k}{\partial w_j}, j=1,2.
\endaligned\end{equation}
On the other hand, for $|w|=1$, calculations show
\begin{equation}\aligned\label{L332}
\tau_1=0,\;\; \tau_2=\frac{\sqrt{1-w_1^2}}{1-w_1}
\endaligned\end{equation}
and
\begin{equation}\aligned\label{L333}
\frac{\partial \tau_2}{\partial w_1}=\frac{\sqrt{1-w_1^2}}{1-w_1},\;\;\frac{\partial \tau_2}{\partial w_2}=-\frac{w_1}{1-w_1}.
\endaligned\end{equation}
From Theorem \ref{Thgamma},
$
\mathcal{W}_{p,\rho}(-\overline{\tau})=\mathcal{W}_{p,\rho}(\tau), p=1,2.
$
It follows that
\begin{equation}\aligned\label{L334}
\frac{\partial}{\partial \tau_1}\mathcal{W}_{p,\rho}(i \tau_2)
=0,\;\; \forall \tau_2\;\in\mathbb{R}, p=1,2.
\endaligned\end{equation}
Plugging \eqref{L332}, \eqref{L333} and \eqref{L334} into \eqref{L331}, one gets the result.

\end{proof}

\section{Monotonicity of  $\theta (s; \frac{z+1}{2})$}

\setcounter{equation}{0}


The main purpose of this section is to establish the monotonicity of the functional $\theta(s;\frac{z+1}{2})$ on its fundamental domain $\mathcal{D}_{\mathcal{G}_2}$ (defined at \eqref{GroupG2}), which is the following

\begin{theorem} \label{Lemma2}
\begin{itemize}
  \item  For any $s>0$, there holds
\begin{equation}\aligned\nonumber
\frac{\partial}{\partial x}\theta(s;\frac{z+1}{2})>0, \;\;\forall \;z\in \mathcal{D}_{\mathcal{G}_2}.
\endaligned\end{equation}
  \item Or equivalently, via the map $ z \mapsto \frac{z+1}{2}$, for any $s>0$,
$$
\frac{\partial}{\partial x}\theta(s;z)<0,\;\;\forall \;z\in\Omega_{\mathcal{C}_1}.
$$
Here
\begin{equation}\aligned\nonumber
\Omega_{\mathcal{C}_1}:=\{z\mid 0<x<\frac{1}{2}, y>\sqrt{x-x^2}
\}.
\endaligned\end{equation}

\end{itemize}

\end{theorem}
\begin{remark} In Lemma 1 of \cite{Mon1988} Montgomery proved that
\begin{equation}\aligned\label{MonLLL}
\frac{\partial}{\partial x}\theta(s;z)<0,\;\;\forall \;z\in \mathcal{D}_{\mathcal{G}_1}:=\{
z\in\mathbb{H}: |z|>1,\; 0<x<\frac{1}{2}
\}
\endaligned\end{equation}
Theorem \ref{Lemma2}  improves this result to a larger domain $\Omega_{\mathcal{C}_1}$ as
$\mathcal{D}_{\mathcal{G}_1} \subset \Omega_{\mathcal{C}_1}$. Furthermore, $\Omega_{\mathcal{C}_1}$ contains a corner at $z=0$, which makes the proof much more involved. We have to divide $\Omega_{\mathcal{C}_1}$  into four different cases to overcome this difficulty.

\end{remark}

 We state two corollaries related to the functionals $\mathcal{W}_{j,\rho}(z),j=1,2$.

\begin{corollary}\label{Coro1} For any $s>0$,
$$
\frac{\partial}{\partial x}\theta(s;z)>0,\;\forall z\in\Omega_{\mathcal{C}_2}.
$$
Here
\begin{equation}\aligned\nonumber
\Omega_{\mathcal{C}_2}:=\{z\mid \frac{1}{2}<x<1, y>\sqrt{x-x^2}
\}.
\endaligned\end{equation}
\end{corollary}

\begin{proof} Since $z\mapsto1-\overline{z}\in\mathcal{G}_1$, by Lemma \ref{G111}, we have $
\theta(s;1-\overline{z})=\theta(s;z)$.
Thus
\begin{equation}\aligned\label{RR}
\frac{\partial}{\partial x}\theta(s;1-\overline{z})=-\frac{\partial}{\partial x}\theta(s;z).
\endaligned\end{equation}
The result follows by \eqref{RR} and Theorem \ref{Lemma2}.

\end{proof}

By Theorem \ref{Lemma2} and Corollary \ref{Coro1} we have

\begin{corollary}\label{Coro2} For any $\rho >0$,
$$
\frac{\partial}{\partial x}\mathcal{W}_{j,\rho}(z)>0,\;\forall z\in{\mathcal{R}_L}, j=1,2.
$$
Here
$$
{\mathcal{R}_L}:=\Omega_{\mathcal{C}_2}\cap\mathcal{D}_{\mathcal{G}_2}=\{z\mid \frac{1}{2}<x<1, |z|>1\}.
$$

\end{corollary}

\vskip0.2in
In the remaining part of this section, we prove Theorem \ref{Lemma2}. To prove Theorem \ref{Lemma2},  we use some delicate analysis of the Jacobi theta function and Poisson summation formula.

We first recall the following well-known Jacob triple product formula:
\begin{equation}\aligned\label{Jacob1}
\prod_{n=1}^\infty(1-x^{2m})(1+x^{2m-1}y^2)(1+\frac{x^{2m-1}}{y^2})=\sum_{n=-\infty}^\infty x^{n^2} y^{2n}
 \endaligned\end{equation}
for complex numbers $x,y$ with $|x|<1$, $y\neq0$.

The Jacob theta function is defined as
\begin{equation}\aligned\nonumber
\vartheta_J(z;\tau):=\sum_{n=-\infty}^\infty e^{i\pi n^2 \tau+2\pi i n z},
 \endaligned\end{equation}
and the classical one-dimensional theta function  is given by
\begin{equation}\aligned\label{TXY}
\vartheta(X;Y):=\vartheta_J(Y;iX)=\sum_{n=-\infty}^\infty e^{-\pi n^2 X} e^{2\pi i Y}.
 \endaligned\end{equation}
Hence by the Jacob triple product formula \eqref{Jacob1}, we have
\begin{equation}\aligned\label{Product}
\vartheta(X;Y)=\prod_{n=1}^\infty(1-e^{-2\pi n X})(1+e^{-2(2n-1)\pi X}+2e^{-(2n-1)\pi X}\cos(2\pi Y)).
 \endaligned\end{equation}

The following two Lemmas improve the bounds in  Montgomery \cite{Mon1988}.  We provide the proof of Lemma \ref{LemmaTTT} and omit the proof of Lemma \ref{LemmaT2} which is similar.

\begin{lemma}\label{LemmaTTT} Assume $X>\frac{1}{5}$. If $\sin(2\pi Y)>0$, then
\begin{equation}\aligned\nonumber
-\overline\vartheta(X)\sin(2\pi Y)\leq\frac{\partial}{\partial Y}\vartheta(X;Y)\leq-\underline\vartheta(X)\sin(2\pi Y).
 \endaligned\end{equation}
If $\sin(2\pi Y)<0$, then
\begin{equation}\aligned\nonumber
-\underline\vartheta(X)\sin(2\pi Y)\leq\frac{\partial}{\partial Y}\vartheta(X;Y)\leq-\overline\vartheta(X)\sin(2\pi Y).
 \endaligned\end{equation}
Here
\begin{equation}\aligned\nonumber
\underline\vartheta(X):=4\pi e^{-\pi X}(1-\mu(X)), \;\; \overline\vartheta(X):=4\pi e^{-\pi X}(1+\mu(X)),
 \endaligned\end{equation}
and
$$
\mu(X):=\sum_{n=2}^\infty n^2 e^{-\pi(n^2-1)X}.
$$

\end{lemma}
\begin{proof}
We use the same method as in Lemma 1 of \cite{Mon1988}.
Taking logarithmic on both sides of \eqref{Product} and differentiating  $\frac{\partial}{\partial Y}$, we have
\begin{equation}\aligned\label{QQQ}
-\frac{\frac{\partial}{\partial Y}\vartheta(X;Y)}{\sin(2\pi Y)}&=4\pi\sum_{n=1}^\infty e^{-(2n-1)\pi X}\frac{\vartheta(X;Y)}{1+e^{-2(2n-1)\pi X}+2e^{-(2n-1)\pi X}\cos(2\pi Y)}\\
&=4\pi\sum_{n=1}^\infty e^{-(2n-1)\pi X}\prod_{m\neq n,m=1}^\infty(1-e^{-2\pi m X})(1+e^{-2(2m-1)\pi X}+2e^{-(2m-1)\pi X}\cos(2\pi Y)).
 \endaligned\end{equation}
One sees from \eqref{QQQ} that the function $-\frac{\frac{\partial}{\partial Y}\vartheta(X;Y)}{\sin(2\pi Y)}$ has a period $1$, is decreasing on $[0,\frac{1}{2}]$
and is an even function for $Y$.

Thus
\begin{equation}\aligned\label{QQQ1}
\lim_{Y\rightarrow\frac{1}{2}}-\frac{\frac{\partial}{\partial Y}\vartheta(X;Y)}{\sin(2\pi Y)}&\leq-\frac{\frac{\partial}{\partial Y}\vartheta(X;Y)}{\sin(2\pi Y)}\leq\lim_{Y\rightarrow0}-\frac{\frac{\partial}{\partial Y}\vartheta(X;Y)}{\sin(2\pi Y)}.
\endaligned\end{equation}
 By L'Hospital rule we have
\begin{equation}\aligned\label{QQQ2}
\frac{1}{2\pi}\frac{\partial^2}{\partial Y^2}\vartheta(X;Y)\mid_{Y=\frac{1}{2}}\leq-\frac{\frac{\partial}{\partial Y}\vartheta(X;Y)}{\sin(2\pi Y)}\leq-\frac{1}{2\pi}\frac{\partial^2}{\partial Y^2}\vartheta(X;Y)\mid_{Y=0}\\
\endaligned\end{equation}

From \eqref{TXY}, we have that
\begin{equation}\aligned\label{QQQ3}
\frac{\partial^2}{\partial Y^2}\vartheta(X;Y)\mid_{Y=0}&=4\pi e^{-\pi X}(1+\sum_{n=2}^\infty n^2 e^{-\pi(n^2-1) X})\\
\frac{1}{2\pi}\frac{\partial^2}{\partial Y^2}\vartheta(X;Y)\mid_{Y=\frac{1}{2}}&=4\pi\sum_{n=1}^\infty (-1)^{n-1}n^2 e^{-n^2\pi X}
\geq4\pi e^{-\pi X}(1-\sum_{n=2}^\infty n^2 e^{-\pi(n^2-1) X}).
\endaligned\end{equation}
Combining \eqref{QQQ1}, \eqref{QQQ2} and \eqref{QQQ3}, we obtain the proof of the Lemma.

\end{proof}


\begin{lemma}\label{LemmaT2}
Assume $X<\frac{\pi}{2}$. If $\sin(2\pi Y)>0$, then
\begin{equation}\aligned\nonumber
-\overline\vartheta(X)\sin(2\pi Y)\leq\frac{\partial}{\partial Y}\vartheta(X;Y)\leq-\underline\vartheta(X)\sin(2\pi Y).
 \endaligned\end{equation}
If $\sin(2\pi Y)<0$, then
\begin{equation}\aligned\nonumber
-\underline\vartheta(X)\sin(2\pi Y)\leq\frac{\partial}{\partial Y}\vartheta(X;Y)\leq-\overline\vartheta(X)\sin(2\pi Y).
 \endaligned\end{equation}
Here
\begin{equation}\aligned\nonumber
\underline\vartheta(X):=X^{-\frac{3}{2}};\;\; \overline\vartheta(X):=\pi e^{-\frac{\pi}{4X}}X^{-\frac{3}{2}}.
 \endaligned\end{equation}

\end{lemma}


In view of \eqref{TXY}, by Poisson summation formula, one has
\begin{equation}\aligned\label{PPP2}
\vartheta(X;Y)=X^{-\frac{1}{2}}\sum_{n\in \mathbb{Z}} e^{-\frac{\pi(n-Y)^2}{X}}.
\endaligned\end{equation}

Thus the two-dimensional theta function can be written in terms of one-dimensional theta function as follows:

\begin{equation}\aligned\label{PPP3}
\theta (s;z)&=\sum_{(m,n)\in\mathbb{Z}^2} e^{-s \pi\frac{1}{y }|nz+m|^2}
=\sum_{n\in\mathbb{Z}}e^{-s \pi y n^2}\sum_{m\in\mathbb{Z}} e^{-\frac{s \pi (nx+m)^2}{y}}\\
&=\sqrt{\frac{y}{s}}\sum_{n\in\mathbb{Z}}e^{-s \pi y n^2}\vartheta(\frac{y}{s};-nx)=\sqrt{\frac{y}{s}}\sum_{n\in\mathbb{Z}}e^{-s \pi y n^2}\vartheta(\frac{y}{s};nx)\\
&=2\sqrt{\frac{y}{s}}\sum_{n=1}^\infty e^{-s \pi y n^2}\vartheta(\frac{y}{s};nx).
\endaligned\end{equation}


Now we are ready to prove Theorem \ref{Lemma2}.

\begin{proof} By Melin transform, (see \cite{Mon1988}), $\theta(\frac{1}{s};z)=s \theta(s;z)$. Thus we  only need  to consider the case $s \geq1$.

From \eqref{PPP3}, we have
\begin{equation}\label{Eab}\aligned
 -\frac{\partial}{\partial x} \theta(s;z)&=-2\sqrt{\frac{y}{s}}\sum_{n=1}^\infty n e^{-\pi s y n^2}\frac{\partial}{\partial Y}\vartheta(\frac{y}{s};Y)|_{Y=nx}\\
 &=2\sqrt{\frac{y}{s}}\Big(-\sum_{n\leq\frac{1}{2x}} n e^{-\pi s y n^2}\frac{\partial}{\partial Y}\vartheta(\frac{y}{s};Y)|_{Y=nx}
 -\sum_{n>\frac{1}{2x}} n e^{-\pi s y n^2}\frac{\partial}{\partial Y}\vartheta(\frac{y}{s};Y)|_{Y=nx}\Big) \\
 & =2\sqrt{\frac{y}{s}}\Big(
 \mathcal{E}^a_{\alpha,x}(z)+\mathcal{E}^b_{\alpha,x}(z)
 \Big),
 \endaligned\end{equation}
where
 \begin{equation}\aligned
\mathcal{E}^a_{s,x}(z):=-\sum_{n\leq\frac{1}{2x}} n e^{-\pi s y n^2}\frac{\partial}{\partial Y}\vartheta(\frac{y}{s};Y)|_{Y=nx}, \;\;\mathcal{E}^b_{s,x}(z):=-\sum_{n>\frac{1}{2x}} n e^{-\pi s y n^2}\frac{\partial}{\partial Y}\vartheta(\frac{y}{s};Y)|_{Y=nx}.
 \endaligned\end{equation}

For $\mathcal{E}^a_{s,x}(z)$, by Lemma \ref{LemmaTTT}, we have that
 \begin{equation}\aligned\label{Ea}
 \mathcal{E}^a_{s,x}(z)\geq\sum_{n\leq\frac{1}{2x}}n e^{-\pi s y n^2}\underline\vartheta(\frac{y}{s})\sin(2\pi nx)\geq
 e^{-\pi s y }\underline\vartheta(\frac{y}{s})\sin(2\pi x).
 \endaligned\end{equation}
Notice that all the terms in the summation of \eqref{Ea} are nonnegative.

Let $n_0$ be the smallest integer such that $n>\frac{1}{2x}$. By Lemma \ref{LemmaTTT},
 \begin{equation}\aligned\label{Eb}
 |\mathcal{E}^b_{s,x}(z)|&\leq\sum_{n>\frac{1}{2x}}n e^{-\pi s y n^2}\overline\vartheta(\frac{y}{s})|\sin(2\pi nx)|\leq\sum_{n>\frac{1}{2x}}n^2 e^{-\pi s y n^2}\overline\vartheta(\frac{y}{s})|\sin(2\pi x)|\\
 &= n_0^2 e^{-\pi s y n_0^2}\overline\vartheta(\frac{y}{s})\sin(2\pi x)\cdot
 \Big(1+\delta(x)
 \Big),\hbox{with}\;\delta(x):=
 \sum_{k=1}^\infty (1+\frac{k}{n_0})^2 e^{-\pi s y(2kn_0+k^2)}.
 \endaligned\end{equation}

 To estimate $\delta(x)$, note that $yn_0>\frac{\sqrt{1-x}}{2\sqrt x}$,
 \begin{equation}\aligned\label{Ed}
\delta(x)&\leq  \sum_{k=1}^\infty (1+\frac{2k}{n_0}+\frac{k^2}{n_0^2}) e^{-2\pi s ykn_0}\leq\sum_{k=1}^\infty (1+\frac{2k}{n_0}+\frac{k^2}{n_0^2}) e^{-\pi\frac{\sqrt{1-x}}{\sqrt x} k}\\
&=\frac{e^{-q(x)}}{1-e^{-q(x)}}+\frac{2}{n_0}\frac{e^{-q(x)}}{(1-e^{-q(x)})^2}+\frac{1}{n_0^2}\frac{e^{-q(x)}(1+e^{-q(x)})}{(1-e^{-q(x)})^2}\\
&\leq\frac{e^{-q(x)}}{1-e^{-q(x)}}+4x\frac{e^{-q(x)}}{(1-e^{-q(x)})^2}+4x^2\frac{e^{-q(x)}(1+e^{-q(x)})}{(1-e^{-q(x)})^2}
 \endaligned\end{equation}
with $q(x):=\pi\frac{\sqrt{1-x}}{\sqrt x}$. Denote that
\begin{equation}\aligned\nonumber
\delta_q(x):=\frac{e^{-q(x)}}{1-e^{-q(x)}}+4x\frac{e^{-q(x)}}{(1-e^{-q(x)})^2}+4x^2\frac{e^{-q(x)}(1+e^{-q(x)})}{(1-e^{-q(x)})^2}.
 \endaligned\end{equation}
It is easy to see that $\delta_q(x)$ is monotonically increasing on $[0,\frac{1}{2}]$ and hence $\delta(x)\leq\delta_q(\frac{1}{2})=0.188822585\cdots<\frac{1}{5}$. Then by \eqref{Eb} and \eqref{Ed}, one has
 \begin{equation}\aligned\label{Ebb}
 |\mathcal{E}^b_{s,x}(z)|\leq\frac{6}{5} n_0^2 e^{-\pi s y n_0^2}\overline\vartheta(\frac{y}{s})\sin(2\pi x).
\endaligned\end{equation}
Combining \eqref{Eab}, \eqref{Ea} with \eqref{Ebb}, one gets
 \begin{equation}\aligned\label{Eabab}
 -\frac{\partial}{\partial x} \theta(s;z)&\geq2\sqrt{\frac{y}{s}}\sin(2\pi x)e^{-\pi s y}\overline\vartheta(\frac{y}{s})\Big(
\frac{\underline\vartheta(\frac{y}{s})}{\overline\vartheta(\frac{y}{s})}-\frac{6}{5}n_0^2 e^{-\pi s y (n_0^2-1)}
\Big),
 \endaligned\end{equation}
with $n_0=[\frac{1}{2x}]+1$.

Let
\begin{equation}\aligned\label{Enew}
\mathcal{E}_{s,x}(z):=
\frac{\underline\vartheta(\frac{y}{s})}{\overline\vartheta(\frac{y}{s})}-\frac{6}{5}n_0^2 e^{-\pi s y (n_0^2-1)}.
 \endaligned\end{equation}
By \eqref{Eabab}
it suffices to prove that
$\mathcal{E}_{s,x}(z)>0$.

$\Omega_{\mathcal{C}_1}$ has a corner $z=0$ which induces the difficulty to get the lower bound estimate for $\mathcal{E}_{s,x}(z)$. Thus we divide the proof into four cases.
\vskip0.1in

\noindent
{\bf Case a: $\frac{y}{s}\leq\frac{1}{2},x\in(0,\frac{1}{3}]$}. In this case, $\frac{s}{y}\geq2$ and $\frac{\sqrt{1-x}(1-4x^2)}{x^{\frac{3}{2}}}-\frac{1}{\sqrt{x-x^2}}>0$.
By Lemma \ref{LemmaT2},
\begin{equation}\aligned\label{EEE1}
\mathcal{E}_{s,x}(z)&\geq (\frac{\pi s}{y}-2)e^{-\frac{\pi s}{4y}}-\frac{6}{5}n_0^2 e^{-\pi s y (n_0^2-1)}\\
&\geq(2\pi-2)e^{-\frac{\pi s}{4\sqrt{x-x^2}}}-\frac{3}{10x^2}e^{-\pi s \sqrt{x-x^2} (\frac{1}{4x^2}-1)}\\
&=\frac{3}{10x^2}e^{-\pi s \sqrt{x-x^2} (\frac{1}{4x^2}-1)}\Big(\frac{20\pi-20}{3}x^2
e^{\frac{\pi s}{4}(\frac{\sqrt{1-x}(1-4x^2)}{x^{\frac{3}{2}}}-\frac{1}{\sqrt{x-x^2}})}-1
\Big)\\
&\geq\frac{3}{10x^2}e^{-\pi s \sqrt{x-x^2} (\frac{1}{4x^2}-1)}\Big(\frac{20\pi-20}{3}x^2
e^{\frac{\pi}{4}(\frac{\sqrt{1-x}(1-4x^2)}{x^{\frac{3}{2}}}-\frac{1}{\sqrt{x-x^2}})}-1
\Big) \\
& >0
 \endaligned\end{equation}
 where the last inequality follows from elementary calculus because $ x\in (0, \frac{1}{3})$.

 \medskip

\noindent
{\bf Case b: $\frac{y}{s}\leq\frac{1}{2},x\in[\frac{1}{3},\frac{1}{2}]$}.
In this case, $n_0=[\frac{1}{2x}]+1\geq\frac{1}{2x}+\frac{1}{2}$ and we have

\begin{equation}\aligned\nonumber
\mathcal{E}_{s,x}(z)&\geq (\frac{\pi s}{y}-2)e^{-\frac{\pi s}{4y}}-\frac{6}{5}n_0^2 e^{-\pi s y (n_0^2-1)}\geq(2\pi-2)e^{-\frac{\pi s}{4y}}-\frac{6}{5}(\frac{1}{2x}+\frac{1}{2})^2e^{-\pi s y((\frac{1}{2x}+\frac{1}{2})^2-1)}\\
&=\frac{3(1+x)^2}{10x^2}e^{-\pi s y((\frac{1}{2x}+\frac{1}{2})^2-1)}
\Big(
\frac{(20\pi-20)x^2}{9(1+x)^2} e^{\pi s \big(y\frac{(1+x)^2-4x^2}{4x^2}-\frac{1}{4y}\big)}-1
\Big)\\
&\geq\frac{3(1+x)^2}{10x^2}e^{-\pi s y((\frac{1}{2x}+\frac{1}{2})^2-1)}
\Big(
\frac{(20\pi-20)x^2}{9(1+x)^2} e^{\pi s \big(\sqrt{x-x^2}\frac{(1+x)^2-4x^2}{4x^2}-\frac{1}{4\sqrt{x-x^2}}\big)}-1
\Big)\\
&\geq\frac{3(1+x)^2}{10x^2}e^{-\pi s y((\frac{1}{2x}+\frac{1}{2})^2-1)}
\Big(
\frac{(20\pi-20)x^2}{9(1+x)^2} e^{\pi\big(\sqrt{x-x^2}\frac{(1+x)^2-4x^2}{4x^2}-\frac{1}{4\sqrt{x-x^2}}\big)}-1
\Big)\\
&>0
 \endaligned\end{equation}
where we have used the following elementary inequalities:
\begin{equation}\aligned\label{Ele2}
&\sqrt{x-x^2}\frac{(1+x)^2-4x^2}{4x^2}-\frac{1}{4\sqrt{x-x^2}}>0,\;\;x\in[0,\frac{1}{2}],\\
&\frac{(20\pi-20)x^2}{9(1+x)^2} e^{\pi\big(\sqrt{x-x^2}\frac{(1+x)^2-4x^2}{4x^2}-\frac{1}{4\sqrt{x-x^2}}\big)}-1>0,\;\;x\in[0,\frac{1}{2}].
 \endaligned\end{equation}

\vskip0.1in

\noindent
{\bf Case c: $\frac{y}{s }\geq\frac{1}{2},x\in[0,\frac{2}{5}]$}.
In this case, $y s\geq\frac{s^2}{2}\geq\frac{1}{2}$.
By Lemma \ref{LemmaTTT},
\begin{equation}\aligned\nonumber
\mathcal{E}_{s,x}(z)&\geq
\frac{1-\mu(\frac{y}{s })}{1+\mu(\frac{y}{s })}
-\frac{6}{5}n_0^2 e^{-\pi s y (n_0^2-1)}\geq\frac{1-\mu(\frac{1}{2})}{1+\mu(\frac{1}{2})}-\frac{3}{10x^2}e^{-\frac{\pi(1-4x^2)}{8x^2}}\\
&\geq\Big(\frac{1-\mu(\frac{1}{2})}{1+\mu(\frac{1}{2})}-\frac{3}{10x^2}e^{-\frac{\pi(1-4x^2)}{8x^2}}\Big)\mid_{x=\frac{2}{5}}=0.1556238052>0.
 \endaligned\end{equation}

\noindent
{\bf Case d: $\frac{y}{s}\geq\frac{1}{2},x\in[\frac{1}{3},\frac{1}{2}]$}.
In this case, $n_0=[\frac{1}{2x}]+1\geq\frac{1}{2x}+\frac{1}{2}$ and $y\geq\frac{s^2}{2}\geq\frac{1}{2}$.
By Lemma \ref{LemmaTTT},
\begin{equation}\aligned\nonumber
\mathcal{E}_{s,x}(z)&\geq
\frac{1-\mu(\frac{y}{s})}{1+\mu(\frac{y}{s})}
-\frac{6}{5}n_0^2 e^{-\pi s y (n_0^2-1)}\geq\frac{1-\mu(\frac{1}{2})}{1+\mu(\frac{1}{2})}-\frac{3(1+x)^2}{10x^2}e^{-\frac{\pi}{2}\big((\frac{1+x}{2x})^2-1\big)}\\
&\geq\Big(\frac{1-\mu(\frac{1}{2})}{1+\mu(\frac{1}{2})}-\frac{3(1+x)^2}{10x^2}e^{-\frac{\pi}{2}\big((\frac{1+x}{2x})^2-1\big)}\Big)\mid_{x=\frac{1}{2}}=0.7866071958\cdots >0.
 \endaligned\end{equation}

Combining cases (a)-(d), \eqref{Eabab} and \eqref{Enew}, the proof of Theorem \ref{Lemma2} is completed.

\end{proof}



\section{Monotonicity of $\mathcal{W}_{1,\rho}(z)$ and $ \mathcal{W}_{2,\rho}(z)$   }

\setcounter{equation}{0}

Let the closure of the left-half fundamental domain corresponding to $\mathcal{G}_2$ be
$$
\mathcal{R}_2=\{z\in\mathbb{H}: 0\leq x \leq\frac{1}{2}, |z|\geq1\}.
$$
In this section, we aim to establish the following  property of the pair $\mathcal{W}_{j,\rho}(z), j=1,2$: there exists $\rho_{*}$ such that
for $\forall z\in \mathcal{R}_2$, $
\frac{\partial}{\partial x}\mathcal{W}_{1,\rho}(z)\geq0$ when   $0\leq\rho\leq\rho_{*} $, and $
\frac{\partial}{\partial x}\mathcal{W}_{2,\rho}(z)\geq0
$ when $ 0\leq \rho \leq \frac{1}{\rho_{*}}$.
(In fact we will choose $\rho_{*}=\frac{1}{20}$.) This property plays an important role in finding the minimizers and will be proved in Propositions \ref{Trans1} and \ref{Trans2}.

We begin with

\begin{proposition}\label{Trans1}
For $0\leq\rho\leq\rho_{*}:=1/20$, there holds
$$
\frac{\partial}{\partial x}\mathcal{W}_{1,\rho}(z)\geq0
$$
for $\forall z\in\mathcal{R}_2$. The equality holds only possible when $x=0$ or $\frac{1}{2}$.
\end{proposition}
\begin{proof}
From \eqref{PPP3}, we obtain that
\begin{equation}\aligned\label{WT}
\frac{\partial}{\partial x}\mathcal{W}_{1,\rho}(z)&=
\frac{\partial}{\partial x}\Big(
(\frac{y}{4}\sum_n e^{-\pi yn^2}\vartheta(\frac{y}{4};n\frac{x+1}{2})+\rho \sqrt y\sum_n e^{-\pi yn^2}\vartheta(y;nx)\Big)\\
&=\frac{\sqrt{y}}{2}\sum_{n=1}^\infty ne^{-\pi yn^2}\frac{\partial}{\partial Y}\vartheta(\frac{y}{4};Y)|_{Y=n\frac{x+1}{2}}
+2\rho\sqrt{y}\sum_{n=1}^\infty ne^{-\pi yn^2}\frac{\partial}{\partial Y}\vartheta(y;Y)|_{Y=nx}
\Big)\\
&=\frac{\sqrt{y}}{2}e^{-\pi y}\frac{\partial }{\partial Y}\vartheta(\frac{y}{4};Y)|_{Y=\frac{x+1}{2}}
+\sqrt{y}e^{-4\pi y}\frac{\partial }{\partial Y}\vartheta(\frac{y}{4};Y)|_{Y={x+1}}\\
&\;+2\rho\sqrt{y}e^{-\pi y}\frac{\partial }{\partial Y}\vartheta(y;Y)|_{Y=x}
+4\rho\sqrt{y}e^{-4\pi y}\frac{\partial }{\partial Y}\vartheta(y;Y)|_{Y=2x}\\
&\;+\frac{\sqrt{y}}{2}\sum_{n=3}^\infty ne^{-\pi yn^2}\frac{\partial}{\partial Y}\vartheta(\frac{y}{4};Y)|_{Y=n\frac{x+1}{2}}
+2\rho\sqrt{y}\sum_{n=3}^\infty ne^{-\pi yn^2}\frac{\partial}{\partial Y}\vartheta(y;Y)|_{Y=nx} \\
& \mathcal{W}_{1, x}^{a}(z)+\mathcal{W}_{1, x}^{b}(z)+\mathcal{W}_{1, x}^{c}(z)
 \endaligned\end{equation}
where $\mathcal{W}_{1, x}^{a}(z),\mathcal{W}_{1, x}^{b}(z)$ and $\mathcal{W}_{1, x}^{c}(z)$ are defined at the last equality.

By Lemma \ref{LemmaTTT}, we see that
\begin{equation}\aligned\label{WT100}
\mathcal{W}_{1, x}^{a}(z)+\mathcal{W}_{1,x}^{b}(z)
&\geq\frac{\sqrt{y}}{2}e^{-\pi y}\underline\vartheta(\frac{y}{4})\sin(\pi x)
-\sqrt{y}e^{-4\pi y}\overline\vartheta(\frac{y}{4})\sin(2\pi x)\\
&\;-2\rho\sqrt{y}e^{-\pi y}\overline\vartheta(y)\sin(2\pi x)
-4\rho\sqrt{y}e^{-4\pi y}\overline\vartheta(y)\sin(4\pi x).
 \endaligned\end{equation}

Since $|\sin(nx)|\leq n|\sin(x)|$ for any $x\in\mathcal{R}_2 $,
again by Lemma \ref{LemmaTTT}, we have
\begin{equation}\aligned\label{WT101}
\mathcal{W}_{1,x}^{c}(z)
&\geq-\frac{\sqrt{y}}{4}\sum_{n=3}^\infty n^2e^{-\pi yn^2}\overline\vartheta(\frac{y}{4})\sin(2\pi x)
-2\rho\sqrt{y}\sum_{n=3}^\infty n^2e^{-\pi yn^2}\overline\vartheta(y)\sin(2\pi x).
 \endaligned\end{equation}
Plugging \eqref{WT100} and \eqref{WT101} in \eqref{WT}, we get
\begin{equation}\aligned\label{WTx1}
\frac{\partial}{\partial x}\mathcal{W}_{1,\rho}(z)
&\geq\frac{\sqrt y}{2}e^{-\pi y}\underline\vartheta(\frac{y}{4})\sin{\pi x}
-\sqrt y e^{-4\pi y}\overline\vartheta(\frac{y}{4})\sin(2\pi x)\big(1+\frac{1}{4}\sum_{n=3}^\infty n^2e^{-\pi y(n^2-4)}\big)\\
&\;-2\rho{\sqrt y}e^{-\pi y}\overline\vartheta(y)\sin(2\pi x)\big(1+\sum_{n=2}^\infty n^2e^{-\pi y(n^2-1)}\big)\\
&=\sqrt y e^{-\pi y}\sin(\pi x)\Big(
\frac{1}{2}\underline\vartheta(\frac{y}{4})-2e^{-3\pi y}\overline\vartheta(\frac{y}{4})\cos(\pi x)(1+\sigma_1)
-4\rho\overline\vartheta(y)\cos(\pi x)(1+\sigma_2)
\Big)\\
&\geq
\sqrt y e^{-\pi y}\sin(\pi x)\Big(
\frac{1}{2}\underline\vartheta(\frac{y}{4})-2e^{-3\pi y}\overline\vartheta(\frac{y}{4})(1+\sigma_1)
-4\rho\overline\vartheta(y)(1+\sigma_2)
\Big),
 \endaligned\end{equation}
where
$$
\sigma_1(y):=\frac{1}{4}\sum_{n=3}^\infty n^2e^{-\pi y(n^2-4)},\;\;\; \sigma_2(y):=\sum_{n=2}^\infty n^2e^{-\pi y(n^2-1)},
$$
and
$\sigma_1(y), \sigma_1(y)$ are small. (In fact
$
\sigma_1(\frac{\sqrt 3}{2})\approx 2.781\cdot 10^{-6},\; \sigma_2(\frac{\sqrt 3}{2})\approx 1.14105\cdot 10^{-3}.
$)

By the lower and upper bound estimates in Lemma \ref{LemmaTTT}, from \eqref{WTx1}, we see that
\begin{equation}\aligned\label{WTxF}
\frac{\partial}{\partial x}\mathcal{W}_{1,\rho}(z)
&\geq
\sqrt y e^{-\pi y}\sin(\pi x)\Big(
2\pi(1-\mu(\frac{y}{4}))e^{-\frac{\pi y}{4}}-8\pi e^{-3\pi y}(1+\mu(\frac{y}{4}))e^{-\frac{\pi y}{4}}(1+\sigma_1)\\
&\;-{16\rho\pi}(1+\mu(y))e^{-\pi y}(1+\sigma_2)
\Big)\\
&=4\pi\sqrt ye^{-\frac{5\pi y}{4}}\sin(\pi x)\Big(
\frac{1}{2}(1-\mu(\frac{y}{4}))-2(1+\sigma_1)e^{-3\pi y}(1+\mu(\frac{y}{4}))\\
&\;-4\rho(1+\sigma_2)e^{-\frac{3\pi y}{4}}(1+\mu(y))
\Big)\\
& = 4\pi\sqrt ye^{-\frac{5\pi y}{4}}\sin(\pi x) \vartheta_{\mathcal{W}_{1,\rho}} (y)
 \endaligned\end{equation}
where $\vartheta_{\mathcal{W}_{1,\rho}} (y)$ is defined at the last equality.

It suffices to prove that
$$
\vartheta_{\mathcal{W}_{1,\rho}}(y)>0.
$$
First it is easy to see that
\begin{equation}\aligned\label{Mono3}
\frac{\partial}{\partial \rho}\vartheta_{\mathcal{W}_{1,\rho}}(y)>0,\;y>0.
 \endaligned\end{equation}
Since the functions $\mu(y), \sigma_1, \sigma_2$ are decreasing on $y>0$, it follows that
\begin{equation}\aligned\label{Mono4}
\frac{\partial}{\partial y}\vartheta_{\mathcal{W}_{1,\rho}} (y)>0,\;y>0.
\endaligned\end{equation}
A direct calculation gives
$$
\vartheta_{\mathcal{W}_{1,\rho}}(y)|_{y=\frac{\sqrt 3}{2},\rho=\frac{1}{20}}=0.1933\cdots>0
$$
which implies
\begin{equation}\aligned\nonumber
\vartheta_{\mathcal{W}_{1,\rho}}>0, \;\hbox{for}\; y\geq\frac{\sqrt 3}{2},\rho\leq\frac{1}{20}
\endaligned\end{equation}
by the monotonicity properties \eqref{Mono3} and \eqref{Mono4}. $\frac{\partial}{\partial x}\mathcal{W}_{1,\rho}(y)$ vanishes only possible when $x=0$ or $\frac{1}{2}$ by \eqref{WTxF}. The proof is completed.

\end{proof}


We then have a similar monotonicity for $ \mathcal{W}_{2,\rho} (z)$.

\begin{proposition}\label{Trans2}
For $\rho\leq\frac{1}{\rho_*}=20$, there holds
$$
\frac{\partial}{\partial x}\mathcal{W}_{2,\rho}(z)\geq0
$$
for $\forall z\in\mathcal{R}_2$. The equality holds only possible when $x=0$ or $\frac{1}{2}$.

\end{proposition}

\begin{proof} The proof is similar to Proposition \ref{Trans1}.
Using \eqref{PPP3}, we see that
\begin{equation}\label{WWT} \aligned
\frac{\partial}{\partial x}\mathcal{W}_{2,\rho_\Gamma}(z)
&=\frac{\partial}{\partial x}\Big(
\sqrt{\frac{y}{2}}\sum_n e^{-\frac{1}{2}\pi yn^2}\vartheta(\frac{y}{2};n\frac{x+1}{2})+\rho \sqrt{\frac{y}{2}}\sum_n e^{-2\pi yn^2}\vartheta(\frac{y}{2};nx)\Big)\\
&=\sqrt{\frac{y}{2}}\sum_{n=1}^\infty ne^{-\frac{1}{2}\pi yn^2}\frac{\partial}{\partial Y}\vartheta(\frac{y}{2};Y)|_{Y=n\frac{x+1}{2}}
+2\rho\sqrt{\frac{y}{2}}\sum_{n=1}^\infty ne^{-2\pi yn^2}\frac{\partial}{\partial Y}\vartheta(\frac{y}{2};Y)|_{Y=nx}\\
&=\sqrt{\frac{y}{2}}e^{-\frac{1}{2}\pi y}\frac{\partial}{\partial Y}\vartheta(\frac{y}{2};Y)|_{Y=\frac{x+1}{2}}
+2\sqrt{\frac{y}{2}}e^{-2\pi y}\frac{\partial}{\partial Y}\vartheta(\frac{y}{2};Y)|_{Y={x+1}}\\
&\;+2\rho\sqrt{\frac{y}{2}}e^{-2\pi y}\frac{\partial}{\partial Y}\vartheta(\frac{y}{2};Y)|_{Y=x}\\
&\;+\sqrt{\frac{y}{2}}\sum_{n=3}^\infty ne^{-\frac{1}{2}\pi yn^2}\frac{\partial}{\partial Y}\vartheta(\frac{y}{2};Y)|_{Y=n\frac{x+1}{2}}
+2\rho\sqrt{\frac{y}{2}}\sum_{n=2}^\infty ne^{-2\pi yn^2}\frac{\partial}{\partial Y}\vartheta(\frac{y}{2};Y)|_{Y=nx}\\
& \: = \mathcal{W}_{2, x}^{a}(z)+\mathcal{W}_{2, x}^{b}(z)+\mathcal{W}_{2, x}^{c}(z)
 \endaligned\end{equation}
 where $\mathcal{W}_{2, x}^{a}(z),\mathcal{W}_{2, x}^{b}(z)$ and $\mathcal{W}_{2, x}^{c}(z)$ are defined at the last equality.

By Lemma \ref{LemmaTTT}, we also have
\begin{equation}\aligned\nonumber
\mathcal{W}_{2, x}^{a}(z)+\mathcal{W}_{2, x}^{b}(z)
&\geq\sqrt{\frac{y}{2}}e^{-\frac{1}{2}\pi y}\underline\vartheta(\frac{y}{2})\sin(\pi x)
-(2+2\rho)\sqrt{\frac{y}{2}}e^{-2\pi y}\overline\vartheta(\frac{y}{2})\sin(2\pi x).
 \endaligned\end{equation}

Since $|\sin(nx)|\leq n|\sin(x)|$ for any $x\in\mathcal{R}_2$,
again by Lemma \ref{LemmaTTT}, we see that
\begin{equation}\aligned\nonumber
\mathcal{W}_{2, x}^{c}(z)
&\geq-\frac{1}{2}\sqrt{\frac{y}{2}}\sum_{n=3}^\infty n^2e^{-\frac{1}{2}\pi yn^2}\overline\vartheta(\frac{y}{2})\sin(2\pi x)
-\rho\sqrt{\frac{y}{2}}\sum_{n=2}^\infty n^2e^{-2\pi yn^2}\overline\vartheta(\frac{y}{2})\sin(2\pi x).
 \endaligned\end{equation}
Plugging the above inequality into \eqref{WWT}, we get that
\begin{equation}\aligned\label{WWTx1}
\frac{\partial}{\partial x}\mathcal{W}_{2,\rho}(z)&\geq
\sqrt{\frac{y}{2}}e^{-\frac{1}{2}\pi y}\underline\vartheta(\frac{y}{2})\sin(\pi x)
-(2+2\rho+\sigma_3(y)+\rho\sigma_4(y))\sqrt{\frac{y}{2}}e^{-2\pi y}\overline\vartheta(\frac{y}{2})\sin(2\pi x)\\
&=\sqrt{\frac{y}{2}}e^{-\frac{1}{2}\pi y}\sin(\pi x)
\Big(
\underline\vartheta(\frac{y}{2})-(4+4\rho+2\sigma_3(y)+2\rho\sigma_4(y))\cos(\pi x) e^{-\frac{3}{2}\pi y}\overline\vartheta(\frac{y}{2})
\Big),
 \endaligned\end{equation}
where
$$
\sigma_3(y):=\frac{1}{2}\sum_{n=3}^\infty n^2e^{-\frac{1}{2}\pi y(n^2-4)},\;\;\; \sigma_4(y):=\sum_{n=2}^\infty n^2 e^{-2\pi y(n^2-1)}.
$$
$\sigma_3(y), \sigma_4(y)$ are functions with small size.  (In fact
$
\sigma_3(\frac{\sqrt 3}{2})\approx 5.00388\cdot 10^{-3},\; \sigma_4(\frac{\sqrt 3}{2})\approx 3.255011\cdot 10^{-7}.
$)

By the lower and upper bound estimates in Lemma \ref{LemmaTTT}, from \eqref{WWTx1} one deduces that
\begin{equation}\aligned\nonumber
\frac{\partial}{\partial x}\mathcal{W}_{2,\rho}(z)
&\geq
\sqrt{\frac{y}{2}}e^{-\frac{1}{2}\pi y}\sin(\pi x)
\Big(
4\pi(1-\mu(\frac{y}{2}))e^{-\frac{\pi y}{2}}\\
&\;-4\pi(4+4\rho+2\sigma_3(y)+2\rho\sigma_4(y))\cos(\pi x)e^{-2\pi y}(1+\mu(\frac{y}{2}))
\Big)\\
&\geq4\pi\sqrt{\frac{y}{2}}e^{-\pi y}\sin(\pi x)
\Big(
(1-\mu(\frac{y}{2}))\\
&\;-(4+4\rho+2\sigma_3(y)+2\rho\sigma_4(y))\cos(\pi x)e^{-\frac{3}{2}\pi y}(1+\mu(\frac{y}{2}))
\Big).
 \endaligned\end{equation}

Let
\begin{equation}\aligned\nonumber
\vartheta_{\mathcal{W}_{2,\rho}}(z):&=(1-\mu(\frac{y}{2}))
-(4+4\rho+2\sigma_3(y)+2\rho\sigma_4(y))\cos(\pi x)e^{-\frac{3}{2}\pi y}(1+\mu(\frac{y}{2})).
 \endaligned\end{equation}
Then
\begin{equation}\aligned\label{WWTxF}
\frac{\partial}{\partial x}\mathcal{W}_{2,\rho}(z)
&\geq
4\pi\sqrt{\frac{y}{2}}e^{-\pi y}\sin(\pi x)\cdot \vartheta_{\mathcal{W}_{2,\rho}}(y)
 \endaligned\end{equation}
It suffices to prove that
$$
\vartheta_{\mathcal{W}_{2,\rho}} (z)>0,\;\;\hbox{for}\;z\in\mathcal{R}_\Gamma,\rho\leq \frac{1}{\rho_{\Gamma}}=20.
$$
Now obviously
\begin{equation}\aligned\label{Mono1}
\frac{\partial}{\partial \rho}\vartheta_{\mathcal{W}_{2,\rho}}(y)<0; y>0,\hbox{ and} \;\frac{\partial}{\partial x}\vartheta_{\mathcal{W}_{2,\rho}} (z)>0; x\in[0,\frac{1}{2}],y>0.
\endaligned\end{equation}
Observe that the functions $\mu(y), \sigma_3, \sigma_4$ are decreasing on $y>0$. It follows that
\begin{equation}\aligned\label{Mono2}
\frac{\partial}{\partial y}\vartheta_{\mathcal{W}_{2,\rho}}(z)>0,\;y>0.
\endaligned\end{equation}

To complete the proof, we prove that $\vartheta_{\mathcal{W}_{2,\rho}}(z)$ is positive on the following three unbounded rectangular domains:

\begin{equation}\aligned\nonumber
\mathcal{R}_a=\{z\mid x\in[0,\frac{1}{4}], y\geq\frac{\sqrt{15}}{4}\};\;
\mathcal{R}_b=\{z\mid x\in[\frac{1}{4},\frac{3}{8}], y\geq\frac{\sqrt{55}}{8}\};\;
\mathcal{R}_c=\{z\mid x\in[\frac{3}{8},\frac{1}{2}], y\geq\frac{\sqrt{3}}{2}\}.
 \endaligned\end{equation}
It is clearly that
\begin{equation}\aligned\label{Rabc}
\mathcal{R}_\Gamma \subset \mathcal{R}_a\cup\mathcal{R}_b\cup\mathcal{R}_c.
\endaligned\end{equation}

A direct calculation gives
\begin{equation}\aligned\nonumber
&\vartheta_{\mathcal{W}_{2,\rho}}(z)|_{x=0,y=\frac{\sqrt {15}}{4},\rho=20}=0.0450964128\cdots>0\\
&\vartheta_{\mathcal{W}_{2,\rho}}(z)|_{x=\frac{1}{4},y=\frac{\sqrt{55}}{8},\rho=20}=0.1583739562\cdots>0\\
&\vartheta_{\mathcal{W}_{2,\rho}}(z)|_{x=\frac{3}{8},y=\frac{\sqrt{3}}{2},\rho=20}=0.3525036217\cdots>0.\\
\endaligned\end{equation}
This yields
\begin{equation}\aligned\nonumber
\vartheta_{\mathcal{W}_{2,\rho}}(z)>0,\;\hbox{for}\;z\in\mathcal{R}_a\cup\mathcal{R}_b\cup\mathcal{R}_c
 \endaligned\end{equation}
by the monotonicity properties \eqref{Mono1} and \eqref{Mono2}.
Therefore by \eqref{Rabc}
\begin{equation}\aligned\nonumber
\vartheta_{\mathcal{W}_{2,\rho}}(z)>0,\;\hbox{for}\;z\in\mathcal{R}_2.
 \endaligned\end{equation}
By \eqref{WWTxF} $\frac{\partial}{\partial x}\mathcal{W}_{2,\rho_{*}}(z)$ vanishes only  at $x=0$ or $\frac{1}{2}$. This completes the proof.

\end{proof}


\section{The behavior of $\mathcal{W}_{1,\rho}(z)$ on the $y-$axis }

\setcounter{equation}{0}
In this section, we study the property of the functional $ \mathcal{W}_{1, \rho}$ on the $y-$axis. We will prove that on the $y-$axis, depending on $\rho$, $\mathcal{W}_{1,\rho}(z)$ has either 1 or 3 critical points. This gives the precise characterization of  the minimizers of $\mathcal{W}_{1,\rho}(z)$ on the $y-$axis.
The proof relies crucially on a novel property of Jacob theta function proved in Theorem \ref{QXY} below.

\begin{proposition}\label{YYY1} There exists a threshold $\rho_{1}$ which is the unique solution of
$
\frac{\partial^2}{\partial y^2}\mathcal{W}_{1,\rho}(yi)\mid_{y=1}=0
$, (in fact, $\rho_1=-\frac{\mathcal{Y}''(1)}{\mathcal{X}''(1)} \sim 0.04016680351\cdots$),  such that

\noindent
1. if $\rho\in[\rho_{1},+\infty)$, the function $y\rightarrow\mathcal{W}_{1,\rho}(yi), y>0$ admits only one critical point
at $y=1$, and $\frac{\partial}{\partial y}\mathcal{W}_{1,\rho}(yi)<0$ if $ y\in(0,1)$ and $\frac{\partial}{\partial y}\mathcal{W}_{1,\rho}(yi)>0$ if $y\in(1,\infty)$;

\noindent
2. if $\rho\in[0,\rho_1)$, the function $y\rightarrow\mathcal{W}_{1,\rho}(yi), y>0$ admits only three critical points
at $y_{1,\rho}$, 1 and $\frac{1}{y_{1,\rho}}$, where $y_{1,\rho}\in(1,\sqrt3]$. Moreover
 \begin{equation}\aligned\nonumber
 \frac{\partial}{\partial y}\mathcal{W}_{1,\rho}(yi)&<0\;\hbox{if}\; y\in(0,\frac{1}{y_{1,\rho}}),\\
 \frac{\partial}{\partial y}\mathcal{W}_{1,\rho}(yi)&>0\;\hbox{if}\; y\in(\frac{1}{y_{1,\rho}},1),\\
\frac{\partial}{\partial y}\mathcal{W}_{1,\rho}(yi)&<0\;\hbox{if}\; y\in(1,y_{1,\rho}),\\
\frac{\partial}{\partial y}\mathcal{W}_{1,\rho}(yi)&>0\;\hbox{if}\; y\in(y_{1,\rho},\infty).
 \endaligned\end{equation}
The critical point $y_{1,\rho}$ is the unique solution of
$
\frac{\partial}{\partial y}\mathcal{W}_{1,\rho}(yi)=0,\;y\in(1,\sqrt3].
$

Furthermore if $\rho\in[0,\rho_1]$, then
\begin{equation}\aligned\label{MXY}
\frac{\partial y_{1, \rho}}{\partial \rho} <0.
 \endaligned\end{equation}

\end{proposition}

\vskip0.1in

To prove Proposition \ref{YYY1},  we need to use some properties of the Jacobi theta functions defined at \eqref{XYAB}-\eqref{XYAB2}. They satisfy the transformation property
\begin{equation}\aligned\label{Theta}
\vartheta_3(\frac{1}{y})&=\sqrt{y}\vartheta_3(y),\;\vartheta_2(\frac{1}{y})=\sqrt y\vartheta_4(y)\\
\vartheta_4(\frac{1}{y})&=\sqrt y\vartheta_2(y),\;\;\vartheta_4(y)=\vartheta_3(4y)-\vartheta_2(4y).
 \endaligned\end{equation}

It is easy to see that for $z=yi$
 \begin{equation}\aligned
\theta(s;yi)=\sum_{m}\sum_{n}e^{-s\frac{\pi}{y}(n^2+m^2y^2)},\; \theta(s;\frac{yi+1}{2})=\sum_{m}\sum_{n}e^{-s\frac{\pi}{y}((\frac{m}{2}+n)^2+\frac{m^2}{4}y^2)}.
 \endaligned\end{equation}

We first express $\theta(s;yi), \theta(s;\frac{yi+1}{2})$ as products of Jacobi theta functions, which is a  starting point
of our analysis.
\begin{lemma}\label{LemmaDP}
It holds that

\begin{equation}\aligned\nonumber
\theta(s;yi)=\vartheta_3(s y)\vartheta_3(\frac{y}{s}),\;\theta(s;\frac{yi+1}{2})=\vartheta_3(s y)\vartheta_3(\frac{y}{s})+\vartheta_2(s y)\vartheta_2(\frac{s}{y}).
 \endaligned\end{equation}

\end{lemma}

\begin{proof} The first one is straightforward:
\begin{equation}\aligned\nonumber
\theta(s;yi)=\sum_n e^{-s\frac{\pi}{y}n^2}\sum_me^{-s\pi y m^2}=\vartheta_3(s y)\vartheta_3(\frac{s}{y}).
 \endaligned\end{equation}
For the second one,
\begin{equation}\aligned\nonumber
\theta(s;\frac{yi+1}{2})&=\sum_{m}\sum_{n}e^{-s\frac{\pi}{y}((\frac{m}{2}+n)^2+\frac{m^2}{4}y^2)}=\sum_{p\equiv q(\mod 2)} e^{-\frac{s \pi}{4}(\frac{1}{y}p^2+y q^2)}\\
&=\sum_{p=2m',q=2n'} e^{-\frac{s \pi}{4}(\frac{1}{y}p^2+y q^2)}
+\sum_{p=2m'+1,q=2n'+1} e^{-\frac{s \pi}{4}(\frac{1}{y}p^2+y q^2)}\\
&=\sum_{m'}e^{-s \pi\frac{1}{y}m'^2}\sum_{n'}e^{-s \pi y n'^2}
+\sum_{m'} e^{-\frac{s \pi}{4}\frac{1}{y}(2m'+1)^2}\sum_{n'} e^{-\frac{s \pi}{4}y (2n'+1)^2}\\
&=\vartheta_3(s y)\vartheta_3(\frac{s}{y})+\vartheta_2(s y)\vartheta_2(\frac{s}{y}).
 \endaligned\end{equation}

\end{proof}

The following Lemma follows from Lemma \ref{G111}. We single it out for the convenience of our analysis here.
\begin{lemma}\label{Lemmaf01} For any $s>0$, $\theta(s;yi)$ and $\theta(s;\frac{yi+1}{2})$
both satisfy the functional equation
\begin{equation}\aligned\label{Heq}
\mathcal{H}(\frac{1}{y})=\mathcal{H}(y).
 \endaligned\end{equation}
Consequently, $\mathcal{H}'(\frac{1}{y})=-y^2\mathcal{H}'(y)$. In particular, $\mathcal{H}'(1)=0$, that is, $y=1$ is always a critical point of
$\theta(s;yi), \theta(s;\frac{yi+1}{2})$.

\end{lemma}

For $s=1$, by Lemma \ref{LemmaDP} and transformation \eqref{Theta}, we  obtain that
\begin{lemma}
\begin{equation}\aligned\label{LemmaE12}
\theta(1;yi)=\sqrt{y}\vartheta_3^2(y), \;\;
\theta(2;\frac{yi+1}{2})&=\frac{\sqrt y}{2}\Big(
\vartheta_3(4y)\vartheta_3(\frac{y}{4})+\vartheta_2(4y)\vartheta_4(\frac{y}{4})
\Big).
 \endaligned\end{equation}

\end{lemma}

To prove Proposition \ref{YYY1}, we first prove  a monotonicity property  of $\theta (1; yi)$ and $\theta (2; \frac{yi+1}{2})$ in Lemma \ref{E111}, which can be viewed as the particular case of Proposition \ref{YYY1}. Then we establish the key Theorem \ref{QXY}, in which a novel property about the quotient of Jacobi theta functions is proved.

The following Lemma is known in \cite{BP,Mon1988}.
\begin{lemma}\label{E111}

\begin{itemize}
  \item The function $y\rightarrow\theta(s;yi), y>0$, has only one critical point at $y=1$. Furthermore
\begin{equation}\aligned \nonumber
\frac{\partial}{\partial y}
\theta(s;yi)<0\;\hbox{for}\; y\in(0,1);\;\frac{\partial}{\partial y}
\theta(s;yi)>0\;\hbox{for}\; y\in(1,\infty).
 \endaligned\end{equation}
  \item  For any $s>0$, the function $y\rightarrow\theta(s;\frac{yi+1}{2}), y>0$, has three critical points at $\frac{\sqrt3}{3}$, 1 and $\sqrt3$.
\end{itemize}

\end{lemma}



We now state Theorem \ref{QXY} whose proof is much involved. We use a combination of functional equations, error terms analysis and several new observations.
Let
\begin{equation}\aligned\nonumber
\mathcal{X}(y):=\vartheta_3(y)\vartheta_3(\frac{1}{y})=\sqrt{y}\vartheta_3^2(y)
,
\mathcal{Y}(y):=2\big(
\vartheta_3(4y)\vartheta_3(\frac{4}{y})+\vartheta_2(4y)\vartheta_2(\frac{4}{y})
\big)=\sqrt y\big(
\vartheta_3(4y)\vartheta_3(\frac{y}{4})+\vartheta_2(4y)\vartheta_4(\frac{y}{4})
\big).
\endaligned\end{equation}
\begin{theorem}\label{QXY}
 The function $y\mapsto \frac{\mathcal{Y}'(y)}{\mathcal{X}'(y)}, y>0$ has only one critical point at $y=1$. Furthermore
$
\Big(\frac{\mathcal{Y}'(y)}{\mathcal{X}'(y)}\Big)'<0 $ for  $y \in(0,1)$ and $ \ \Big(\frac{\mathcal{Y}'(y)}{\mathcal{X}'(y)}\Big)'>0$ for $y \in(1,\infty)$.

\end{theorem}

\begin{proof} Denote $
\mathcal{Z}(y):=\frac{\mathcal{Y}'(y)}{\mathcal{X}'(y)}.$ By Lemma \ref{E111}, the function $\mathcal{Z}(y)$  is well-defined.
By Lemma \ref{Lemmaf01}, we also have
\begin{equation}\aligned\label{2ABD}
\mathcal{X}'(\frac{1}{y})=-y^2\mathcal{X}'(y),\;\mathcal{Y}'(\frac{1}{y})=-y^2\mathcal{Y}'(y).
\endaligned\end{equation}
Hence
\begin{equation}\aligned\nonumber
\mathcal{Z}(\frac{1}{y})=\mathcal{Z}(y),
 \endaligned\end{equation}
and
\begin{equation}\aligned\label{ZZZ}
\mathcal{Z}'(\frac{1}{y})=-y^2\mathcal{Z}'(y).
 \endaligned\end{equation}
Consequently,
$
\mathcal{Z}'(1)=0,
$
i.e., $y=1$ is the critical point of $\mathcal{Z}(y)$.

 By \eqref{ZZZ}, it suffices to prove that
 \begin{equation}\aligned\label{Z000}
\mathcal{Z}'(y)>0,\;\;\hbox{for}\;y\in(1,\infty).
 \endaligned\end{equation}
By the explicit expression of Jacobi theta functions \eqref{XYAB2} and \eqref{Theta}, we start with
\begin{equation}\aligned\nonumber
\mathcal{X}(y)&=\sqrt y(1+2\sum_{n=1}^\infty e^{-\pi n^2 y})^2\\
&=\Big(\sqrt y+4\sqrt y e^{-\pi y}+4\sqrt y e^{-2\pi y}+4\sqrt y e^{-4\pi y}\Big)\\
&+\Big(4\sqrt y\sum_{n=3}^\infty e^{-\pi n^2 y}+4\sqrt y\big(\sum_{n=1}^\infty e^{-\pi n^2 y}\big)^2
+8\sqrt y\sum_{n=1}^\infty e^{-\pi(n^2+1)y}\Big)\\
:&= \mathcal{X}_a (y)+ \mathcal{X}_e (y)
 \endaligned\end{equation}
 where $\mathcal{X}_a (y)$ and $\mathcal{X}_e (y)$ are defined at the last equality. $\mathcal{X}_a$ is the major part and $\mathcal{X}_e$ is the error part. In fact, we have that for some constant $C>0$
\begin{equation}\aligned\label{Xae}
\|\mathcal{X}_a(y)\|_{C^2}\leq C \sqrt y e^{-5\pi y},\;\hbox{for}\; y>1.
 \endaligned\end{equation}
For $\mathcal{Y}(y)$, again by \eqref{XYAB2} and \eqref{Theta}, one first has
\begin{equation}\aligned\nonumber
\sqrt y\vartheta_3(4y)\vartheta_3(\frac{y}{4})
&=\sqrt y(1+2\sum_{n=1}^\infty e^{-4\pi n^2 y})(1+2\sum_{n=1}^\infty e^{-\frac{1}{4}\pi n^2 y})\\
&=\sqrt y+2\sqrt y e^{-\frac{1}{4}\pi y}+2\sqrt y e^{-\pi y}+2\sqrt y e^{-\frac{9}{4}\pi y}+4\sqrt y e^{-4\pi y}\\
&+2\sqrt y \sum_{n=2}^\infty e^{-4\pi n^2 y}+2\sqrt y\sum_{n=5}^\infty e^{-\frac{1}{4}\pi n^2 y}
+4\sqrt y\sum_{n=1}^\infty e^{-4\pi n^2 y}\sum_{n=1}^\infty e^{-\frac{1}{4}\pi n^2 y}
 \endaligned\end{equation}
We regroup the terms as
\begin{equation}\aligned\nonumber
\sqrt y\vartheta_2(4y)\vartheta_4(\frac{y}{4})&=\sqrt\vartheta_2(4y)\big(
\vartheta_3(y)-\vartheta_2(y)
\big)=\sqrt y\vartheta_2(4y)\vartheta_3(y)-\sqrt y\vartheta_2(4y)\vartheta_2(y)\\
&=2\sqrt y\sum_{n=1}^\infty e^{-\pi (2n-1)^2y}+4\sqrt y\sum_{n=1}^\infty e^{-\pi(2n-1)^2y}\sum_{n=1}^\infty e^{-\pi n^2 y}\\
&-4\sqrt y e^{-\frac{5}{4}\pi y}(1+\sum_{n=2}^\infty e^{-\pi((n-\frac{1}{2})^2-\frac{1}{4})y})(1+\sum_{n=2}^\infty e^{-\pi((2n-1)^2-1)y})\\
&=2\sqrt y e^{-\pi y}+4\sqrt y e^{-2\pi y}-4\sqrt y e^{-\frac{5}{4}\pi y}-4\sqrt y e^{-\frac{13}{4}\pi y}\\
&+4\sqrt y\big(
\sum_{n=2}^\infty e^{-\pi((2n-1)^2+1)y}+\sum_{n=2}^\infty e^{-\pi (n^2+1)y}+\sum_{n=2}^\infty e^{-\pi(2n-1)^2 y}\sum_{n=2}^\infty e^{-\pi n^2 y}
\big)\\
&-4\sqrt y e^{-\frac{5}{4}\pi y}\big(
\sum_{n=3}^\infty e^{-\pi((n-\frac{1}{2})^2-\frac{1}{4})y}+\sum_{n=2}^\infty e^{-\pi((2n-1)^2-1)y}\\
&+
\sum_{n=2}^\infty e^{-\pi((n-\frac{1}{2})^2-\frac{1}{4})y}\cdot\sum_{n=2}^\infty e^{-\pi((2n-1)^2-1)y}
\big).
 \endaligned\end{equation}

Now let the approximate part of $\mathcal{Y} (y)$ be
\begin{equation}\aligned\nonumber
\mathcal{Y}_a(y):=\sqrt y+2\sqrt y e^{-\frac{1}{4}\pi y}+4\sqrt y e^{-\pi y}+2\sqrt y e^{-\frac{9}{4}\pi y}+4\sqrt y e^{-2\pi y}
+4\sqrt y e^{-4\pi y}-4\sqrt y e^{-\frac{5}{4}\pi y}-4\sqrt y e^{-\frac{13}{4}\pi y}
 \endaligned\end{equation}
and the error part by
\begin{equation}\aligned\nonumber
\mathcal{Y}_e(y):=&2\sqrt y \sum_{n=2}^\infty e^{-4\pi n^2 y}+2\sqrt y\sum_{n=5}^\infty e^{-\frac{1}{4}\pi n^2 y}
+4\sqrt y\sum_{n=1}^\infty e^{-4\pi n^2 y}\sum_{n=1}^\infty e^{-\frac{1}{4}\pi n^2 y}\\
&+4\sqrt y\big(
\sum_{n=2}^\infty e^{-\pi((2n-1)^2+1)y}+\sum_{n=2}^\infty e^{-\pi (n^2+1)y}+\sum_{n=2}^\infty e^{-\pi(2n-1)^2 y}\sum_{n=2}^\infty e^{-\pi n^2 y}
\big)\\
&-4\sqrt y e^{-\frac{5}{4}\pi y}\big(
\sum_{n=3}^\infty e^{-\pi((n-\frac{1}{2})^2-\frac{1}{4})y}+\sum_{n=2}^\infty e^{-\pi((2n-1)^2-1)y}
+
\sum_{n=2}^\infty e^{-\pi((n-\frac{1}{2})^2-\frac{1}{4})y}\cdot\sum_{n=2}^\infty e^{-\pi((2n-1)^2-1)y}
\big).
 \endaligned\end{equation}
Then
\begin{equation}\aligned\label{Yae}
\mathcal{Y}(y)=\mathcal{Y}_a(y)+\mathcal{Y}_e(y)
 \endaligned\end{equation}
and we have  following estimate for $\mathcal{Y}_e(y)$:
\begin{equation}\aligned\nonumber
\|\mathcal{Y}_e(y)\|_{C^2}\leq C\sqrt y e^{-\frac{17}{4}\pi y}.
 \endaligned\end{equation}

To prove \eqref{Z000}, we divide the proof into two regions of $y$:  the large $y$ case $ y\in[1.1,\infty)$ and the small $y$ case $ y\in(1,1.1).$

\vskip0.1in

\noindent
{\bf Case (a): $y\in [1.1,\infty)$}. In this case  we have
$$
\mathcal{Z}'(y)=\frac{\mathcal{Y}''(y)\mathcal{X}'(y)-\mathcal{X}''(y)\mathcal{Y}'(y)}{(\mathcal{X}'(y))^2}.
$$
By Lemma \ref{E111}, to prove Case $(a)$ it suffices to prove that
\begin{equation}\aligned\nonumber
\mathcal{Y}''(y)\mathcal{X}'(y)-\mathcal{X}''(y)\mathcal{Y}'(y)>0\;\;\hbox{if}\;\;y\in(1.1,\infty).
 \endaligned\end{equation}
By \eqref{Xae} and \eqref{Yae}, there holds
\begin{equation}\aligned\nonumber
\mathcal{Y}''\mathcal{X}'-\mathcal{Y}''\mathcal{X}'
=\Big(\mathcal{Y}_a''\mathcal{X}_a'-\mathcal{X}_a''\mathcal{Y}_a'\Big)
+\Big(\mathcal{Y}_e''\mathcal{X}'-\mathcal{Y}_e' \mathcal{X}''+\mathcal{Y}_a''\mathcal{X}_e'-\mathcal{X}_e''\mathcal{Y}_a'\Big)
 \endaligned\end{equation}
 where $\Big(\mathcal{Y}_a''\mathcal{X}_a'-\mathcal{X}_a''\mathcal{Y}_a'\Big)$ and $\Big(\mathcal{Y}_e''\mathcal{X}'-\mathcal{Y}_e' \mathcal{X}''+\mathcal{Y}_a''\mathcal{X}_e'-\mathcal{X}_e''\mathcal{Y}_a'\Big)$ are the approximate part and the error part of $\mathcal{Y}''\mathcal{X}'-\mathcal{Y}''\mathcal{X}'$ respectively. We shall use the approximate part to control the error part.

To obtain the lower bound of $\Big(\mathcal{Y}_a''\mathcal{X}_a'-\mathcal{X}_a''\mathcal{Y}_a'\Big)$, after subtracting some proper factors, one finds
 \begin{equation}\aligned\label{2Estimate1}
y\rightarrow\frac{16y}{\pi}e^{\frac{1}{4}\pi y}\Big(\mathcal{Y}_a''\mathcal{X}_a'-\mathcal{X}_a''\mathcal{Y}_a'\Big)(y)
\endaligned\end{equation}
is monotonically increasing.

For the error part $\Big(\mathcal{Y}_e''\mathcal{X}'-\mathcal{Y}_e' \mathcal{X}''+\mathcal{Y}_a''\mathcal{X}_e'-\mathcal{X}_e''\mathcal{Y}_a'\Big)$,
one has the estimate
\begin{equation}\aligned\label{2Estimate2}
|\Big(\mathcal{Y}_e''\mathcal{X}'-\mathcal{Y}_e' \mathcal{X}''+\mathcal{Y}_a''\mathcal{X}_e'-\mathcal{X}_e''\mathcal{Y}_a'\Big)(y)|\leq C\sqrt y e^{-\frac{17}{4}\pi y},
 \endaligned\end{equation}
which decays to zero very fast.

Combining \eqref{2Estimate1} with \eqref{2Estimate2}, one deduces that
\begin{equation}\aligned\label{2AB105}
\mathcal{Y}''\mathcal{X}'-\mathcal{X}''\mathcal{Y}'>\;\;\hbox{if}\;\;y\in[1.1,\infty).
 \endaligned\end{equation}

The detailed proof of \eqref{2Estimate1}, \eqref{2Estimate2} and \eqref{2AB105} will be provided in the Appendix 2.

This proves that
\begin{equation}\aligned\label{2Cpk000}
\mathcal{Z}'(y)>0\;\;\hbox{if}\;\;y\in[1.1,\infty).
\endaligned\end{equation}

\vskip0.1in

\noindent
{\bf Case (b): $y \in (1,1.1)$.}  In this case $ 0<1-y<0.1$. To prove
\begin{equation}\aligned\label{2AB1k}
\mathcal{Z}'(y)=\Big(\frac{\mathcal{Y}'(y)}{\mathcal{X}'(y)}\Big)'>0, \hbox{on}\; y\in(1,1.1),
\endaligned\end{equation}
it suffices to prove that
$$
\Big(\frac{\mathcal{Y}''(y)}{\mathcal{X}''(y)}\Big)'>0, \hbox{on}\; y\in(1,1.1),
$$
given that
\begin{equation}\aligned\label{2AB10}
\mathcal{X}'(1)=\mathcal{Y}'(1)=0
\endaligned\end{equation}
which follows from \eqref{2ABD}. In fact, there exists $y_1 \in(1,y)$ such that
\begin{equation}\aligned\label{XY1}
\Big(\frac{\mathcal{Y}'(y)}{\mathcal{X}'(y)}\Big)'&=\frac{\mathcal{Y}''(y)\mathcal{X}'(y)-\mathcal{Y}'(y)\mathcal{X}''(y)}{\mathcal{X}'^2(y)}=\frac{\mathcal{X}''(y)}{\mathcal{X}'(y)}\Big(
\frac{\mathcal{Y}''(y)}{\mathcal{X}''(y)}-\frac{\mathcal{Y}'(y)}{\mathcal{X}'(y)}
\Big)\\
&=\frac{\mathcal{X}''(y)}{\mathcal{X}'(y)-\mathcal{X}'(1)}\Big(
\frac{\mathcal{Y}''(y)}{\mathcal{X}''(y)}-\frac{\mathcal{Y}'(y)-\mathcal{Y}'(1)}{\mathcal{X}'(y)-\mathcal{X}'(1)}
\Big)\\
&=\frac{\mathcal{X}''(y)}{\mathcal{X}''(y_2)(y-1)}\Big(
\frac{\mathcal{Y}''(y)}{\mathcal{X}''(y)}-\frac{\mathcal{Y}''(y_1)}{\mathcal{X}''(y_1)}
\Big)
 \endaligned\end{equation}
using \eqref{2AB10}.

We also have that

\begin{equation}\aligned\label{2L1}
\mathcal{X}''(y)>0,\;\;\hbox{if}\;y\in(1,\infty)
\endaligned\end{equation}
by the same decomposition method as used above. We omit the details here. (Actually, we only need \eqref{2L1} holds for small interval such as $(1,1.2]$).

Moreover, $\Big(\frac{\mathcal{Y}''(y)}{\mathcal{X}''(y)}\Big)'>0$ implies
\begin{equation}\aligned\label{2L2}
\frac{\mathcal{Y}''(y)}{\mathcal{X}''(y)}-\frac{\mathcal{Y}''(y_1)}{\mathcal{X}''(y_1)}>0.
\endaligned\end{equation}
Then the claim follows from \eqref{2L2}, \eqref{2L1} and \eqref{XY1}.

For the derivative of the quotient of second order derivatives, one has
\begin{equation}\aligned\nonumber
\Big(\frac{\mathcal{Y}''(y)}{\mathcal{X}''(y)}\Big)'
=\frac{\mathcal{Y}'''(y)\mathcal{X}''(y)-\mathcal{Y}''(y)\mathcal{X}'''(y)}{\mathcal{X}''^2(y)}.
\endaligned\end{equation}
Define
\begin{equation}\aligned\nonumber
f_{\mathcal{XY}}(y):=\mathcal{Y}'''(y)\mathcal{X}''(y)-\mathcal{Y}''(y)\mathcal{X}'''(y).
\endaligned\end{equation}

Equivalently, to show \eqref{2AB1k} one needs to show that
\begin{equation}\aligned\label{2ABf}
f_{\mathcal{XY}}(y)>0\;\;\hbox{for}\;\;y\in(1,1.1).
\endaligned\end{equation}
Differrentiating\eqref{2ABD}, the functions $\mathcal{X}(y)$ and $\mathcal{Y}(y)$ both satisfy the following functional equations
\begin{equation}\aligned\label{2ABH}
\mathcal{H}''(\frac{1}{y})&=2y^3\mathcal{H}'(y)+y^4\mathcal{H}''(y)\\
\mathcal{H}'''(\frac{1}{y})&=-6y^4\mathcal{H}'(y)-6y^5\mathcal{H}''(y)-y^6\mathcal{H}'''(y).
\endaligned\end{equation}
Plugging $y=1$ in \eqref{2ABH} and using \eqref{2AB10}, one deduces
\begin{equation}\aligned\label{2AB31}
\mathcal{X}'''(1)=-3\mathcal{X}''(1),\;\;\mathcal{Y}'''(1)=-3\mathcal{Y}''(1).
\endaligned\end{equation}
From \eqref{2AB31}, one has
\begin{equation}\aligned\label{2ABf000}
f_{\mathcal{XY}}(1)=0.
\endaligned\end{equation}

Then to prove \eqref{2ABf}, by \eqref{2ABf000}, it suffices to prove that
\begin{equation}\aligned\label{2ABfp}
f_{\mathcal{XY}}'(y)>0\;\;\hbox{for}\;\;y\in(1,1.1).
\endaligned\end{equation}
Proceed by \eqref{Xae} and \eqref{Yae}
\begin{equation}\aligned\label{2fAB12}
f_{\mathcal{XY}}'&=\mathcal{Y}''''\mathcal{X}''-\mathcal{Y}''\mathcal{X}''''\\
&=\Big(\mathcal{Y}_a''''\mathcal{X}_a''-\mathcal{Y}_a''\mathcal{X}_a''''\Big)
+\Big(\mathcal{X}_e''\mathcal{Y}''''+\mathcal{Y}_e''''\mathcal{X}_a''-\mathcal{X}_e''''\mathcal{Y}''-\mathcal{Y}_e''\mathcal{X}_a''''\Big).
\endaligned\end{equation}
 We use $\Big(\mathcal{Y}_a''''\mathcal{X}_a''-\mathcal{Y}_a''\mathcal{X}_a''''\Big)$ and $\Big(\mathcal{X}_e''\mathcal{Y}''''+\mathcal{Y}_e''''\mathcal{X}_a''-\mathcal{X}_e''''\mathcal{Y}''-\mathcal{Y}_e''\mathcal{X}_a''''\Big)$
as the approximate and error parts of $f_{\mathcal{XY}}'$ respectively.

For the approximate part, after subtracting some proper factor, one finds
\begin{equation}\aligned\label{2fM1}
y\rightarrow \frac{512y^4}{\pi}e^{\frac{1}{4}\pi y}\Big(\mathcal{Y}_a''''\mathcal{X}_a''-\mathcal{Y}_a''\mathcal{X}_a''''\Big)(y)
\endaligned\end{equation}
is monotonically decreasing on $(1,1.2)$.

For the error part, one has the following estimate
\begin{equation}\aligned\label{2fM2}
|\Big(\mathcal{X}_e''\mathcal{Y}''''+\mathcal{Y}_e''''\mathcal{X}_a''-\mathcal{X}_e''''\mathcal{Y}''-\mathcal{Y}_e''\mathcal{X}_a''''\Big)(y)|
\leq C y e^{-5\pi y},
\endaligned\end{equation}
which has fast decay.

Combining \eqref{2fM1}, \eqref{2fM2} and \eqref{2fAB12}, we can prove that
\begin{equation}\aligned\label{fXYF}
f_{\mathcal{XY}}'(y)>0\;\;\hbox{if}\;\;y\in(1,1.11].
\endaligned\end{equation}

The detailed proof of \eqref{2fM1}, \eqref{2fM2} and \eqref{fXYF} will be given in the Appendix 2.

This completes the proof.

\end{proof}




Finally we give the proof of Proposition \ref{YYY1}.

\begin{proof} By Lemma \ref{Lemmaf01}, $y=1$ is a critical point of $\mathcal{W}_{1,\rho}(yi)$. Furthermore
\begin{equation}\aligned\label{H111}
\frac{\partial}{\partial y}\mathcal{W}_{1,\rho}(\frac{1}{y}i)=-y^2\frac{\partial}{\partial y}\mathcal{W}_{1,\rho}(yi)(y).
 \endaligned\end{equation}
By Lemma \ref{E111}, we have
\begin{equation}\aligned\label{R000}
\mathcal{X}'(y)>0\;\;\hbox{if}\;\; y\in(1,\infty)\;\;\;\hbox{and}\;\;\;\mathcal{Y}'(\sqrt3)=0.
 \endaligned\end{equation}

Hence we obtain that
\begin{equation}\aligned\label{Ysqrt3}
\frac{\partial}{\partial y}\mathcal{W}_{1,\rho}(yi)>0\;\;\hbox{if}\;\;y\in(\sqrt3,\infty).
 \endaligned\end{equation}

To study the monotonicity of $\mathcal{W}_{1,\rho}(yi)$ on the interval $(1,\sqrt3)$, we rewrite $\frac{\partial}{\partial y}\mathcal{W}_{1,\rho}(yi)$ as
\begin{equation}\aligned\label{W1yyy}
\frac{\partial}{\partial y}\mathcal{W}_{1,\rho}(yi)&=\frac{\partial}{\partial y}\Big(\theta(2;\frac{yi+1}{2})+\rho\theta(1;yi)\Big)=\mathcal{Y}'(y)+\rho\mathcal{X}'(y)\\
&=\mathcal{X}'(y)\cdot\Big(
\frac{\mathcal{Y}'(y)}{\mathcal{X}'(y)}+\rho
\Big).
 \endaligned\end{equation}
 By \eqref{R000},  the zeroes of $\frac{\partial}{\partial y}\mathcal{W}_{1,\rho}(yi)$ on $(1,\sqrt3)$ satisfy the following functional equation
\begin{equation}\aligned\label{XYR}
\frac{\mathcal{Y}'(y)}{\mathcal{X}'(y)}+\rho
=0, \;\; y\in(1,\sqrt3).
 \endaligned\end{equation}

Furthermore, by Theorem \ref{QXY}, we see that
\begin{equation}\aligned\label{QXY000}
\frac{\mathcal{Y}'(y)}{\mathcal{X}'(y)}+\rho\;\;\hbox{is strictly decreasing on}\;\;(1,\sqrt3).
\endaligned\end{equation}
\eqref{QXY000} and \eqref{XYR} imply that $\frac{\partial}{\partial y}\mathcal{W}_{1,\rho}(yi)$ admits at most one zero point on $(1,\sqrt3)$.
This  fact combined with \eqref{Ysqrt3} yields that $
\frac{\partial}{\partial y}\mathcal{W}_{1,\rho}(yi)\;\;\hbox{admits either one or three critical points on}\;\;(0,\infty)$.

Since $\mathcal{X}'(1)=\mathcal{Y}'(1)=0$, $\frac{\mathcal{Y}'(1)}{\mathcal{X}'(1)}=\frac{\mathcal{Y}''(1)}{\mathcal{X}''(1)}$.

At the other end point $\sqrt3$, since $\mathcal{Y}'(\sqrt3)=0$ (see \eqref{R000}), we have that
\begin{equation}\aligned\nonumber
\frac{\mathcal{Y}'(\sqrt3)}{\mathcal{X}'(\sqrt3)}+\rho=0+\rho>0,\;\;\rho>0.
\endaligned\end{equation}

By \eqref{QXY000}, we see that the equation \eqref{XYR} has a zero point if and only if
\begin{equation}\aligned\label{XY6}
\frac{\mathcal{Y}''(1)}{\mathcal{X}''(1)}+\rho<0.
\endaligned\end{equation}

The condition in \eqref{XY6} is
\begin{equation}\aligned\label{XY7}
\rho<\rho_{1}:=-\frac{\mathcal{Y}''(1)}{\mathcal{X}''(1)}.
\endaligned\end{equation}
Combining \eqref{XY6},\eqref{XY7} with \eqref{Ysqrt3}, one has
\begin{equation}\aligned\nonumber
\frac{\partial}{\partial y}\mathcal{W}_{1,\rho}(yi)>0\;\;\hbox{on}\;\;(1,\infty)\;\;\hbox{provided}\;\;\rho\geq\rho_{1}.
\endaligned\end{equation}
This and \eqref{H111} give the proof of part  1 of Proposition \ref{YYY1}. (For the case $\rho=0$, $y_{1,\rho}=\sqrt3$ by \eqref{R000}.)

In the case when $\rho\in(0,\rho_1)$, there exists  unique  root of \eqref{XYR} as $y_{1,\rho}\in(1,\sqrt3)$. By duality \eqref{H111}, there exists another
root $\frac{1}{y_{1,\rho}}\in(\frac{\sqrt3}{3},1)$. So part 2 of Proposition \ref{YYY1}  follows from  \eqref{H111} and \eqref{QXY000}.

Finally \eqref{MXY} follows from \eqref{QXY000}.

This completes the proof.

\end{proof}

\section{The behavior of $\mathcal{W}_{2,\rho}(z)$ on the $y-$axis }

\setcounter{equation}{0}

Let $
\mathcal{W}_{2,\rho}(z):=\theta(1;\frac{z+1}{2})+\rho \theta(2;z)
$ be the conjugate of $\mathcal{W}_{1, \rho} (z)$. In this section we prove similar properties of Section 6 for $\mathcal{W}_{2, \rho}$.
 As in Section 6, $\mathcal{W}_{2,\rho}(yi)$ admits either 1 or 3 three critical points depending on different vales of $\rho$.  These are stated  in Proposition \ref{YYY2}. The proof relies critically on a novel property of the classical theta functions proved in Theorem \ref{QAB}.

\begin{proposition}\label{YYY2} There exists a threshold $\rho_{2}$ which is the unique solution of
$$
\frac{\partial^2}{\partial y^2}\mathcal{W}_{2,\rho}(yi)\mid_{y=1}=0
$$
 (in fact $\rho_2=-1-\frac{\mathcal{B}''(1)}{\mathcal{A}''(1)}$, numerically, $\rho_{2}=1.190861337\cdots$) such that

\noindent
1. when $\rho\in[0,\rho_2)$, the function $y\rightarrow\mathcal{W}_{2,\rho}(yi), y>0$ admits only three critical points
at $y_{2,\rho}$, 1 and $\frac{1}{y_{2,\rho}}$, where $y_{2,\rho}\in(1,\sqrt3]$. Furthermore we have $ \frac{\partial}{\partial y}\mathcal{W}_{2,\rho}(yi)<0\;\hbox{if}\; y\in(0,\frac{1}{y_{2,\rho}})$,  $\frac{\partial}{\partial y}\mathcal{W}_{2,\rho}(yi)>0\;\hbox{if}\; y\in(\frac{1}{y_{2,\rho}},1)$, $ \frac{\partial}{\partial y}\mathcal{W}_{2,\rho}(yi)<0\;\hbox{if}\; y\in(1,y_{2,\rho})$, and $ \frac{\partial}{\partial y}\mathcal{W}_{2,\rho}(y i) >0\;\hbox{if}\; y\in(y_{2,\rho},\infty)$

The critical point  $y_{2,\rho}$ is the unique solution of
$
\frac{\partial}{\partial y}\mathcal{W}_{2,\rho}(yi)=0,\;y\in(1,\sqrt3].
$

Moreover, if $\rho \in(0,\rho_{2})$, then
 \begin{equation}\aligned\label{YYY}
\frac{\partial y_{2, \rho}}{\partial \rho} <0.
 \endaligned\end{equation}

\noindent
2. when $\rho\in[\rho_2,+\infty)$, the function $y\rightarrow\mathcal{W}_{2,\rho}(yi), y>0$ admits only one critical point
at 1, and we have  $\frac{\partial}{\partial y}\mathcal{W}_{2,\rho}(yi)<0\;\hbox{if}\; y\in(0,1)$, $\frac{\partial}{\partial y}\mathcal{W}_{2,\rho}(yi)>0\;\hbox{if}\; y\in(1,\infty)$.

\end{proposition}

As in Section 6, by Lemma \ref{LemmaDP} and transformation \eqref{Theta}, we have that
\begin{lemma}\label{LemmaE21Gamma}
\begin{equation}\aligned\nonumber
\theta(2;yi)=\sqrt{\frac{y}{2}}\vartheta_3(2y)\vartheta_3(\frac{y}{2}),\;\;
\theta(1;\frac{yi+1}{2})=\sqrt{\frac{y}{2}}\big(
\vartheta_3(2y)\vartheta_3(\frac{y}{2})+\vartheta_2(2y)\vartheta_4(\frac{y}{2})
\big).
 \endaligned\end{equation}

\end{lemma}


Recall by \eqref{XYAB2} and \eqref{Theta},
\begin{equation}\aligned\nonumber
\mathcal{A}(y):=\sqrt2\vartheta_3(2y)\vartheta_3(\frac{2}{y})=\sqrt y\vartheta_3(2y)\vartheta_3(\frac{y}{2}),\;
\mathcal{B}(y):=\sqrt 2\vartheta_2(2y)\vartheta_2(\frac{2}{y})=\sqrt y\vartheta_2(2y)\vartheta_4(\frac{y}{2}).
 \endaligned\end{equation}

Next we state Theorem \ref{QAB}, which provides the key argument to prove Proposition \ref{YYY2}.

\begin{theorem}\label{QAB}

The function $y\mapsto \frac{\mathcal{B}'(y)}{\mathcal{A}'(y)}, y>0$ has only one critical point at $y=1$, and furthermore
  $\Big(\frac{\mathcal{B}'(y)}{\mathcal{A}'(y)}\Big)'<0, \;\;y \in(0,1)$ and $\Big(\frac{\mathcal{B}'(y)}{\mathcal{A}'(y)}\Big)'>0, \;\;y \in(1,\infty)$.
\end{theorem}

\begin{proof}

By Lemma \ref{Lemmaf01},
\begin{equation}\aligned\label{ABD}
\mathcal{A}'(\frac{1}{y})=-y^2\mathcal{A}'(y),\;\mathcal{B}'(\frac{1}{y})=-y^2\mathcal{B}'(y).
\endaligned\end{equation}
Let
$$
\mathcal{C}(y):=\frac{\mathcal{B}'(y)}{\mathcal{A}'(y)}.
$$
Then
\begin{equation}\aligned\nonumber
\mathcal{C}(\frac{1}{y})=\mathcal{C}(y).
 \endaligned\end{equation}
Hence
\begin{equation}\aligned\label{CV}
\mathcal{C}'(\frac{1}{y})=-y^2\mathcal{C}'(y).
 \endaligned\end{equation}
In particular,
$
\mathcal{C}'(1)=0,
$
i.e., $y=1$ is the critical point of $\mathcal{C}(y)$.
This, combining with Lemma \ref{E111}, shows that the $\mathcal{C}(y)$ by the quotient form is well defined.

By \eqref{CV}, it suffices to prove that
\begin{equation}\aligned\nonumber
\mathcal{C}'(y)>0\;\;y\in(1,\infty).
 \endaligned\end{equation}

To prove this, we need to divide it into two parts of $y$: the small case $ y\in[k,\infty)$ and the large case $y\in(1,k)$,
where the parameter $k$ is sightly bigger than 1 and will be determined later. (In fact $ k= 1.05$.)
\vskip0.1in

\noindent
{\bf Case (a): $y\in [k,\infty)$}
One has
$$
\mathcal{C}'(y)=\frac{\mathcal{B}''(y)\mathcal{A}'(y)-\mathcal{A}''(y)\mathcal{B}'(y)}{(\mathcal{A}'(y))^2}.
$$
Then we need to estimate the lower bound of $\mathcal{B}''(y)\mathcal{A}'(y)-\mathcal{A}''(y)\mathcal{B}'(y)$.

By \eqref{XYAB2},
\begin{equation}\aligned\label{Aae}
\mathcal{A}(y)&=\sqrt{y}\big(1+2\sum_{n=1}^\infty e^{-2\pi n^2y}\big)\big(1+2\sum_{n=1}^\infty e^{-\frac{\pi}{2} n^2y}\big)\\
&=\Big(\sqrt y+2\sqrt y e^{-\frac{\pi y}{2}}+4\sqrt y e^{-2\pi y}+4\sqrt y e^{-\frac{5}{2}\pi y}+4\sqrt y e^{-4\pi y}+2\sqrt y e^{-\frac{9}{2}\pi y}+2\sqrt y\big(\sum_{n=2}^\infty e^{-2\pi n^2y}+\sum_{n=4}^\infty e^{-\frac{\pi}{2} n^2y}\big)\Big)\\
&+\Big(4\sqrt y e^{-\frac{5}{2}\pi y}\Big(
\sum_{n=2}^\infty e^{-\frac{1}{2}\pi(n^2-1) y}+\sum_{n=2}^\infty e^{-2\pi(n^2-1) y}+\sum_{n=2}^\infty e^{-\frac{1}{2}\pi(n^2-1) y}\cdot\sum_{n=2}^\infty e^{-2\pi(n^2-1) y}
\Big)\Big)\\
:& = \mathcal{A}_a (y)+ \mathcal{A}_e (y)
 \endaligned\end{equation}
 where $\mathcal{A}_a (y)$ and $\mathcal{A}_e (y)$ are defined at the last equality. $\mathcal{A}_e (y)$ is the error part which will be proved to satisfy
$$\|\mathcal{A}_e \|_{C^2}\leq C \sqrt y e^{-\frac{13}{2}\pi y}.
$$

For $\mathcal{B}(y)$,  by \eqref{XYAB2}, we rewrite as
\begin{equation}\aligned\nonumber
\mathcal{B}(y)&=\sqrt y\vartheta_2(2y)\Big(\vartheta_3(2y)-\vartheta_2(2y)\Big)\\
&=2\sqrt y\sum_{n=1}^\infty e^{-2\pi y(n-\frac{1}{2})^2}+4\sqrt y\sum_{n=1}^\infty e^{-2\pi y(n-\frac{1}{2})^2}\sum_{n=1}^\infty e^{-2\pi n^2 y}
-4\sqrt y\Big(\sum_{n=1}^\infty e^{-2\pi y(n-\frac{1}{2})^2}\Big)^2\\
&=\Big(2\sqrt y e^{-\frac{1}{2}\pi y}+4\sqrt y e^{-\frac{5}{2}\pi y}+2\sqrt y e^{-\frac{9}{2}\pi y}-4\sqrt y e^{-\pi y}\Big)\\
&+\Big(2\sqrt y \sum_{n=3}^\infty e^{-\frac{1}{2}(2n-1)^2\pi y}
+4\sqrt y e^{-\frac{5}{2}\pi y}\Big(
 \sum_{n=2}^\infty e^{-\frac{1}{2}((2n-1)^2-1)\pi y}+\sum_{n=2}^\infty e^{-2(n^2-1)\pi y}\\
 &+\sum_{n=2}^\infty e^{-\frac{1}{2}((2n-1)^2-1)\pi y}\sum_{n=2}^\infty e^{-2(n^2-1)\pi y}
\Big)
-8\sqrt y \sum_{n=2}^\infty e^{-\frac{1}{2}(2n-1)^2\pi y}\\
&-4\sqrt y (\sum_{n=2}^\infty e^{-\frac{1}{2}(2n-1)^2\pi y})^2\Big)\\
:&=\mathcal{B}_a(y)+\mathcal{B}_e(y)
 \endaligned\end{equation}
where $\mathcal{B}_a (y)$ and $\mathcal{B}_e (y)$ are defined at the last equality. That is, we have
\begin{equation}\aligned\label{Bae}
\mathcal{B}(y)=\mathcal{B}_a(y)+\mathcal{B}_e(y),
\endaligned\end{equation}
where $\mathcal{B}_a(y), \mathcal{B}_e(y)$ is the approximate part and the error part of $\mathcal{B}(y)$ respectively.

We have the following estimate
$$\|\mathcal{B}_e \|_{C^2}\leq C \sqrt y e^{-\frac{13}{2}\pi y}, \;\;y\geq1.
$$
To prove that
\begin{equation}\aligned\label{Cpk}
\mathcal{C}'(y)>0\;\;\hbox{if}\;\;y\in(k,\infty),
\endaligned\end{equation}
it suffices to prove that
\begin{equation}\aligned\nonumber
\mathcal{B}''(y)\mathcal{A}'(y)-\mathcal{A}''(y)\mathcal{B}'(y)>0\;\;\hbox{if}\;\;y\in(k,\infty).
\endaligned\end{equation}

By \eqref{Aae}, there holds
\begin{equation}\aligned\nonumber
\mathcal{B}''\mathcal{A}'-\mathcal{A}''\mathcal{B}'
=\Big(\mathcal{B}_a''\mathcal{A}_a'-\mathcal{A}_a''\mathcal{B}_a'\Big)
+\Big(\mathcal{B}_e''\mathcal{A}'-\mathcal{B}_e' \mathcal{A}''+\mathcal{B}_a''\mathcal{A}_e'-\mathcal{A}_e''\mathcal{B}_a'\Big).
 \endaligned\end{equation}
Here $\Big(\mathcal{B}_a''\mathcal{A}_a'-\mathcal{A}_a''\mathcal{B}_a'\Big)$ and
$\Big(\mathcal{B}_e''\mathcal{A}'-\mathcal{B}_e' \mathcal{A}''+\mathcal{B}_a''\mathcal{A}_e'-\mathcal{A}_e''\mathcal{B}_a'\Big)$ are the approximate and error part of $\mathcal{B}''\mathcal{A}'-\mathcal{A}''\mathcal{B}'$ respectively.

To estimate the approximate part, we use the monotonicity of a weighted function, i.e.
  \begin{equation}\aligned\label{Estimate1}
y\rightarrow\frac{4y}{\pi}e^{\frac{1}{2}\pi y}\Big(\mathcal{B}_a''\mathcal{A}_a'-\mathcal{A}_a''\mathcal{B}_a'\Big)(y)
\endaligned\end{equation}
 is strictly increasing.

For the error term, we have  the following control
\begin{equation}\aligned\label{Estimate2}
|\Big(\mathcal{B}_e''\mathcal{A}'-\mathcal{B}_e' \mathcal{A}''+\mathcal{B}_a''\mathcal{A}_e'-\mathcal{A}_e''\mathcal{B}_a'\Big)(y)|\leq C\sqrt y e^{-\frac{13}{2}\pi y},\;\;y\geq1
 \endaligned\end{equation}
which decays fast.

Combining \eqref{Estimate1} and \eqref{Estimate2}, one deduces that
\begin{equation}\aligned\label{AB105}
\Big(\mathcal{B}''\mathcal{A}'-\mathcal{A}''\mathcal{B}'\Big)(y)>\;\;\hbox{if}\;\;y\in[1.05,\infty).
 \endaligned\end{equation}
This proves that
\begin{equation}\aligned\label{Cpk000}
\mathcal{C}'(y)>0\;\;\hbox{if}\;\;y\in[1.05,\infty).
\endaligned\end{equation}

The detailed proofs of \eqref{Estimate1}, \eqref{Estimate2} and \eqref{AB105} will be given in the Appendix 2.

\vskip0.1in

\noindent
{\bf Case (b): $ y\in (1,k)$}

To prove
\begin{equation}\aligned\label{AB1k}
\Big(\frac{\mathcal{B}'(y)}{\mathcal{A}'(y)}\Big)'>0, \hbox{on}\; y\in(1,k),
\endaligned\end{equation}
by \eqref{XY1}, it suffices to prove that
$$
\Big(\frac{\mathcal{B}''(y)}{\mathcal{A}''(y)}\Big)'>0, \hbox{on}\; y\in(1,k),
$$
given that
\begin{equation}\aligned\label{AB10}
\mathcal{A}'(1)=\mathcal{B}'(1)=0
\endaligned\end{equation}
which follows from \eqref{ABD}. Here as in \eqref{2L1}, we need $\mathcal{A}''(y)>0$ in  small interval such as $(1,1.2]$ (we omit the details here).

To proceed, we notice that
\begin{equation}\aligned\label{L2}
\big(\frac{\mathcal{B}''(y)}{\mathcal{A}''(y)}\big)'
=\frac{\mathcal{B}'''(y)\mathcal{A}''(y)-\mathcal{B}''(y)\mathcal{A}'''(y)}{\mathcal{A}''^2(y)}.
\endaligned\end{equation}
Define
\begin{equation}\aligned\nonumber
f_{\mathcal{AB}}(y):=\mathcal{B}'''(y)\mathcal{A}''(y)-\mathcal{B}''(y)\mathcal{A}'''(y).
\endaligned\end{equation}

Same as \eqref{2ABf000}, we see that
\begin{equation}\aligned\label{ABf000}
f_{\mathcal{AB}}(1)=0.
\endaligned\end{equation}
Then to prove \eqref{AB1k}, it suffices to prove that
\begin{equation}\aligned\label{ABfp}
f_{\mathcal{AB}}'(y)>0\;\;\hbox{for}\;\;y\in(1,k).
\endaligned\end{equation}

Now by \eqref{Aae} and \eqref{Bae} we can write as
\begin{equation}\aligned\label{fAB12}
f_{\mathcal{AB}}'&=\mathcal{B}''''\mathcal{A}''-\mathcal{B}''\mathcal{A}''''\\
&=\Big(\mathcal{B}_a''''\mathcal{A}_a''-\mathcal{B}_a''\mathcal{A}_a''''\Big)
+\Big(\mathcal{B}_e''''\mathcal{A}''-\mathcal{B}_e'' \mathcal{A}''''+\mathcal{B}_a''''\mathcal{A}_e''-\mathcal{A}_e''''\mathcal{B}_a''\Big).
\endaligned\end{equation}

The main part is $\Big(\mathcal{B}_a''''\mathcal{A}_a''-\mathcal{B}_a''\mathcal{A}_a''''\Big)$ which is is not monotonically decreasing or increasing. Instead, a weighted
\begin{equation}\aligned\label{fM1}
y\rightarrow \frac{32y^4}{\pi}e^{\frac{1}{2}\pi y}\Big(\mathcal{B}_a''''\mathcal{A}_a''-\mathcal{B}_a''\mathcal{A}_a''''\Big)(y)
\endaligned\end{equation}
is strictly decreasing on $(1,\infty)$.

 For the error part in \eqref{fAB12}, one deduces the following upper bound estimate,
\begin{equation}\aligned\label{fM2}
|\Big(\mathcal{B}_e''''\mathcal{A}''-\mathcal{B}_e'' \mathcal{A}''''+\mathcal{B}_a''''\mathcal{A}_e''-\mathcal{A}_e''''\mathcal{B}_a''\Big)(y)|
\leq C \sqrt y e^{-\frac{13}{2}\pi y}, y\geq1
\endaligned\end{equation}
which decays  very fast.

Combining \eqref{fM1}, \eqref{fM2} and  \eqref{fAB12}, we can show that
\begin{equation}\aligned\label{fAB1200}
f_{\mathcal{AB}}'(y)>0\;\;\hbox{if}\;\;y\in(1,1.12].
\endaligned\end{equation}

The detailed proof of \eqref{fM1}, \eqref{fM2} and \eqref{fAB1200} is tedious and will be given in the Appendix 2.  This completes the proof.

\end{proof}



Finally we give the proof of Proposition \ref{YYY2}.

\begin{proof} By Lemma \ref{Lemmaf01}, the functional $\mathcal{W}_{2,\rho}(yi)$ satisfies the functional equations
 \begin{equation}\aligned\label{HL}
\mathcal{H}'(\frac{1}{y})=-y^2\mathcal{H}'(y).
 \endaligned\end{equation}
Hence $\mathcal{H}'(1)=0$, i.e., $y=1$ is a critical point of $\mathcal{W}_{2,\rho}(yi)$.

By \eqref{HL}, we just need to consider the functional $\mathcal{W}_{2,\rho}(yi)$ on $(1,\infty)$. For this, one uses Theorem \ref{QAB} by rewriting $\frac{\partial}{\partial y}\mathcal{W}_{2,\rho}(yi)$ as
 \begin{equation}\aligned\label{Hrho}
\sqrt2\frac{\partial}{\partial y}\mathcal{W}_{2,\rho}(yi)&
=\frac{\partial}{\partial y}\Big(\sqrt2\theta(1;\frac{yi+1}{2})+\rho \sqrt2\theta(2;yi)\Big)=\mathcal{A}'(y)+\mathcal{B}'(y)+\rho\mathcal{A}'(y)\\
&=\mathcal{A}'(y)\cdot\Big(
1+\frac{\mathcal{B}'(y)}{\mathcal{A}'(y)}+\rho
\Big).
 \endaligned\end{equation}
By Lemma \ref{E111}, we see that
\begin{equation}\aligned\label{Ay000}
\mathcal{A}'(y)>0\;\; y\in(1,\infty)\;\;\;\hbox{and}\;\;\;1+\frac{\mathcal{B}'(\sqrt3)}{\mathcal{A}'(\sqrt3)}=0.
 \endaligned\end{equation}
By Theorem \ref{QAB}, there holds
 \begin{equation}\aligned\label{HMMM}
\frac{d}{dy}\Big(
1+\frac{\mathcal{B}'(y)}{\mathcal{A}'(y)}+\rho
\Big)>0,\;\; y\in(1,\infty).
 \endaligned\end{equation}
From \eqref{HMMM}, in view of \eqref{Hrho} and \eqref{Ay000}, we infer that
 \begin{equation}\aligned\nonumber
\frac{\partial}{\partial y}\mathcal{W}_{2,\rho}(yi)\;\;\hbox{admits at most one zero point on}\;\;(1,\infty).
 \endaligned\end{equation}
 By \eqref{Ay000}, we see that
 \begin{equation}\aligned\label{HK3}
\frac{\partial}{\partial y}\mathcal{W}_{2,\rho}(yi)>0\;\;\hbox{if}\;\;y\in(\sqrt3,\infty).
 \endaligned\end{equation}
 Then one further concludes that the admissible zero point of $\frac{\partial}{\partial y}\mathcal{W}_{2,\rho}(yi)$ must lie on
$(1,\sqrt3]$ (if exists).

Next we consider the function $1+\frac{\mathcal{B}'(y)}{\mathcal{A}'(y)}+\rho$ for $\rho>0 \in (1,\sqrt3)$. At the end point $\sqrt3$, we have that
 \begin{equation}\aligned\label{LM3}
\Big(
1+\frac{\mathcal{B}'(y)}{\mathcal{A}'(y)}+\rho
\Big)\mid_{y=\sqrt3}=0+\rho=\rho>0
 \endaligned\end{equation}
because of \eqref{Ay000}.

Since $\mathcal{A}'(1)=\mathcal{B}'(1)$, at the other end point $1$, one evaluates
 \begin{equation}\aligned\label{LM1}
\Big(
1+\frac{\mathcal{B}'(y)}{\mathcal{A}'(y)}+\rho
\Big)\mid_{y=1}&=1+\rho+\lim_{y\rightarrow1}\frac{\mathcal{B}'(y)}{\mathcal{A}'(y)}=1+\rho+\lim_{y\rightarrow1}\frac{\mathcal{B}''(y)}{\mathcal{A}''(y)}\\
&=1+\rho+\frac{\mathcal{B}''(1)}{\mathcal{A}''(1)}
 \endaligned\end{equation}
 by L'Hospital's rule.

In view of \eqref{LM3} and \eqref{LM1}, one deduces from \eqref{HMMM} that
 \begin{equation}\aligned\label{ABppp}
&\Big(
1+\frac{\mathcal{B}'(y)}{\mathcal{A}'(y)}+\rho
\Big)\;\;\hbox{admits one zero point on}\;\;(1,\sqrt3)\\
&\Leftrightarrow
1+\rho+\frac{\mathcal{B}''(1)}{\mathcal{A}''(1)}<0
 \endaligned\end{equation}
 which implies that
 \begin{equation}\aligned\nonumber
\rho<\rho_{2}:=-1-\frac{\mathcal{B}''(1)}{\mathcal{A}''(1)}.
 \endaligned\end{equation}
It follows that by
\eqref{Hrho} and \eqref{ABppp}, for $\rho\geq\rho_{2}$, $\frac{\partial}{\partial y}\mathcal{W}_{2,\rho}(yi)$ admits no zero point on $(1,\infty)$  Therefore the part 2 of Proposition \ref{YYY2}
follows from \eqref{HL}.

For $\rho\in(0,\rho_2)$, we denote the zero root of $\Big(
1+\frac{\mathcal{B}'(y)}{\mathcal{A}'(y)}+\rho
\Big)$ (and hence also of $\frac{\partial}{\partial y}\mathcal{W}_{2,\rho}(yi)$) as
$y_{2,\rho}$. Then  by \eqref{ABppp} $y_{2,\rho}\in(1,\sqrt3)$. Thus by \eqref{HL} there is another zero point $\frac{1}{y_{2,\rho}}\in(\frac{\sqrt3}{3},1)$ of $\frac{\partial}{\partial y}\mathcal{W}_{2,\rho}(yi)$. By \eqref{HK3}, \eqref{HL}, \eqref{HMMM} and \eqref{Hrho},  the part 1 of Proposition \ref{YYY2} is proved.

Finally from \eqref{HMMM}, we have that
 \begin{equation}\aligned\nonumber
\frac{d}{d\rho}y_{2,\rho}<0.
 \endaligned\end{equation}
This proves \eqref{YYY}. (For $\rho=0$, one has $y_{2,\rho}=\sqrt3$ by \eqref{Ay000}). The proof is thus completed.

\end{proof}

\section{Proofs of Theorems \ref{Th1} \ref{Th2} and \ref{Th3}}
\setcounter{equation}{0}

 In this section, we are ready to finish the proof of the main results of Theorems \ref{Th1} \ref{Th2} and \ref{Th3}.  To make the presentation clear, we introduce the following notations to denote various geometric sets:

\begin{equation}\aligned\nonumber
\mathbb{H}:&=\{z\mid y >0
\},\\
\Omega_a:&=\{z\mid |z|\geq1, 0\leq x<1
\},\\
\Omega_b:&=\{z\mid |z|\geq1, 0\leq x\leq \frac{1}{2}
\}\cup\{z\mid |z|=1, \frac{1}{2}\leq x <1
\},\\
\Omega_c:&=\{z\mid |z|\geq1, 0\leq x \leq \frac{1}{2}
\},\\
\Omega_d:&=\{z\mid |z|=1, 0\leq x \leq\frac{1}{2}\}\cup \{
z\mid \ x=0, 1\leq y <\infty
\},\\
\Omega_e:&=\{z\mid |z|=1, 0\leq x \leq\frac{1}{2}\}\cup \{
z\mid x=0, 1\leq y \leq\sqrt3
\},\\
\Omega_{ea}:&=\{z\mid x=0, 1\leq y \leq\sqrt3\},\\
\Omega_{eb}:&=\{z\mid |z|=1, 0\leq x <\frac{1}{2}\}.
 \endaligned\end{equation}

We divide the proof into the following steps:

\medskip

\noindent
{\bf Step 1: Reducing minimization problem from $\mathbb{H}$  to $\rightarrow\Omega_a$}.
\vskip0.1in

 This is a consequence of Theorem \ref{Thgamma}  and the properties of the fundamental group \eqref{GroupG2} and fundamental domain \eqref{Fd4}:
\begin{equation}\aligned\label{HTa}
\min_{z\in \mathbb{H}}\mathcal{W}_{1,\rho}(z)\equiv\min_{z\in \Omega_a}\mathcal{W}_{1,\rho}(z),\;
\min_{z\in \mathbb{H}}\mathcal{W}_{2,\rho}(z)\equiv\min_{z\in \Omega_a}\mathcal{W}_{2,\rho}(z).
\endaligned\end{equation}

\vskip0.1in

\noindent
{\bf Step 2: Reducing minimization problem from $\Omega_a$ to $\Omega_b$}.
\vskip0.1in

This follows from  Corollary \ref{Coro2}:
\begin{equation}\aligned\nonumber
\min_{z\in \Omega_a}\mathcal{W}_{1,\rho}(z)\equiv\min_{z\in \Omega_b}\mathcal{W}_{1,\rho}(z),\;
\min_{z\in \Omega_a}\mathcal{W}_{2,\rho}(z)\equiv\min_{z\in \Omega_b}\mathcal{W}_{2,\rho}(z).
\endaligned\end{equation}

\medskip

\noindent
{\bf Step 3: Reducing minimization problem  from $\Omega_b$ to $\Omega_c$}.
\vskip0.1in

We first show that
\begin{equation}\aligned\label{Arc1}
\min_{z\in \{z\mid |z|=1, \frac{1}{2}\leq x <1
\}}\mathcal{W}_{j,\rho}(z)&\equiv\mathcal{W}_{1,\rho}(\frac{1}{2}+i\frac{\sqrt3}{2}), j=1,2.
\endaligned\end{equation}
One can further conclude that the minimizer $\frac{1}{2}+i\frac{\sqrt3}{2}$ is unique by the monotonicity shown below.

In fact, by Propositions \ref{YYY1} and \ref{YYY2}, we see that
\begin{equation}\aligned\label{MM1}
\frac{\partial}{\partial y}\mathcal{W}_{j,\rho}(yi)&>0,\;\;y\in[\sqrt3,\infty), j=1,2.
\endaligned\end{equation}

By the special map $z\mapsto w:=\frac{z-1}{z+1}$,  the set
$\{ yi, y\in[\sqrt3,\infty)\}$ is mapped bijectively to $ \{ |z|=1, \frac{1}{2}\leq \Re(z)<1\}$. By Lemma \ref{Reduce} and \eqref{MM1} we see that both $\mathcal{W}_{1,\rho}(z)$ and $ \mathcal{W}_{2,\rho}(z)$ are monotonically decreasing along the set $\;\; \{|z|=1, \frac{1}{2}\leq x<1\}$. This proves \eqref{Arc1}.

By \eqref{Arc1}, we conclude that
\begin{equation}\aligned\nonumber
\min_{z\in \Omega_b}\mathcal{W}_{1,\rho}(z)\equiv\min_{z\in \Omega_c}\mathcal{W}_{1,\rho}(z),\;\min_{z\in \Omega_b}\mathcal{W}_{2,\rho}(z)\equiv\min_{z\in \Omega_c}\mathcal{W}_{2,\rho}(z).
\endaligned\end{equation}

\medskip

\noindent
{\bf Step 4: Reducing minimization problem  from $\Omega_c$ to  $\Omega_d$}.
\vskip0.1in

In this case, let
$
\rho_{*}=\frac{1}{20}
$
be as  in Propositions \ref{Trans1}. For $\rho\in[0,\rho_{*}]$,  Proposition \ref{Trans1} implies that
\begin{equation}\aligned\nonumber
\min_{z\in \Omega_c}\mathcal{W}_{1,\rho}(z)&\equiv\min_{z\in \Omega_d}\mathcal{W}_{1,\rho}(z), \rho\in[0,\rho_*].
\endaligned\end{equation}

For $\rho\in(\rho_{*},\infty)$, using Lemma \ref{ThWW}, Lemma \ref{Trans2}, and \eqref{Arc1}, we get that
\begin{equation}\aligned\nonumber
\min_{z\in \Omega_c}\mathcal{W}_{1,\rho}(z)&\equiv\rho\min_{w\in \Omega_c}\mathcal{W}_{2,1/\rho}(w), 1/\rho\in(0,1/\rho_{*})\\
&\equiv\rho\min_{w\in \Omega_d}\mathcal{W}_{2,1/\rho}(w), 1/\rho\in(0,1/\rho_{*})\\
&\equiv\min_{z\in \Omega_d}\mathcal{W}_{1,\rho}(z), \rho\in(\rho_{*},\infty).
\endaligned\end{equation}

Therefore, we obtain that
\begin{equation}\aligned\label{CD1}
\min_{z\in \Omega_c}\mathcal{W}_{1,\rho}(z)&\equiv\min_{z\in \Omega_d}\mathcal{W}_{1,\rho}(z), \rho\in[0,\infty).
\endaligned\end{equation}
By Theorem \ref{ThWW}, \eqref{Arc1} and \eqref{CD1}, we have that
\begin{equation}\aligned\label{CD2}
\min_{z\in \Omega_c, \rho\in[0,\infty)}\mathcal{W}_{2,\rho}(z), &\equiv\rho\min_{w\in \Omega_c, 1/\rho\in[0,\infty)}\mathcal{W}_{1,1/\rho}(w), \\
&\equiv\rho\min_{w\in \Omega_d, 1/\rho\in[0,\infty)}\mathcal{W}_{1,1/\rho}(w), \\
&\equiv\min_{z\in \Omega_d, \rho\in[0,\infty)}\mathcal{W}_{2,\rho}(z).
\endaligned\end{equation}

\medskip

\noindent
{\bf Step 5: Reducing minimization problem  from $\Omega_d$ to $\Omega_e$}.
\vskip0.1in
The follows from \eqref{MM1}.

In summary, from Steps 1-5, we conclude that
\begin{equation}\aligned\label{HTE}
\min_{z\in \mathbb{H}}\mathcal{W}_{1,\rho}(z)\equiv\min_{z\in \Omega_e}\mathcal{W}_{1,\rho}(z),\;
\min_{z\in \mathbb{H}}\mathcal{W}_{2,\rho}(z)\equiv\min_{z\in \Omega_e}\mathcal{W}_{2,\rho}(z).
\endaligned\end{equation}
From \eqref{HTE}, we just need to find the minimizer in a much smaller  curve $\Omega_e$. But this gives  no information about uniqueness or multiplicity of the minimizers. In fact,
one can further rule out the possible minimizers of $\min_{z\in \Omega_a}\mathcal{W}_{1,\rho}(z)$, $\min_{z\in \Omega_a}\mathcal{W}_{2,\rho}(z)$ in a large set. Namely, for $z\in\Omega_a\backslash\Omega_e$, there is no any possible minimizer for $\min_{z\in \Omega_a}\mathcal{W}_{1,\rho}(z)$, $\min_{z\in \Omega_a}\mathcal{W}_{2,\rho}(z)$. The possible multiplicity of minimizer is admitted only in Step 1, see \eqref{HTa}. But up the group transformation $\mathcal{G}_2$, the possible minimizer in \eqref{HTa} is unique. Therefore, one can conclude the reduction in \eqref{HTE} is unique up to the group transformation $\mathcal{G}_2$. In the next step we will show that $\min_{z\in \Omega_e}\mathcal{W}_{1,\rho}(z)$, $\min_{z\in \Omega_e}\mathcal{W}_{2,\rho}(z)$ exists , is unique and can be located precisely.

Let $w$ be the map $w(z)=\frac{z-1}{z+1}$ whose inverse is $ z(w)= \frac{1+w}{1-w}$.  Under this map we have $
z=yi\in \Omega_{ea}\mapsto w=\frac{y^2-1}{y^2+1}+i\frac{2y}{y^2+1}\in\Omega_{eb},
w=u+iv\in\Omega_{eb}\mapsto z=i\frac{\sqrt{1-u^2}}{1-u}\in\Omega_{ea}. $

We note that
$$
\rho_{1}<{1}/{\rho_{2}}
<\rho_{2}<{1}/{\rho_{1}}.
$$
See in Propositions \ref{YYY1} and \ref{YYY2}.

Now we consider the minimizer of  $\mathcal{W}_{1,\rho}(z)$ on $ \Omega_{e}$.   We divide into three cases.

\medskip

\noindent
{\bf Case 1. $\rho\in[\rho_{1},{1}/{\rho_{2}}]$.}
 \vskip0.1in

 In this case, $\rho\geq\rho_{1},1/\rho\geq\rho_{2}$. Then by Propositions \ref{YYY1} and \ref{YYY2}, both
$\mathcal{W}_{1,\rho}(z)$ and $\mathcal{W}_{2,\rho}(z)$ are monotonically increasing on $\Omega_{ea}$ along positive $y$ axis direction. Then it follows that $\mathcal{W}_{1,\rho}(z)$ is monotonically increasing on $\Omega_{eb}$ clockwise. Therefore, the minimizer of of $\mathcal{W}_{1,\rho}(z)$ on $\Omega_e$ is uniquely achieved at $y=i$.
\vskip0.1in

\noindent
{\bf Case 2. $\rho\in(0,\rho_{1}$).}
\vskip0.1in
In this case, $1/\rho>1/{\rho_{1}}>\rho_{2}$. Then by Proposition \ref{YYY2}, $\mathcal{W}_{2,1/\rho}(z)$ is monotonically increasing on $\Omega_{ea}$ along positive $y$ axis direction. It follows from Lemma \ref{Reduce} or Theorem \ref{ThWW}
that $\mathcal{W}_{1,\rho}(z)$ is monotone increasing on $\Omega_{eb}$ clockwise. On the other hand, by Proposition \ref{YYY1}, $\mathcal{W}_{1,\rho}(z)$
admits a unique minimizer at $y=iy_{1,\rho}\in i(1,\sqrt3)$ on $\Omega_{ea}$. We conclude that  $\mathcal{W}_{1,\rho}(z)$ has a unique minimizer at $z_{1,\rho}=iy_{1,\rho}\in(1,\sqrt3)$ on $\Omega_{e}$.

\vskip0.1in

\noindent
{\bf Case 3. $\rho\in(1/\rho_{2},\infty)$.}
\vskip0.1in
In this case, since $1/\rho<\rho_{2}$, by Proposition \ref{YYY2}, $\mathcal{W}_{2,1/\rho}(z)$ has a unique minimizer at $y=y_{2,1/\rho}\in(1,\sqrt3)$ on $\Omega_{ea}$. Then by Theorem \ref{ThWW} or Lemmas \ref{Reduce}, $\mathcal{W}_{1,\rho}(\cdot)$ has a unique minimizer
\begin{equation}\aligned\label{2Gamma2}
z_{1,\rho}=\frac{y_{2,1/\rho}^2-1}{y_{2,1/\rho}^2+1}+i\frac{2y_{2,1/\rho}}{y_{2,1/\rho}^2+1}\in\hbox{inner points of}\;\;\Omega_{eb}.
\endaligned\end{equation}
 On the other side, one has $\rho>1/{\rho_{2}}>\rho_{1}$. Then by Proposition \ref{YYY1}, $\mathcal{W}_{1,\rho}(z)$ is monotone increasing on $\Omega_{ea}$ along the positive $y$ axis direction. Therefore, \eqref{2Gamma2} gives the minimizer of $\mathcal{W}_{1,\rho}(z)$ on $\Omega_e$.

This proves Theorems \ref{Th1} and \ref{Th3}. Theorem \ref{Th2} follows from Theorem \ref{Th1} and Lemma 3.3.

\section{Proof of Mueller-Ho functional and Mueller-Ho conjecture}

\setcounter{equation}{0}

\noindent
{\bf Proof of Lemma \ref{Point1}.} Since the computation is elementary, we omit the details here.


\noindent
{\bf Proof of Lemma \ref{Cw3}.}
\begin{equation}\aligned\nonumber
\mathcal{J}(z;\frac{1}{2}, \frac{1}{2}) &=\sum_{m,n}e^{-\frac{\pi}{y}|mz-n|^2}\cos((m+n)\pi)\\
&=\sum_{m,n}e^{-\frac{\pi}{y}|mz-n|^2}\big(1+\cos((m+n)\pi)\big)-\sum_{m,n}e^{-\frac{\pi}{y}|mz-n|^2}\\
&=\sum_{m,n}e^{-\frac{\pi}{y}|mz-n|^2}2\cos^2(\frac{(m+n)\pi}{2})-\theta(1;z)=\sum_{m+n=2k,k\in\mathbb{Z}}2e^{-\frac{\pi}{y}|mz+n|^2}-\theta(1;z)\\
&=2\sum_{m,k}e^{-\frac{\pi}{y}|m(z+1)-2k|^2}-\theta(1;z)=2\sum_{m,k}e^{-\frac{2\pi}{\Im(\frac{z+1}{2})}|m\frac{z+1}{2}-k|^2}-\theta(1;z)\\
&=2\theta(2;\frac{z+1}{2})-\theta(1;z).
\endaligned\end{equation}

\noindent
{\bf Proof of Theorem \ref{ThBEC}.} This  follows by Theorems \ref{Th1}, \ref{Th2} and \ref{Th3}, by the relation $\rho=\frac{1-\alpha}{2\alpha}$.


\section{Appendix 1: Proof of Lemma \ref{Lemma1313}}

\setcounter{equation}{0}

Recall that
\begin{equation}\aligned
\mathcal{J}(z;a,b)=\sum_{(m,n)\in\mathbb{Z}^2}e^{-\frac{\pi}{y}|mz-n|^2}\cos(2\pi(ma+nb)).
\endaligned
\end{equation}
In this appendix we show that when the lattice is square type, then $(\frac{1}{3},\frac{1}{3})$ is not a critical point while when the lattice is   hexagonal (or triangular) it is a critical point.

First  we show that
\begin{lemma}
\begin{equation}\aligned
\frac{\partial}{\partial a}\mathcal{J}(z;a,b)|_{z=i,(a,b)=(\frac{1}{3},\frac{1}{3})}
=\frac{\partial}{\partial b}\mathcal{J}(z;a,b)|_{z=i,(a,b)=(\frac{1}{3},\frac{1}{3})}<0.
\endaligned\end{equation}
\end{lemma}
This implies that $\mathcal{J}(z;a,b)$ is not always critical point for any lattice shape.

\begin{proof}
\begin{equation}\aligned
\frac{\partial}{\partial a}\mathcal{J}(z;a,b)|_{z=i,(a,b)=(\frac{1}{3},\frac{1}{3})}=-2\pi\sum_{m,n}me^{-\pi(m^2+n^2)}\sin(\frac{2\pi(m+n)}{3})\\
\frac{\partial}{\partial a}\mathcal{J}(z;a,b)|_{z=i,(a,b)=(\frac{1}{3},\frac{1}{3})}=-2\pi\sum_{m,n}ne^{-\pi(m^2+n^2)}\sin(\frac{2\pi(m+n)}{3}).
\endaligned\end{equation}
It is clear that
$$
\frac{\partial}{\partial a}\mathcal{J}(a,b;z)|_{z=i,(a,b)=(\frac{1}{3},\frac{1}{3})}=\frac{\partial}{\partial b}\mathcal{J}(a,b;z)|_{z=i,(a,b)=(\frac{1}{3},\frac{1}{3})}.
$$
Let
$$
A:=\sum_{m,n}e^{-\pi(m^2+n^2)}\sin(\frac{2\pi(m+n)}{3})m.
$$
Equivalently, we show that
$$
A>0.
$$
Grouping by $m+n=3k+j,j=0,1,2$, we have
\begin{equation}\aligned\label{AA}
\frac{A}{\sin(\frac{\pi}{3})}
&=\sum_{m+n\equiv1(\mod 3)}me^{-\pi(m^2+n^2)}
-\sum_{m+n\equiv2(\mod 3)}me^{-\pi(m^2+n^2)}.
\endaligned\end{equation}
For the first part in \eqref{AA}, splitting the summation by $m>0$ or $m<0$, we have (dropping the $\mod3$)
\begin{equation}\aligned\label{AA1}
\sum_{m+n\equiv1}e^{-\pi(m^2+n^2)}m
=\sum_{m>0,m+n\equiv1}me^{-\pi(m^2+n^2)}-\sum_{m>0,m+n\equiv2}me^{-\pi(m^2+n^2)}
\endaligned\end{equation}
For the second part in \eqref{AA}, similarly, one has
\begin{equation}\aligned\label{AA2}
\sum_{m+n\equiv2}e^{-\pi(m^2+n^2)}m
=\sum_{m>0,m+n\equiv2}me^{-\pi(m^2+n^2)}-\sum_{m>0,m+n\equiv1}me^{-\pi(m^2+n^2)}.
\endaligned\end{equation}
By \eqref{AA1} and \eqref{AA2}, we have
\begin{equation}\aligned\nonumber
\sum_{m+n\equiv2}me^{-\pi(m^2+n^2)}=-\sum_{m+n\equiv1}me^{-\pi(m^2+n^2)}
\endaligned\end{equation}
and by \eqref{AA}
\begin{equation}\aligned\label{AA3}
\frac{A}{2\sin(\frac{\pi}{3})}=
\sum_{m>0,m+n\equiv1}me^{-\pi(m^2+n^2)}-\sum_{m>0,m+n\equiv2}me^{-\pi(m^2+n^2)}.
\endaligned\end{equation}
Notice that $e^{-\pi}$ is one term in the first summation in  \eqref{AA3}, it suffices to prove that
$$
\sum_{m>0,m+n\equiv2}me^{-\pi(m^2+n^2)}<e^{-\pi}.
$$
Now we have
\begin{equation}\aligned\nonumber
&\sum_{m>0,m+n\equiv2}e^{-\pi(m^2+n^2)}m=
\sum_{m=1}^\infty\sum_{k\in\mathbb{N}} me^{-\pi(m^2+(3k+2)^2)}\\
=&\sum_{m=1}^\infty me^{-\pi m^2}\sum_{k\in\mathbb{N}} e^{-\pi(3k+2)^2}
<(e^{-\pi}+4e^{-4\pi})(e^{-\pi}+2e^{-4\pi})<e^{-\pi}.
\endaligned\end{equation}
This completes the proof.
\end{proof}


Next we show that $(a,b)= (\frac{1}{3}, \frac{1}{3})$ is a critical point when $z=\frac{1}{2}+i \frac{\sqrt{3}}{2}$.

\begin{proof}
We first claim that
\begin{equation}\aligned\label{I1000}
\sum_{(m,n)\in \mathbb{Z}^2}e^{-x(m^2+n^2-mn)}m\sin(\frac{2\pi(m+n)}{3})=0,\;\;\hbox{for}\;\;\forall x>0.
\endaligned
\end{equation}
To prove \eqref{I1000}, it suffices to prove that
\begin{equation}\aligned
\sum_{n}e^{-x(m^2+n^2-mn)}\sin(\frac{2\pi(m+n)}{3})=0,\;\;\hbox{for}\;\;\forall x>0.
\endaligned
\end{equation}
In fact,
\begin{equation}\aligned
&\sum_{n}e^{-x(m^2+n^2-mn)}\sin(\frac{2\pi(m+n)}{3})\\
=&-e^{-\frac{3}{4}x m^2}\sum_{n}e^{-\frac{x}{4}(2n-m)^2}\sin{\frac{\pi(2n-m)}{3}}\\
=&0.
\endaligned
\end{equation}
In the last equality, one uses $2n-m, n\in\mathbb{Z}$ and takes all the even or odd integers when $m$ is even or odd.

By simple calculation, now the second part of Lemma \ref{Lemma1313} is equivalent to
\begin{equation}\aligned
\sum_{m,n}e^{-\frac{\pi}{2y}\big((m-n)^2y^2+(m+n)^2\big)}m\sin{\frac{2\pi(m+n)}{3}}=0,\;\;\hbox{if}\;\;y=\sqrt 3
\endaligned
\end{equation}
which is of consequence of  \eqref{I1000}. This completes the proof.

\end{proof}


\section{Appendix 2: The rest of proof in Theorem \ref{QXY} and Theorem \ref{QAB}}

In this appendix, we finish the technical proofs of Theorems \ref{QXY} and \ref{QAB}.

Throughout this appendix we frequently  use the following Lemma whose proof is straightforward calculus and is omitted:

\begin{lemma}\label{Lappendix} Let $f(y)^{(j)}$ denote $\frac{d^j}{dy^j}f(y)$. For $, j=1,2,3\cdots$, there holds

\begin{itemize}
\item For $a>0, b>0$,
 \begin{equation}\aligned\nonumber
\Big(y^b e^{-ay}\Big)'<0,\;\;\hbox{if}\;\; y>\frac{b}{a};\;\;\Big(y^b e^{-ay}\Big)''>0,\;\;\hbox{if}\;\; y>\frac{b+\sqrt b}{a}.
\endaligned\end{equation}

  \item For $a>0$,
  \begin{equation}\aligned\nonumber
(-1)^j\Big(\sqrt y e^{-ay}\Big)^{(j)}>0,\;\;\hbox{if}\;\; y>f_j(a).
\endaligned\end{equation}
Here
\begin{equation}\aligned\nonumber
f_1(a)=\frac{1}{2a}, f_2(a)=\frac{1+\sqrt2}{2a}, f_3(a)=\frac{1}{a}, f_4(a)=\frac{1}{2a}.
\endaligned\end{equation}
  \item For $y\geq1$ and $a_n>0$
\begin{equation}\aligned\nonumber
&|\Big(
\sum_{n=k}^\infty \sqrt y e^{-a_ny}
\Big)^{(j)}|\leq(1+\sigma_{j,k})\sqrt y (a_k)^j e^{-a_ky},\;\; \sigma_{j,k}=\sum_{n=k+1}^\infty(\frac{a_n}{a_k})^je^{-(a_n-a_k)}.
\endaligned\end{equation}
\end{itemize}
\end{lemma}

The structure of this appendix is organized as follows.
\eqref{2Estimate1}$\Leftrightarrow$ Lemma \ref{U1};
\eqref{2Estimate2}$\Leftrightarrow$ Lemma \ref{U2};
\eqref{2AB105}$\Leftrightarrow$ Lemma \ref{U3};
\eqref{2fM1}$\Leftrightarrow$ Lemma \ref{V1};
\eqref{2fM2}$\Leftrightarrow$ Lemma \ref{V2};
\eqref{2fAB12}$\Leftrightarrow$ Lemma \ref{V3};
\eqref{Estimate1}$\Leftrightarrow$ Lemma \ref{UU1};
\eqref{Estimate2}$\Leftrightarrow$ Lemma \ref{UU2};
\eqref{AB105}$\Leftrightarrow$ Lemma \ref{UU3};
\eqref{fM1}$\Leftrightarrow$ Lemma \ref{VV1};
\eqref{fM2}$\Leftrightarrow$ Lemma \ref{VV2};
\eqref{fAB1200}$\Leftrightarrow$ Lemma \ref{VV3}.

\subsection{The rest of proof in Theorem \ref{QXY}}

\begin{lemma}\label{U1} $y\mapsto\frac{16y}{\pi}e^{\frac{1}{4}\pi y}\Big(
\mathcal{Y}_a''\mathcal{X}_a'-\mathcal{X}_a''\mathcal{Y}_a'
\Big)(y), y\in[1,\infty)$ is monotonically increasing.
\end{lemma}
\begin{proof} Calculating and grouping the terms, we get
\begin{equation}\aligned\label{PXY}
&\frac{16y}{\pi}e^{\frac{1}{4}\pi y}\Big(
\mathcal{Y}_a''\mathcal{X}_a'-\mathcal{X}_a''\mathcal{Y}_a'
\Big)(y)\\
=&\Big(\pi y-2496e^{-7\pi y}\pi^2y^2-144e^{-7\pi y}
-700e^{-6\pi y}\pi y
-1440e^{-5\pi y}\pi^2 y^2
-288e^{-5\pi y}-2176e^{-4\pi y}\pi y\\
&-840e^{-3\pi y}\pi^2y^2-108e^{-3\pi y}
-243e^{-2\pi y}\pi y
-110e^{-\pi y}\pi y-6\Big)\\
&+\Big(696e^{-7\pi y}\pi y+2016e^{-6\pi y}\pi^2y^2
+168e^{-6\pi y}+1008e^{-5\pi y}\pi y+2208e^{-4\pi y}\pi^2y^2+768e^{-4\pi y}\\
&+234e^{-3\pi y}\pi y+192e^{-2\pi y}\pi^2y^2
+162e^{-2\pi y}
+24e^{-\pi y}\pi^2y^2+132e^{-\pi y}\Big)\\
\endaligned\end{equation}
Denote the terms in first and second brackets of $\frac{16y}{\pi}e^{\frac{1}{4}\pi y}\Big(
\mathcal{Y}_a''\mathcal{X}_a'-\mathcal{X}_a''\mathcal{Y}_a'
\Big)(y)$ by $\mathcal{P}_{\mathcal{XY}}^+$ and $\mathcal{P}_{\mathcal{XY}}^-$ respectively.
One has
$
\frac{16y}{\pi}e^{\frac{1}{4}\pi y}\Big(
\mathcal{Y}_a''\mathcal{X}_a'-\mathcal{X}_a''\mathcal{Y}_a'
\Big)(y)
=\mathcal{P}_{\mathcal{XY}}^+(y)+\mathcal{P}_{\mathcal{XY}}^-(y)
$
by \eqref{PXY}. It remains to prove that
$
\Big(\mathcal{P}_{\mathcal{XY}}^++\mathcal{P}_{\mathcal{XY}}^-\Big)'>0, \;\; y\in[1,\infty).
$

It is clear that the leading order term is $\pi y$, this gives that
$\Big(\mathcal{P}_{\mathcal{XY}}^++\mathcal{P}_{\mathcal{XY}}^-\Big)'>0$ when $y$ is large.

By Lemma \ref{Lappendix}, one has
\begin{equation}\aligned\label{PXYD}
\Big (\mathcal{P}_{\mathcal{XY}}^+\Big)'>\pi,\;\;\Big(\mathcal{P}_{\mathcal{XY}}^-\Big)'<0,\;\;
\Big (\mathcal{P}_{\mathcal{XY}}^+\Big)''<0,\;\;\Big(\mathcal{P}_{\mathcal{XY}}^-\Big)''>0\;\;\hbox{if}\;\; y\geq1.
 \endaligned\end{equation}
Direct calculation shows that $\Big(\mathcal{P}_{\mathcal{XY}}^-\Big)'\mid_{y=2.2}=-3.012967072\cdots$. Then by \eqref{PXYD}
\begin{equation}\aligned
\Big(\mathcal{P}_{\mathcal{XY}}^++\mathcal{P}_{\mathcal{XY}}^-\Big)'(y)>\pi-3.012967072\cdots>0, \;\;\hbox{if}\;\;y\geq2.2.
 \endaligned\end{equation}
Next we prove that
\begin{equation}\aligned
\Big(\mathcal{P}_{\mathcal{XY}}^++\mathcal{P}_{\mathcal{XY}}^-\Big)'(y)>0, \;\;\hbox{for}\;\;y\in[1,2.2].
 \endaligned\end{equation}
To prove this, we regroup the terms by
\begin{equation}\aligned
&\mathcal{P}_{\mathcal{XY}}^+(y)+\mathcal{P}_{\mathcal{XY}}^-(y)\\
=&(\pi y-6)+e^{-\pi y}(-110\pi y+24\pi^2 y^2+132)
+e^{-2\pi y}(-243\pi y+192\pi^2y^2+162)\\
&+e^{-3\pi y}(-840\pi^2y^2-108+234\pi y)
+e^{-4\pi y}(-2176\pi y+2208\pi^2y^2+768)\\
&+e^{-5\pi y}(-1440\pi^2y^2-288+1008\pi y)
+e^{-6\pi y}(-700\pi y+2016\pi^2y^2+168)\\
&+e^{-7\pi y}(-2496\pi^2y^2-144+696\pi y).
\endaligned\end{equation}

To prove this, one divides the interval $[1,2.2]$ into, say, ten subintervals, $[1,2.2)=\cup_{i=0}^9 [a_i,a_{i+1})$. In each intervals, by careful calculations, we can show that the function is positive on each interval.

\end{proof}

\begin{lemma}\label{U2}  The estimates hold:
$
|\Big(\mathcal{Y}_e''\mathcal{X}'-\mathcal{Y}_e' \mathcal{X}''+\mathcal{Y}_a''\mathcal{X}_e'-\mathcal{X}_e''\mathcal{Y}_a'\Big)(y)|\leq
(44\pi^2+18\pi+36\pi y)e^{-\frac{17}{4}\pi y}.
$
\begin{remark}
The coefficient of the bound is not sharp, but the exponential term captures the main feature.
\end{remark}

\end{lemma}

\begin{proof} By Lemma \ref{Lappendix}, one infers that
\begin{equation}\aligned\nonumber
|\mathcal{Y}_e'(y)|\leq 18\pi\sqrt y e^{-\frac{17}{4}\pi y},
|\mathcal{Y}_e''(y)|\leq \frac{290\pi^2}{4}\sqrt y e^{-\frac{17}{4}\pi y},
|\mathcal{X}_e'(y)|\leq 41\pi\sqrt y e^{-5\pi y},
|\mathcal{X}_e''(y)|\leq 201\pi^2\sqrt y e^{-5\pi y}
 \endaligned\end{equation}
For $\mathcal{X}',\mathcal{X}'', \mathcal{Y}_a',\mathcal{Y}_a''$, by their expressions, one has
\begin{equation}\aligned\nonumber
|\mathcal{X}'(y)|\leq \frac{3}{5\sqrt y}, |\mathcal{X}''(y)|\leq(\frac{1}{4y^{3/2}}+2\sqrt y),
|\mathcal{Y}_a'(y)|\leq(\frac{1}{\sqrt y}+2\sqrt y), |\mathcal{Y}_a'(y)|\leq(\frac{1}{4y^{3/2}}+2\sqrt y).
 \endaligned\end{equation}
Thus, one can get the result.

\end{proof}

\begin{lemma}\label{U3} There holds
$
\Big(\mathcal{Y}''\mathcal{X}'-\mathcal{Y}''\mathcal{X}'\Big)(y)>0,\;\;y\in[1.1,\infty).
$
\end{lemma}
\begin{proof} It remains to prove that
$
\frac{16y}{\pi}e^{\frac{1}{4}\pi y}\Big(\mathcal{Y}''\mathcal{X}'-\mathcal{Y}''\mathcal{X}'\Big)(y)>0,\;\;y\in[1.1,\infty).
$

By Lemmas \ref{U1} and \ref{U2},

\begin{equation}\aligned\nonumber
&\frac{16y}{\pi}e^{\frac{1}{4}\pi y}\Big(\mathcal{Y}''\mathcal{X}'-\mathcal{Y}''\mathcal{X}'\Big)(y)\\
=&\frac{16y}{\pi}e^{\frac{1}{4}\pi y}\Big(\mathcal{Y}_a''\mathcal{X}_a'-\mathcal{X}_a''\mathcal{Y}_a'\Big)(y)
+\frac{16y}{\pi}e^{\frac{1}{4}\pi y}\Big(\mathcal{Y}_e''\mathcal{X}'-\mathcal{Y}_e' \mathcal{X}''+\mathcal{Y}_a''\mathcal{X}_e'-\mathcal{X}_e''\mathcal{Y}_a'\Big)(y)\\
\geq&\frac{16y}{\pi}e^{\frac{1}{4}\pi y}\Big(\mathcal{Y}_a''\mathcal{X}_a'-\mathcal{X}_a''\mathcal{Y}_a'\Big)(y)
-\frac{16y}{\pi}(44\pi^2+18\pi+36\pi y)e^{-4\pi y}\\
\geq&\Big(\frac{16y}{\pi}e^{\frac{1}{4}\pi y}\Big(\mathcal{Y}_a''\mathcal{X}_a'-\mathcal{X}_a''\mathcal{Y}_a'\Big)(y)
-{16y}(44\pi+18+36 y)e^{-4\pi y}\Big)\mid_{y=1.1}=0.001671778\cdots, y\in[1.1,\infty)\\
>&0, \;\;\;\;\; y\in[1.1,\infty).
 \endaligned\end{equation}
In the second last step, one uses  the fact that $y\mapsto-{16y}(44\pi+18+36 y)e^{-4\pi y}, y>1$ is strictly increasing.

\end{proof}

\begin{lemma}\label{V1}
$y\rightarrow \frac{512y^4}{\pi}e^{\frac{1}{4}\pi y}\Big(\mathcal{Y}_a''''\mathcal{X}_a''-\mathcal{Y}_a''\mathcal{X}_a''''\Big)(y)
$
is monotonically decreasing on $(1,1.2)$.

\end{lemma}
\begin{proof} By direct calculations, one regroups the terms by
\begin{equation}\aligned
&\frac{512y^4}{\pi}e^{\frac{1}{4}\pi y}
\Big(
\mathcal{Y}_a''''\mathcal{X}_a''-\mathcal{Y}_a''''\mathcal{X}_a''''
\Big)(y)\\
=&-\pi^3y^3+8\pi^2 y^2+84\pi y-144\\
&+e^{-\pi y}\Big(
-240\pi^5y^5-9240\pi y-6320\pi^2y^2+1392\pi^4y^4+350\pi^3y^3+3168
\Big)\\
&+e^{-2\pi y}\Big(
-11232\pi^5y^5-14877\pi^3y^3-20412\pi y-32856\pi^2y^2+36096\pi^4y^4+3888
\Big)\\
&+e^{-3\pi y}\Big(
-348240\pi^4y^4-2592+178854\pi^3y^3+209040\pi^5y^5+19656\pi y+91536\pi^2y^2
\Big)\\
&+e^{-4\pi y}\Big(
-804576\pi^5y^5-121856\pi^3y^3-472576\pi^2y^2-182784\pi y+1465533\pi^4y^4+18432
\Big)\\
&+e^{-5\pi y}\Big(
-140064\pi64y^4-6912+160272\pi^3y^3+685440\pi^5y^5+84672\pi y+284544\pi^2y^2
\Big)\\
&+e^{-6\pi y}\Big(
-570500\pi^3y^3-3628800\pi^5y^5-58800\pi y-301280\pi^2y^2+3100608\pi^4y^4+4032
\Big)\\
&+e^{-7\pi y}\Big(
-5236608\pi^4y^4-3456+862344\pi^3y^3+7527936\pi^5y^5+361152\pi^2y^2+58464\pi y
\Big).
\endaligned\end{equation}
The rest is careful calculations by taking derivatives.

\end{proof}

\begin{lemma}\label{V2} There has
$
|\Big(\mathcal{Y}_e''''\mathcal{X}''-\mathcal{Y}_e'' \mathcal{X}''''+\mathcal{Y}_a''''\mathcal{X}_e''-\mathcal{X}_e''''\mathcal{Y}_a''\Big)(y)|
\leq 16(\frac{17}{4}\pi)^4 \sqrt y e^{-\frac{17}{4}\pi y},\;\;y\geq1.
$
\begin{remark} The coefficient of the bound is rather rough but is enough to get our result.
The exponential power captures the main feature.

\end{remark}
\end{lemma}

\begin{proof}
By Lemma \ref{Lappendix}, one infers that
\begin{equation}\aligned\label{V2a}
|\mathcal{Y}_e''(y)|\leq 4(\frac{17}{4}\pi)^2(1+\sigma_{\mathcal{Y}_e,2})\sqrt y e^{-\frac{17}{4}\pi y}, |\mathcal{Y}_e''''(y)|\leq 4(\frac{17}{4}\pi)^4(1+\sigma_{\mathcal{Y}_e,4})\sqrt y e^{-\frac{17}{4}\pi y}
 \endaligned\end{equation}
and
\begin{equation}\aligned\label{V2b}
|\mathcal{X}_e''(y)|\leq 8(5\pi)^2(1+\sigma_{\mathcal{X}_e,2})\sqrt y e^{-5\pi y},
|\mathcal{X}_e''''(y)|\leq 8(5\pi)^4(1+\sigma_{\mathcal{X}_e,4})\sqrt y e^{-5\pi y}.
 \endaligned\end{equation}
 Here $\sigma_{\mathcal{X}_e,j}, \sigma_{\mathcal{Y}_e,j}, j=2,4$ are small and can be bounded by $\frac{1}{4}$.
For $\mathcal{X}'', \mathcal{X}'''',\mathcal{Y}_a''$ and $\mathcal{Y}_a''''$, by their explicit expressions, one has
\begin{equation}\aligned\label{V2c}
|\mathcal{X}''''(y)|\leq 10,\;\; |\mathcal{X}''(y)|\leq 1.2,
|\mathcal{Y}_a''''(y)|\leq \frac{1}{10},\;\;|\mathcal{Y}_a''''(y)|\leq 1, \;\;y\geq1.
 \endaligned\end{equation}
Combining \eqref{V2a}, \eqref{V2b} with \eqref{V2c}, one gets the estimate.

\end{proof}

\begin{lemma}\label{V3} There holds
$
\Big(\mathcal{Y}''''\mathcal{X}''-\mathcal{Y}''\mathcal{X}''''\Big)(y)>0,\;\;y\in[1,1.11].
$

\end{lemma}
\begin{proof} It suffices to prove that

$
\frac{512y^4}{\pi}e^{\frac{1}{4}\pi y}\Big(\mathcal{Y}''''\mathcal{X}''-\mathcal{Y}''\mathcal{X}''''\Big)(y)>0,\;\;y\in[1,1.11].
$
By the decomposition and Lemmas \ref{V1} and \ref{V2}, we obtain that
\begin{equation}\aligned
&\frac{512y^4}{\pi}e^{\frac{1}{4}\pi y}\Big(\mathcal{Y}''''\mathcal{X}''-\mathcal{Y}''\mathcal{X}''''\Big)(y)\\
=&\frac{512y^4}{\pi}e^{\frac{1}{4}\pi y}\Big(\mathcal{Y}_a''''\mathcal{X}_a''-\mathcal{Y}_a''\mathcal{X}_a''''\Big)(y)
+\frac{512y^4}{\pi}e^{\frac{1}{4}\pi y}\Big(\mathcal{Y}_e''''\mathcal{X}''-\mathcal{Y}_e'' \mathcal{X}''''+\mathcal{Y}_a''''\mathcal{X}_e''-\mathcal{X}_e''''\mathcal{Y}_a''\Big)(y)\\
\geq&\frac{512y^4}{\pi}e^{\frac{1}{4}\pi y}\Big(\mathcal{Y}_a''''\mathcal{X}_a''-\mathcal{Y}_a''\mathcal{X}_a''''\Big)(y)
-\frac{72}{5}\cdot 17^4\pi^3y^{9/2}e^{-4\pi y}\\
\geq&\frac{512y^4}{\pi}e^{\frac{1}{4}\pi y}\Big(\mathcal{Y}_a''''\mathcal{X}_a''-\mathcal{Y}_a''\mathcal{X}_a''''\Big)(y)\mid_{y=1.11}
-\frac{72}{5}\cdot 17^4\pi^3y^{9/2}e^{-4\pi y}\mid_{y=1}, y\in[1,1.11]\\
=&158.4646175\cdots-130.0476135\cdots\\
>&0.
\endaligned\end{equation}

\end{proof}

\subsection{The rest of proof in Theorem \ref{QAB} }

\begin{lemma}\label{UU1}   The function $
y\rightarrow\frac{4y}{\pi}e^{\frac{1}{2}\pi y}\Big(\mathcal{B}_a''\mathcal{A}_a'-\mathcal{A}_a''\mathcal{B}_a'\Big)(y), y>1
$ is monotone increasing.

\end{lemma}

\begin{proof} By direct calculations, one regroups the terms by
\begin{equation}\aligned\label{Qab}
&\frac{4y}{\pi}e^{\frac{1}{2}\pi y}\Big(
\mathcal{B}_a''\mathcal{A}_a'-\mathcal{A}_a''\mathcal{B}_a'
\Big)(y)\\
=&\Big(\pi y-3-288e^{-8\pi y}\pi^2 y^2
-12e^{-8\pi y}-144e^{-\frac{9}{2}\pi y}
-72e^{-3\pi y}
-48e^{-\frac{5}{2}\pi y}-84e^{-5\pi y}
-12\pi e^{-\pi y} y\\
&-8\pi e^{-\frac{1}{2}\pi y}y-768\pi^2e^{-\frac{9}{2}\pi y}y^2
-128\pi^2e^{-\frac{5}{2}\pi y}y^2
-240y^2e^{-3\pi y}\pi^2
-504e^{-5\pi y}\pi^2y^2-52e^{-6\pi y}\pi y\\
&-99e^{-4 \pi y}\pi y-10e^{-2\pi y}\pi y\Big)\\
+&\Big(
68e^{-8\pi y}\pi y+240e^{-6\pi y}\pi^2 y^2
+12e^{-\frac{1}{2}\pi y}
+12e^{-6\pi y}+33e^{-4\pi y}+6e^{-2\pi y}
+12e^{-\pi y}
+96\pi e^{-\frac{5}{2}\pi y} y\\
&+480\pi e^{-\frac{9}{2}\pi y}y
+8\pi^2e^{-\pi y} y^2
+168ye^{-3\pi y}\pi
+64e^{-4\pi y}\pi^2y^2
+48e^{-2\pi y}\pi^2 y^2
+308e^{-5\pi y} \pi y\Big)
\endaligned\end{equation}
Denote the terms in the first and second bracket of \eqref{Qab} by $\mathcal{P}_{\mathcal{AB}}^+$ and $\mathcal{P}_{\mathcal{AB}}^-$. Then
\begin{equation}\aligned
\frac{4y}{\pi}e^{\frac{1}{2}\pi y}\Big(
\mathcal{B}_a''\mathcal{A}_a'-\mathcal{A}_a''\mathcal{B}_a'
\Big)(y)
=\mathcal{P}_{\mathcal{AB}}^+(y)+\mathcal{P}_{\mathcal{AB}}^-(y).
\endaligned\end{equation}
It remains to prove that
$
\mathcal{P}_{\mathcal{AB}}^+(y)+\mathcal{P}_{\mathcal{AB}}^-(y)>0, \;\;y>1.
$

By Lemma \ref{Lappendix},
\begin{equation}\aligned
\Big(\mathcal{P}_{\mathcal{AB}}^+(y)\Big)'(y)>\pi, \;\;\Big(\mathcal{P}_{\mathcal{AB}}^+(y)\Big)''(y)<0,\;
\Big(\mathcal{P}_{\mathcal{AB}}^-(y)\Big)'(y)<0, \;\;\Big(\mathcal{P}_{\mathcal{AB}}^-(y)\Big)''(y)>0
\endaligned\end{equation}

Since $\Big(\mathcal{P}_{\mathcal{AB}}^-(y)\Big)'(y)\mid_{y=1.82}=-3.051954266\cdots$, one has
\begin{equation}\aligned
\mathcal{P}_{\mathcal{AB}}^+(y)+\mathcal{P}_{\mathcal{AB}}^-(y)
\geq&\pi-3.051954266\cdots, y\in[1.82,\infty)>0.
\endaligned\end{equation}
It remains to prove that $\mathcal{P}_{\mathcal{AB}}^+(y)+\mathcal{P}_{\mathcal{AB}}^-(y)
>0$ on the bounded interval $(1,1.82]$. To this end, we divide the interval $(1,1.82]$ into 10 smaller subintervals, and compute the derivatives on each interval to arrive the result.

\end{proof}

\begin{lemma}\label{UU2} There holds:
$
|\Big(\mathcal{B}_e''\mathcal{A}'-\mathcal{B}_e' \mathcal{A}''+\mathcal{B}_a''\mathcal{A}_e'-\mathcal{A}_e''\mathcal{B}_a'\Big)(y)|
\leq 8(\frac{13}{8}\pi)^2\sqrt y e^{-\frac{13}{2}\pi y},\;\;y\geq1.
$
\end{lemma}

By Lemma \ref{Lappendix}, one has for $j=1,2,\cdots$
\begin{equation}\aligned\label{pab1}
|\mathcal{A}_e^{(j)}(y)|\leq  4(1+\sigma_{\mathcal{A}_e,j})(\frac{13}{2}\pi)^j\sqrt ye^{-\frac{13}{2}\pi y},\;\;
|\mathcal{B}_e^{(j)}(y)|\leq  4(1+\sigma_{\mathcal{B}_e,j})(\frac{13}{2}\pi)^j\sqrt ye^{-\frac{13}{2}\pi y}.
 \endaligned\end{equation}
Here the $\sigma_{\mathcal{A}_e,j}, \sigma_{\mathcal{B}_e,j}$ are small and can be bounded by $\frac{1}{2}$.
For $\mathcal{A}', \mathcal{A}', \mathcal{B}_a', \mathcal{B}_a''$, by their explicit expressions, one deduces that
\begin{equation}\aligned\label{pab2}
|\mathcal{A}'(y)|\leq 0.3, \;\;|\mathcal{A}''(y)|\leq\frac{1}{2},\;\;
|\mathcal{B}_a'(y)|\leq\frac{1}{5},\;\;|\mathcal{B}_a''(y)|\leq \frac{1}{5}.
 \endaligned\end{equation}
Combining \eqref{pab1} and \eqref{pab2}, one gets the estimate.

\begin{lemma} \label{UU3} There holds
$
\Big(\mathcal{B}''\mathcal{A}'-\mathcal{A}''\mathcal{B}'\Big)(y)>0\;\;\hbox{if}\;\;y\in[1.05,\infty).
$
\end{lemma}

\begin{proof} Equivalently, it suffice to prove that
$
\frac{4y}{\pi}e^{\frac{1}{2}\pi y}\Big(\mathcal{B}''\mathcal{A}'-\mathcal{A}''\mathcal{B}'\Big)(y)>0\;\;\hbox{if}\;\;y\in[1.05,\infty).
$
By Lemmas \ref{UU1} and \ref{UU2}, we  deduce that
\begin{equation}\aligned
&\frac{4y}{\pi}e^{\frac{1}{2}\pi y}\Big(\mathcal{B}''\mathcal{A}'-\mathcal{A}''\mathcal{B}'\Big)(y)\\
=&\frac{4y}{\pi}e^{\frac{1}{2}\pi y}\Big(\mathcal{B}_a''\mathcal{A}_a'-\mathcal{A}_a''\mathcal{B}_a'\Big)(y)
+\frac{4y}{\pi}e^{\frac{1}{2}\pi y}\Big(\mathcal{B}_e''\mathcal{A}'-\mathcal{B}_e' \mathcal{A}''+\mathcal{B}_a''\mathcal{A}_e'-\mathcal{A}_e''\mathcal{B}_a'\Big)(y)\\
\geq&\frac{4y}{\pi}e^{\frac{1}{2}\pi y}\Big(\mathcal{B}_a''\mathcal{A}_a'-\mathcal{A}_a''\mathcal{B}_a'\Big)(y)
-1352\pi y^{3/2}e^{-6\pi y}\\
\geq&\Big(
\frac{4y}{\pi}e^{\frac{1}{2}\pi y}\Big(\mathcal{B}_a''\mathcal{A}_a'-\mathcal{A}_a''\mathcal{B}_a'\Big)(y)
-1352\pi y^{3/2}e^{-6\pi y}
\Big)\mid_{y=1.05}=0.001189906301\cdots\\
>&0.
 \endaligned\end{equation}
Here we use the fact  that $y\mapsto-y^{3/2}e^{-6\pi y}, y>1$ is strictly increasing in the second last inequality.

\end{proof}

\begin{lemma}\label{VV1}
$
y\rightarrow \frac{32y^4}{\pi}e^{\frac{1}{2}\pi y}\Big(\mathcal{B}_a''''\mathcal{A}_a''-\mathcal{B}_a''\mathcal{A}_a''''\Big)(y)
$
is strictly decreasing on $(1,1.12)$.

\end{lemma}

\begin{proof}
By Direct calculations, one regroups the terms by
\begin{equation}\aligned\label{1335}
&\frac{32y^4}{\pi}e^{\frac{1}{2}\pi y}
\Big(
\mathcal{B}_a''''\mathcal{A}_a''-\mathcal{B}_a''''\mathcal{A}_a''''
\Big)(y)\\
=&-\pi^3y^3+4\pi^2y^2+21\pi y-18\\
&+e^{-\frac{1}{2}\pi y}(32\pi^3y^3+72-64\pi^2y^2-168\pi y)\\
&+e^{-\pi y}(176\pi^4y^4+72-48\pi^5y^5-252\pi y-304\pi^2y^2-132\pi^3y^3)\\
&+e^{-2\pi y}(2784\pi^4y^4+36-960\pi^5y^5-2150\pi^3y^3-1160\pi^2y^2-210\pi y)\\
&+e^{-\frac{5}{2}\pi y}(
6144\pi^5y^5+4224\pi^3y^3+2016\pi y+4864\pi^2y^2-11264\pi^4y^4-288
)\\
&+e^{-3\pi y}(8568\pi^3y^3+16800\pi^5y^5+9504\pi^2y^2+3528\pi y-28320\pi^4y^4-432
)\\
&+e^{-4\pi y}(
2007\pi^3y^3+28800\pi^5y^5+8708\pi^2y^2+3213\pi y-32320\pi^4y^4-306
)\\
&+e^{-5\pi y}(
99792\pi^5y^5+18172\pi^3y^3+23632\pi^2y^2+6468\pi y-140112\pi^4y^4-504
)\\
&+e^{-6\pi y}(
49660\pi^3y^3+336960\pi^5y^5+27920\pi^2y^2+5460\pi y-295200\pi^4y^4-360
).
\endaligned\end{equation}
Using the explicit  expression in \eqref{1335} and  dividing the interval $(1,1.12)$ into 10 smaller intervals and calculating the derivatives on each interval, we obtain the result.

\end{proof}

\begin{lemma}\label{VV2} The error estimate holds:
\begin{equation}\aligned
|\Big(\mathcal{B}_e''''\mathcal{A}''-\mathcal{B}_e'' \mathcal{A}''''+\mathcal{B}_a''''\mathcal{A}_e''-\mathcal{A}_e''''\mathcal{B}_a''\Big)(y)|
\leq 8(\frac{13}{2}\pi)^4 \sqrt y e^{-\frac{13}{2}\pi y}.
\endaligned\end{equation}
\begin{remark} The coefficient of the bound is rather rough but is enough to get our result.
The exponential power captures the main feature.

\end{remark}
\end{lemma}

\begin{proof} Using the explicit expressions of $\mathcal{A}$ and $\mathcal{B}_a$, after tedious estimates,  we arrive at
\begin{equation}\aligned
|\mathcal{A}''''(y)|\leq 8,\;\; |\mathcal{B}_a''''(y)|\leq 5.
\endaligned\end{equation}
This, combining with \eqref{pab1} and \eqref{pab2}, gives the estimate.

\end{proof}

\begin{lemma}\label{VV3} There holds
\begin{equation}\aligned
\Big(\mathcal{B}''''\mathcal{A}''-\mathcal{B}''\mathcal{A}''''\Big)(y)>0, y\in[1,1.12].
\endaligned\end{equation}

\end{lemma}

\begin{proof} It is equivalent to proving that
$
\frac{32y^4}{\pi}e^{\frac{1}{2}\pi y}\Big(\mathcal{B}''''\mathcal{A}''-\mathcal{B}''\mathcal{A}''''\Big)(y)>0, y\in[1,1.12].
$
By Lemmas \ref{VV1} and \ref{VV2}, we have that
\begin{equation}\aligned
&\frac{32y^4}{\pi}e^{\frac{1}{2}\pi y}\Big(\mathcal{B}''''\mathcal{A}''-\mathcal{B}''\mathcal{A}''''\Big)(y)\\
=&\frac{32y^4}{\pi}e^{\frac{1}{2}\pi y}\Big(\mathcal{B}_a''''\mathcal{A}_a''-\mathcal{B}_a''\mathcal{A}_a''''\Big)(y)
+\frac{32y^4}{\pi}e^{\frac{1}{2}\pi y}\Big(\mathcal{A}_2''\mathcal{B}''''+\mathcal{B}_2''''\mathcal{A}_a''-\mathcal{A}_2''''\mathcal{B}''-\mathcal{B}_2''\mathcal{A}_a''''\Big)(y)\\
\geq&\frac{32y^4}{\pi}e^{\frac{1}{2}\pi y}\Big(\mathcal{B}_a''''\mathcal{A}_a''-\mathcal{B}_a''\mathcal{A}_a''''\Big)(y)
-26^4\pi^3y^{9/2}e^{-6\pi y}\\
\geq&\frac{32y^4}{\pi}e^{\frac{1}{2}\pi y}\Big(\mathcal{B}_a''''\mathcal{A}_a''-\mathcal{B}_a''\mathcal{A}_a''''\Big)(y)\mid_{y=1.12}
-26^4\pi^3y^{9/2}e^{-6\pi y}\mid_{y=1}\\
=&49.93918473\cdots-0.09227517899\cdots\\
>&0.
\endaligned\end{equation}

\end{proof}

\bigskip

\noindent
{\bf Acknowledgements.} We thank Professor L. B$\acute{e}$termin for pointing out Mueller-Ho conjecture to us and Professors A. Aftalion and  X. Ren for useful discussions. The research of J. Wei is partially supported by NSERC of Canada.

\bigskip


\end{document}